\documentclass[12pt]{amsart}
\usepackage{cases}
\usepackage{mathrsfs}
\usepackage{amsfonts}
\usepackage{amscd}

\usepackage{latexsym,amssymb}
\usepackage{a4wide}
\usepackage{epic,eepic,latexsym, amssymb, amscd, amsfonts, xypic, floatflt}
\usepackage{amsthm}
\usepackage{amsmath}
\usepackage{amscd}
\usepackage{graphicx}
\usepackage[mathscr]{eucal}
\input xy
\xyoption{all}
\usepackage{verbatim}
\usepackage{hyperref}

\numberwithin{equation}{section}
\begin{document}
\theoremstyle{plain}
\newtheorem{thm}{Theorem}[section]
\newtheorem{theorem}[thm]{Theorem}
\newtheorem{lemma}[thm]{Lemma}
\newtheorem{corollary}[thm]{Corollary}
\newtheorem{proposition}[thm]{Proposition}
\newtheorem{conjecture}[thm]{Conjecture}
\theoremstyle{definition}
\newtheorem{remark}[thm]{Remark}
\newtheorem{remarks}[thm]{Remarks}
\newtheorem{definition}[thm]{Definition}
\newtheorem{example}[thm]{Example}

\newcommand{\caA}{{\mathcal A}}
\newcommand{\caB}{{\mathcal B}}
\newcommand{\caC}{{\mathcal C}}
\newcommand{\caD}{{\mathcal D}}
\newcommand{\caE}{{\mathcal E}}
\newcommand{\caF}{{\mathcal F}}
\newcommand{\caG}{{\mathcal G}}
\newcommand{\caH}{{\mathcal H}}
\newcommand{\caI}{{\mathcal I}}
\newcommand{\caJ}{{\mathcal J}}
\newcommand{\caK}{{\mathcal K}}
\newcommand{\caL}{{\mathcal L}}
\newcommand{\caM}{{\mathcal M}}
\newcommand{\caN}{{\mathcal N}}
\newcommand{\caO}{{\mathcal O}}
\newcommand{\caP}{{\mathcal P}}
\newcommand{\caQ}{{\mathcal Q}}
\newcommand{\caR}{{\mathcal R}}
\newcommand{\caS}{{\mathcal S}}
\newcommand{\caT}{{\mathcal T}}
\newcommand{\caU}{{\mathcal U}}
\newcommand{\caV}{{\mathcal V}}
\newcommand{\caW}{{\mathcal W}}
\newcommand{\caX}{{\mathcal X}}
\newcommand{\caY}{{\mathcal Y}}
\newcommand{\caZ}{{\mathcal Z}}
\newcommand{\fA}{{\mathfrak A}}
\newcommand{\fB}{{\mathfrak B}}
\newcommand{\fC}{{\mathfrak C}}
\newcommand{\fD}{{\mathfrak D}}
\newcommand{\fE}{{\mathfrak E}}
\newcommand{\fF}{{\mathfrak F}}
\newcommand{\fG}{{\mathfrak G}}
\newcommand{\fH}{{\mathfrak H}}
\newcommand{\fI}{{\mathfrak I}}
\newcommand{\fJ}{{\mathfrak J}}
\newcommand{\fK}{{\mathfrak K}}
\newcommand{\fL}{{\mathfrak L}}
\newcommand{\fM}{{\mathfrak M}}
\newcommand{\fN}{{\mathfrak N}}
\newcommand{\fO}{{\mathfrak O}}
\newcommand{\fP}{{\mathfrak P}}
\newcommand{\fQ}{{\mathfrak Q}}
\newcommand{\fR}{{\mathfrak R}}
\newcommand{\fS}{{\mathfrak S}}
\newcommand{\fT}{{\mathfrak T}}
\newcommand{\fU}{{\mathfrak U}}
\newcommand{\fV}{{\mathfrak V}}
\newcommand{\fW}{{\mathfrak W}}
\newcommand{\fX}{{\mathfrak X}}
\newcommand{\fY}{{\mathfrak Y}}
\newcommand{\fZ}{{\mathfrak Z}}

\newcommand{\bA}{{\mathbb A}}
\newcommand{\bB}{{\mathbb B}}
\newcommand{\bC}{{\mathbb C}}
\newcommand{\bD}{{\mathbb D}}
\newcommand{\bE}{{\mathbb E}}
\newcommand{\bF}{{\mathbb F}}
\newcommand{\bG}{{\mathbb G}}
\newcommand{\bH}{{\mathbb H}}
\newcommand{\bI}{{\mathbb I}}
\newcommand{\bJ}{{\mathbb J}}
\newcommand{\bK}{{\mathbb K}}
\newcommand{\bL}{{\mathbb L}}
\newcommand{\bM}{{\mathbb M}}
\newcommand{\bN}{{\mathbb N}}
\newcommand{\bO}{{\mathbb O}}
\newcommand{\bP}{{\mathbb P}}
\newcommand{\bQ}{{\mathbb Q}}
\newcommand{\bR}{{\mathbb R}}
\newcommand{\bT}{{\mathbb T}}
\newcommand{\bU}{{\mathbb U}}
\newcommand{\bV}{{\mathbb V}}
\newcommand{\bW}{{\mathbb W}}
\newcommand{\bX}{{\mathbb X}}
\newcommand{\bY}{{\mathbb Y}}
\newcommand{\bZ}{{\mathbb Z}}
\newcommand{\id}{{\rm id}}

\title[New structures for colored HOMFLY-PT invariants]{New structures for colored HOMFLY-PT invariants}
\author[Shengmao Zhu]{Shengmao Zhu}
\address{Department of Mathematics \\ Zhejiang International Studies
University}
\address{Center of Mathematical Sciences \\ Zhejiang University
\\Hangzhou, 310027, China }
\email{shengmaozhu@126.com, szhu@zju.edu.cn}

\maketitle

\begin{abstract}
In this paper, we present several new structures for the colored
HOMFLY-PT invariants of framed links. First, we prove the strong
integrality property for the normalized colored HOMFLY-PT invariants
by purely using the HOMFLY-PT skein theory developed by H. Morton
and his collaborators. By this strong integrality property, we
immediately obtain several symmetric properties for the full colored
HOMFLY-PT invariants of links. Then, we apply our results to refine
the mathematical structures appearing in the
Labastida-Mari\~no-Ooguri-Vafa (LMOV) integrality conjecture for
framed links. As another application of the strong integrality, we
obtain that the $q=1$ and $a=1$ specializations of the normalized
colored HOMFLY-PT invariant are well-defined link polynomials. We
find that a conjectural formula for the colored Alexander polynomial
which is the $a=1$ specialization of the normalized colored
HOMFLY-PT invariant implies that a special case of the LMOV
conjecture for frame knot holds.
\end{abstract}

\tableofcontents

\section{Introduction}
\subsection{Colored HOMFLY-PT invariants}
\label{subsection-coloredhom} Colored HOMFLY-PT invariant is an
invariant of framed oriented links in $S^3$ whose components are
colored with partitions. It can be viewed as a generalization of the
HOMFLY-PT polynomial which was first discovered by Freyd-Yetter,
Hoste, Lickorish-Millet, Ocneanu \cite{FYHLMO85} and
Przytycki-Traczyk \cite{PT87}.  Based on Turaev's work
\cite{Tu88-2}, the HOMFLY-PT polynomial can be derived from the
quantum group invariant labeled with the fundamental representation
of the quantum group $U_q(sl_N)$. More generally, the colored
HOMFLY-PT invariants can be  constructed by using the quantum group
invariants labeled with arbitrary irreducible representations of
$U_q(sl_N)$ which are labeled by partitions. We refer to
\cite{LaMa02,LiZh10} for detailed definition of the colored
HOMFLY-PT invariants along this way.

Historically, the quantum group approach to the colored HOMFLY-PT
invariant originates from the $SU(N)$ Chern-Simons theory due to E.
Witten \cite{Witten89}. The large $N$ duality between the $SU(N)$
Chern-Simons theory and topological string theory established in
\cite{Witten95,GV99,OV00} makes the colored HOMFLY-PT invariant
become a central object in geometry and physics. There have been
spectacular developments on colored HOMFLY-PT invariant and its
related topics over the years. For examples,
 the dependence of the colors is governed by recurrence relations which was proposed in \cite{AV12}, and proved in \cite{GLL18}.
Moreover, these recurrence relations are conjectured to have deep
relationships to the augmentation polynomials in knot contact
homology \cite{AV12,AENV13}. The relationship to the topological
string theory connects the colored HOMFLY-PT invariants to
Gromov-Witten theory \cite{MV02,LLZ03,LLZ07,DSV13}, Hilbert schemes
\cite{OS12}, stable pair invariants \cite{DHS12,Ma16}, quiver
varieties \cite{LuoZhu16,KRSS17-1,KRSS17-2,PSS18,SW20-1,SW20-2},
knot contact homology \cite{AV12,AENV13,EN18,ES21} etc. There also
have being a lot of works devoted to the computations of the colored
HOMFLY-PT invariants, see
\cite{LiZh10,IMMM12,NPZ12,NRZ13,GJ15,Ka16,Wed19} and references
therein.

On the other hand side, it is well-known that the HOMFLY-PT
polynomial satisfies the fantastic skein relations which provide a
combinatorial way to study them. From this point of view, Morton and
his collaborators
\cite{Mo93,Ai96,Ai97-1,Ai97-2,MA98,Mo02,HaMo06,Mo07,MoMa08,MoSa17}
developed a systematic method, referred as HOMFLY-PT skein theory in
this paper, to investigate the colored HOMFLY-PT invariants. The
equivalence of the two approaches to colored HOMFLY-PT invariants
via quantum group theory and HOMFLY-PT skein theory are presented in
\cite{Ai96,Lu01}, we also refer the reader to the appendix Section
\ref{section-Appendix} for an explanation of this equivalence.

Comparing to the quantum group theory, HOMFLY-PT skein theory is
more intuitional and straightforward. Therefore, in the paper, we
mainly use the HOMFLY-PT skein theory as our tool to derive several
new structures for colored HOMFLY-PT invariants.

\subsubsection{Strong integrality}
Let $\mathcal{L}$ be a framed oriented link with $L$ components
$\mathcal{K}_\alpha$ for $\alpha=1,...,L$ \footnote{In the
following, when we mention a link $\mathcal{L}$, we always assume it
is a framed oriented link with $L$ components given by
$\mathcal{K}_\alpha$ for $\alpha=1,...,L$.}, given
$\vec{\lambda}=(\lambda^1,...,\lambda^L)$ and
$\vec{\mu}=(\mu^1,...,\mu^L)$, where each $\lambda^\alpha$ (resp.
$\mu^{\alpha}$) for $\alpha=1,..,L$ denotes a partition of a
positive integer, the framed full colored HOMFLY-PT invariant of
$\mathcal{L}$ colored by $[\vec{\lambda},\vec{\mu}]$ is given by
\begin{align}
\mathcal{H}_{[\vec{\lambda},\vec{\mu}]}(\mathcal{L};q,a)=
\mathcal{H}(\mathcal{L}\star
\otimes_{\alpha=1}^LQ_{\lambda^\alpha,\mu^{\alpha}}),
\end{align}
where $\mathcal{H}(\mathcal{L}\star
\otimes_{\alpha=1}^LQ_{\lambda^\alpha,\mu^{\alpha}})$ denotes the
framed HOMFLY-PT polynomial of the link $\mathcal{L}$ decorated by
 $\otimes_{\alpha=1}^LQ_{\lambda^\alpha,\mu^{\alpha}}$, where each $%
Q_{\lambda^\alpha,\mu^{\alpha}}$ denotes the skein basis element in
the full skein $\mathcal{C}$ of the annulus, we recall their
definitions in Section \ref{Section-homflyptskein}. In particular,
when all the $\mu^{\alpha}$ is equal to the empty partition
$\emptyset$, we write
$Q_{\lambda^{\alpha}}=Q_{\lambda^{\alpha},\mu^{\alpha}}$ for
$\alpha=1,...,L$, and denote
$\mathcal{H}_{\vec{\lambda}}(\mathcal{L};q,a)=
\mathcal{H}(\mathcal{L}\star
\otimes_{\alpha=1}^LQ_{\lambda^\alpha})$. Then we obtain the
standard (framing independent) colored HOMFLY-PT invariant for
$\mathcal{L}$ (e.g. the definitions in \cite{LiZh10,Zhu13})
\begin{align}
W_{\vec{\lambda}}(\mathcal{L};q,a)=q^{-\sum_{\alpha=1}^L\kappa_{\lambda^{\alpha}}
w(\mathcal{K}_{\alpha})}a^{-\sum_{\alpha=1}^L|\lambda^{\alpha}|w(\mathcal{K}_{\alpha})}\mathcal{H}_{\vec{\lambda}}(\mathcal{L};q,a).
\end{align}

When $\mathcal{L}$ has only one component, i.e. it is a knot which
is usually denoted by $\mathcal{K}$.  We define the {\em normalized
framed full colored HOMFLY-PT invariants} of $\mathcal{K}$ as follow
\begin{align}
\mathcal{P}_{[\lambda,\mu]}(\mathcal{K};q,a)=\frac{\mathcal{H}_{[\lambda,\mu]}(\mathcal{K};q,a)}{\mathcal{H}_{[\lambda,\mu]}(U;q,a)}
\end{align}
where $U$ denotes the unknot throughout this paper. Then we have the
following strong integrality property for
$\mathcal{P}_{[\lambda,\mu]}(\mathcal{K};q,a)$.
\begin{theorem} \label{Theorem-strongintegral1}
For any framed knot $\mathcal{K}$, the normalized framed full
colored HOMFLY-PT invariant satisfies
\begin{align} \label{formula-strongintegrality0}
\mathcal{P}_{[\lambda,\mu]}(\mathcal{K};q,a)\in
a^{\epsilon}\mathbb{Z}[q^{\pm 2},a^{\pm 2}].
\end{align}
where $\epsilon\in \{0,1\}$ is determined by the following formula
\begin{align}
(|\lambda|+|\mu|)w(\mathcal{K})&\equiv \epsilon \mod 2.
\end{align}
where $w(\mathcal{K})$ is the writhe number of $\mathcal{K}$.
\end{theorem}
In particular, the {\em normalized colored HOMFLY-PT invariant}
$P_{\lambda}(\mathcal{K};q,a)$ of $\mathcal{K}$, defined as
$P_{\lambda}(\mathcal{K};q,a)=\frac{W_{\lambda}(\mathcal{K};q,a)}{W_{\lambda}(U;q,a)}$,
satisfies
\begin{align} \label{formula-strongcolorhomfly}
P_{\lambda}(\mathcal{K};q,a)\in \mathbb{Z}[q^2,a^2].
\end{align}

More generally, for a link $\mathcal{L}$ with $L$ components, one
can define the {\em $\alpha$-normalized framed full colored
HOMFLY-PT invariants} for $\alpha=1,...,L$ as follow. We only
decorate the $\alpha$-component $\mathcal{K}_{\alpha}$ by
$Q_{\lambda^{\alpha},\mu^{\alpha}}$, and keep the other components
invariant, we denote the corresponding colored HOMFLY-PT invariant
of $\mathcal{L}$ by
$\mathcal{H}_{[\lambda^{\alpha},\mu^{\alpha}]}(\mathcal{L};q,a)$.
Then the $\alpha$-normalized colored framed full HOMFLY-PT invariant
of $\mathcal{L}$ is given by
\begin{align} \label{formula-tauframing0}
\mathcal{P}_{[\lambda^{\alpha},\mu^{\alpha}]}(\mathcal{L};q,a)=
\frac{\mathcal{H}_{[\lambda^{\alpha},\mu^{\alpha}]}(\mathcal{L};q,a)}{\mathcal{H}_{[\lambda^\alpha,\mu^\alpha]}(U;q,a)}.
\end{align}

We also have the following strong integrality theorem.
\ref{Section-strongintegrality}.
\begin{theorem} \label{Theorem-strongintegral2}
For any framed link $\mathcal{L}$ with $L$ components
$\mathcal{K}_{\alpha}$ for $\alpha=1,...,L$, the $\alpha$-normalized
framed full colored HOMFLY-PT invariants satisfies
\begin{align}  \label{formula-strongintegrality1}
\mathcal{P}_{[\lambda^{\alpha},\mu^{\alpha}]}(\mathcal{L};q,a)\in
a^{\epsilon}\mathbb{Z}[q^{\pm 2},a^{\pm 2}].
\end{align}
where $\epsilon\in \{0,1\}$ which is determined by the following
formula
\begin{align}
(|\lambda^\alpha|+|\mu^\alpha|)w(\mathcal{K}_{\alpha})+\sum_{\beta\neq
\alpha}w(\mathcal{K}_\beta) +L-1 &\equiv \epsilon \mod 2.
\end{align}
\end{theorem}

As a consequence of Theorem \ref{Theorem-strongintegral2}, we have
\begin{corollary} \label{corollary-symmetries}
For any framed link $\mathcal{L}$, the framed full colored HOMFLY-PT
invariants
$\mathcal{H}_{[\vec{\lambda},\vec{\mu}]}(\mathcal{L};q,a)$ satisfies
the following symmetric properties
\begin{align}
\mathcal{H}_{[\vec{\lambda},\vec{\mu}]}(\mathcal{L};-q,a)&=(-1)^{\sum_{\alpha=1}^L(|\lambda^{\alpha}|+|\mu^{\alpha}|)}
\mathcal{H}_{[\vec{\lambda},\vec{\mu}]}(\mathcal{L};q,a)\\
\mathcal{H}_{[\vec{\lambda},\vec{\mu}]}(\mathcal{L};q,-a)&=(-1)^{\sum_{\alpha=1}^L(|\lambda^{\alpha}|+|\mu^{\alpha}|)(w(\mathcal{K}_{\alpha})+1)}
\mathcal{H}_{[\vec{\lambda},\vec{\mu}]}(\mathcal{L};q,a)\\
\mathcal{H}_{[\vec{\lambda},\vec{\mu}]}(\mathcal{L};q^{-1},a)&=(-1)^{\sum_{\alpha=1}^L(|\lambda^{\alpha}|+|\mu^{\alpha}|)}\mathcal{H}_{[\vec{\lambda}^\vee,\vec{\mu}^\vee]}(\mathcal{L};q,a).
\end{align}
\end{corollary}
\begin{remark}
We prove Theorem \ref{Theorem-strongintegral1} and
\ref{Theorem-strongintegral2} by purely using HOMFLY-PT skein
theory. It essentially uses all the crucial results developed by
Morton and his collaborators. Comparing to our previous work
\cite{Zhu19-1}, the proof presented here is essential and
self-contained.
\end{remark}
\begin{remark}
From Corollary \ref{corollary-symmetries}, one immediately obtain
the following symmetries for colored HOMFLY-PT invariants
$W_{\vec{\lambda}}(\mathcal{L};q,a)$:

\begin{align}
W_{\vec{\lambda}}(\mathcal{L};-q,a)&=(-1)^{\sum_{\alpha=1}^L|\lambda^{\alpha}|}
W_{\vec{\lambda}}(\mathcal{L};q,a),\\
W_{\vec{\lambda}}(\mathcal{L};q,-a)&=(-1)^{\sum_{\alpha=1}^L|\lambda^{\alpha}|}
W_{\vec{\lambda}}(\mathcal{L};q,a),\\
W_{\vec{\lambda}}(\mathcal{L};q^{-1},a)&=(-1)^{\sum_{\alpha=1}^L|\lambda^{\alpha}|}
W_{\vec{\lambda}^{\vee}}(\mathcal{L};q,a),
\end{align}
which have appeared in previous literatures, see for example
\cite{LP11,Zhu13}.
\end{remark}

\subsubsection{Specializations}
From Theorem \ref{Theorem-strongintegral1}, we obtain that the $q=1$
specialization of the normalized framed full colored HOMFLY-PT
invariant $\mathcal{P}_{[\lambda,\mu]}(\mathcal{K};1,a)$ belongs to
the ring $a^{\epsilon}\mathbb{Z}[a^2]$. With a slight modification
of the proof presented in \cite{CZ20}, we obtain the formula
\begin{align}
\mathcal{P}_{[\lambda,\mu]}(\mathcal{K};1,a)=\mathcal{P}(\mathcal{K};1,a)^{|\lambda|+|\mu|},
\end{align}
where
$\mathcal{P}(\mathcal{K};q,a)=\frac{\mathcal{H}(\mathcal{K};q,a)}{\mathcal{H}(U;q,a)}$
is the normalized framed HOMFLY-PT polynomial.

 We also consider the composite invariants studied in
\cite{Mar10,CZ20} which is defined as follow. Given a link
$\mathcal{L}$ with $L$ components and $\vec{\lambda},
\vec{\mu},\vec{\nu} \in \mathcal{P}^{L}$. Let
$c_{\vec{\lambda},\vec{\mu}}^{\nu}=\prod_{\alpha=1}^{L}c_{\lambda^{\alpha},\mu^{\alpha}}^{\nu^{\alpha}}$,
where each $c_{\lambda^{\alpha},\mu^{\alpha}}^{\nu^{\alpha}}$ is the
Littlewood-Richardson coefficient, then the framed composite
invariant of $\mathcal{L}$ is defined by
\begin{align}
\mathcal{C}_{\vec{\nu}}(\mathcal{L};q,a)=\sum_{\vec{\lambda},\vec{\mu}}c_{\vec{\lambda},\vec{\mu}}^{\vec{\nu}}
\mathcal{H}_{[\vec{\lambda},\vec{\mu}]}(\mathcal{L};q,a).
\end{align}

In particular, for the knot $\mathcal{K}$, we prove that the special
framed composite invariant which is the following limit exists and
satisfies the relation
\begin{align}
\lim_{q\rightarrow
1}\frac{\mathcal{C}_{\nu}(\mathcal{K};q,a)}{\mathcal{C}_{\nu}(U;q,a)}=\mathcal{P}(\mathcal{K};1,a)^{|\nu|}.
\end{align}

On the other hand side, one can also consider the $a=1$
specialization of the normalized framed full colored HOMFLY-PT
invariants, we define the full colored Alexander polynomial of
$\mathcal{K}$ as follow
\begin{align}
\mathcal{A}_{[\lambda,\mu]}(\mathcal{K};q)=\mathcal{P}_{[\lambda,\mu]}(\mathcal{K};q,1).
\end{align}
Then Theorem \ref{Theorem-strongintegral1} implies that
\begin{align}
\mathcal{A}_{[\lambda,\mu]}(\mathcal{K};q)\in \mathbb{Z}[q^2].
\end{align}
\begin{remark}
In \cite{IMMM12}, H. Itoyama et al. first studied the following
limit of the colored HOMFLY-PT invariant
\begin{align}
A_{\lambda}(\mathcal{K};q)=\lim_{a\rightarrow
1}\frac{W_{\lambda}(\mathcal{K};q,a)}{W_{\lambda}(U;q,a)},
\end{align}
and they proposed that $
A_{\lambda}(\mathcal{K};q)=A(\mathcal{K};q^{|\lambda|}), $ where
$A(\mathcal{K};q)$ is the classical Alexander polynomial of
$\mathcal{K}$. Later, in \cite{Zhu13}, the author found that this
formula doesn't hold for general partitions.   The colored Alexander
polynomial in the sense of \cite{IMMM12} can be written in our
notation as
\begin{align}
A_{\lambda}(\mathcal{K};q)=q^{-\kappa_{\lambda}w(\mathcal{K})}\mathcal{A}_{[\lambda,\emptyset]}(\mathcal{K};q).
\end{align}
that's also why we refer
$\mathcal{A}_{[\lambda,\mu]}(\mathcal{K};q)$ as full colored
Alexander polynomial, see also \cite{MST21}.
\end{remark}

In conclusion, we have the following standard conjecture for the
colored Alexander polynomial $A_{\lambda}(\mathcal{K};q)$.
\begin{conjecture} \label{conjecture-Alexander0}
For any hook partition $\lambda$, we have
\begin{align}
A_{\lambda}(\mathcal{K};q)=A(\mathcal{K};q^{|\lambda|}).
\end{align}
\end{conjecture}
This conjecture can be reduced to an identity for the  character of
Hecke algebra by using the work of \cite{LiZh10}, see formula
(\ref{formula-characteridentity}) in Section
\ref{Section-coloredAlender}.  We prove that the Conjecture
\ref{conjecture-Alexander0} holds for the torus knot $T_{r,s}$. We
remark that, recently, Mishnyakov et al. \cite{MST21} proposed
another new conjectural formula for the colored Alexander polynomial
$A_{\lambda}(\mathcal{K};q)$.

Interestingly, we find that the Conjecture
\ref{conjecture-Alexander0} is closely related to the LMOV
integrality conjecture for the framed knot which will be introduced
in the following.
\subsection{LMOV integrality structures}
Based on the large $N$ duality \cite{Witten95,GV99} and BPS
integrality structures in topological strings, Labastida, Mari\~no,
Ooguri and Vafa \cite{OV00,LaMaVa00,LaMa02} formulated an
interesting integrality conjecture for colored HOMFLY-PT invariants
which was referred as LMOV conjecture in the later literatures.
Roughly speaking, the LMOV conjecture states that the LMOV functions
which is a plethystic transformation of the generating functions for
colored HOMFLY-PT invariants inherits an amazing integrality
structure. During the studying of the LMOV conjecture, we find that
such integrality structures appear frequently in other settings.  So
we think it is worth to put these questions together and to see if
there is a unified theory to deal with them. In the following, we
first formulate the general framework for such integrality structure
which we refer as the LMOV integrality structure. Then we describe
the subtle LMOV integrality structure for framed links in detail and
present several new results as the application of the results
obtained in previous subsection \ref{subsection-coloredhom}.

Let $\Lambda(\mathbf{x})$ be the ring of symmetric functions of
$\mathbf{x}=(x_1,x_2,...)$ over the field $\mathbb{Q}(q,a)$. For
$\vec{\mathbf{x}}=(\mathbf{x}^1,...,\mathbf{x}^L)$,  we consider the
ring
$\Lambda(\vec{\mathbf{x}})=\Lambda(\mathbf{x})\otimes_{\mathbb{Z}}\cdots
\otimes_{\mathbb{Z}}\Lambda(\mathbf{x}^L)$ of functions separatedly
symmetric in $\mathbf{x}^1,...,\mathbf{x}^L$, where
$\mathbf{x}^{i}=(x_1^i,x_2^i,...)$.  Suppose
$B_{\vec{\lambda}}(\vec{\mathbf{x}})$, $\vec{\lambda}\in
\mathcal{P}^{L}$ forms a basis of $\Lambda(\vec{\mathbf{x}})$ (e.g.
taking $B_{\vec{\lambda}}(\vec{\mathbf{x}})$ as the Schur function,
Macdonald function etc.), consider the generating function of a
series of functions $S_{\vec{\lambda}}(q,a)\in \mathbb{Q}(q,a)$,
\begin{align}
Z(\vec{\mathbf{x}};q,a)=\sum_{\vec{\lambda}}S_{\vec{\lambda}}(q,a)B_{\vec{\lambda}}(\vec{\mathbf{x}}).
\end{align}

The LMOV function for $\{S_{\vec{\lambda}}(q,a)|\vec{\lambda}\in
\mathcal{P}^L\}$ is defined as the plethystic transformation of
$Z(\vec{\mathbf{x}};q,a)$, i.e.
\begin{align}
f_{\vec{\lambda}}(q,a)=\langle Log(Z(\vec{\mathbf{x}};q,a)),
B_{\vec{\lambda}}(\vec{\mathbf{x}})\rangle,
\end{align}
where $\langle \cdot, \cdot\rangle$ denotes the Hall pair in the
ring of symmetric functions $\Lambda(\vec{\mathbf{x}})$.

The LMOV integrality structure states that the LMOV function will be
an integral coefficient polynomial of $q$ and $a$, moreover, these
integral coefficients have different geometric meanings under
different settings. Here, we provide two examples.
\subsubsection*{Character varieties}
In \cite{HLRV11}, Hausel et al. proposed a general conjecture for
the mixed Hodge polynomial of the generic character varieties. In
our language,  they conjectured that the mixed Hodge polynomial is
the LMOV function for the HLV kernel functions introduced in
\cite{HLRV11}, and these integral coefficients appearing in LMOV
functions can be interpreted as the mixed Hodge numbers of the
generic character varieties. The integrality of the LMOV function in
this case was proved by A. Mellit \cite{Mel18}. Moreover, a possible
interpretation for such LMOV structure as a kind of topological
quantum field theory (TQFT) was described in \cite{Mel20} (cf.
section 1.6 of \cite{Mel20}). See also \cite{CDDNP20} and the
references therein for several physical generalizations of the above
framework.

\subsubsection*{Donaldson-Thomas invariants for quivers}
Given a quiver, there is an associated algebra named chomological
Hall algebra introduced by Kontsevich and Soilbelman \cite{KS11}. In
order to study the integrality of the Donaldson-Thomas invariants,
they introduced the notion of admissibility for a series of rational
functions. The admissibility condition is equivalent to say that the
 corresponding LMOV function carries integrality structure in
our language. Based on this observation, we started to interpret the
LMOV invariant for frame unknots in term of quiver representation
theory \cite{LuoZhu16,Zhu19-2}. There have been great developments
in this direction, see subsection \ref{subsubsection-new} for more
descriptions about it.

\subsubsection{LMOV structures for framed links}
Actually, the original LMOV structures for framed links are more
subtle. Given a framed link $\mathcal{L}$ with $L$ components
$\mathcal{K}_1,...,\mathcal{K}_L$. Denoted by
$\vec{\tau}=(\tau^1,...,\tau^L)\in \mathbb{Z}^L$ the framing of
$\mathcal{L}$, i.e. $\tau^\alpha=w(\mathcal{K}_{\alpha})$ for
$\alpha=1,...,L$. In the following, we use the notation
$\mathcal{L}_{\vec{\tau}}$ to denote this framed link $\mathcal{L}$
if we want to emphasize its framing.

We introduce the notions of  $\tau$-framed full colored HOMFLY-PT
invariants and  $\tau$-framed composite invariant for
$\mathcal{L}_{\vec{\tau}}$ as follows:
\begin{align}
H_{\vec{\lambda}}(\mathcal{L}_{\vec{\tau}};q,a)=(-1)^{\sum_{\alpha=1}^L|\lambda^\alpha|\tau^\alpha}
a^{-\sum_{\alpha=1}^L|\lambda^\alpha|\tau^\alpha}\mathcal{H}_{\vec{\lambda}}(\mathcal{L}_{\vec{\tau}};q,a).
\end{align}
\begin{align}
C_{\vec{\nu}}(\mathcal{L}_{\vec{\tau}};q,a)=(-1)^{\sum_{\alpha=1}^L|\lambda^\alpha|\tau^\alpha}\sum_{\vec{\lambda},\vec{\mu}}
a^{-\sum_{\alpha=1}^L|\nu^\alpha|\tau^\alpha}
c_{\vec{\lambda},\vec{\mu}}^{\vec{\nu}}
\mathcal{H}_{[\vec{\lambda},\vec{\mu}]}(\mathcal{L}_{\vec{\tau}};q,a).
\end{align}
We remark that the above definitions of $\tau$-framed invariants
were motivated by \cite{MV02}, here the term
$(-1)^{\sum_{\alpha=1}^L|\lambda^\alpha|\tau^\alpha}$ is essential
if we consider the framing change in Chern-Simon quantum field
theory. The LMOV functions for
$\{H_{\vec{\lambda}}(\mathcal{L}_{\vec{\tau}};q,a)\}$ and
$\{C_{\vec{\lambda}}(\mathcal{L}_{\vec{\tau}};q,a)\}$ are given by
\begin{align}
f_{\vec{\lambda}}^{(0)}(\mathcal{L}_{\vec{\tau}};q,a)=\langle
\text{Log}\sum_{\vec{\lambda}}H_{\vec{\lambda}}(\mathcal{L}_{\vec{\tau}};q,a)s_{\vec{\lambda}}(\vec{\mathbf{x}}),s_{\vec{\lambda}}(\vec{\mathbf{x}})
\rangle,
\end{align}
and
\begin{align}
f_{\vec{\lambda}}^{(1)}(\mathcal{L}_{\vec{\tau}};q,a)=\langle
\text{Log}\sum_{\vec{\lambda}}C_{\vec{\lambda}}(\mathcal{L}_{\vec{\tau}};q,a)s_{\vec{\lambda}}(\vec{\mathbf{x}}),s_{\vec{\lambda}}(\vec{\mathbf{x}})
\rangle.
\end{align}

The reformulated LMOV functions is defined as the character
transformation of the LMOV functions,
\begin{align}
g_{\vec{\mu}}^{(h)}(\mathcal{L}_{\vec{\tau}};q,a)
=\frac{1}{\{\vec{\mu}\}}\sum_{\vec{\lambda}}\chi_{\vec{\lambda}}(\vec{\mu})f_{\vec{\lambda}}^{(h)}(\mathcal{L}_{\vec{\tau}};q,a),
\   \ h=0,1.
\end{align}

From formula (\ref{formula-tauframing0}), we obtain that for
$\vec{\lambda}\neq \vec{0}$, both
$H_{\vec{\lambda}}(\mathcal{L}_{\vec{\tau}};q,a)$  and
$C_{\vec{\lambda}}(\mathcal{L}_{\vec{\tau}};q,a)$ contain a factor
$\frac{a-a^{-1}}{q-q^{-1}}$. As an application of Theorem
\ref{Theorem-strongintegral2} and Corollary
\ref{corollary-symmetries},  we obtain
\begin{theorem} \label{Theorem-reformulatedLMOV}
The reformulated LMOV functions
$g^{(h)}_{\vec{\mu}}(\mathcal{L}_{\vec{\tau}};q,a)$ for both $h=0$
and $1$ can be written in the following form
\begin{align}
g^{(h)}_{\vec{\mu}}(\mathcal{L}_{\vec{\tau}};q,a)=(a-a^{-1})\tilde{g}^{(h)}_{\vec{\mu}}(\mathcal{L}_{\vec{\tau}};q,a),
\end{align}
and $\tilde{g}^{(h)}_{\vec{\mu}}(\mathcal{L}_{\vec{\tau}};q,a)$ has
the properties
\begin{align}
\tilde{g}^{(h)}_{\vec{\mu}}(\mathcal{L}_{\vec{\tau}};q,-a)&=(-1)^{|\vec{\mu}|-1}\tilde{g}^{(h)}_{\vec{\mu}}(\mathcal{L}_{\vec{\tau}};q,a),
\\
\tilde{g}^{(h)}_{\vec{\mu}}(\mathcal{L}_{\vec{\tau}};-q,a)&=\tilde{g}^{(h)}_{\vec{\mu}}(\mathcal{L}_{\vec{\tau}};q,a),
\\
\tilde{g}^{(h)}_{\vec{\mu}}(\mathcal{L}_{\vec{\tau}};q^{-1},a)&=\tilde{g}^{(h)}_{\vec{\mu}}(\mathcal{L}_{\vec{\tau}};q,a).
\end{align}

\end{theorem}
\begin{remark}
We remark that the notation $z=q-q^{-1}$ will be used frequently
throughout the article.
\end{remark}

Furthermore, the LMOV conjecture for framed colored HOMFLY-PT
invariants \cite{MV02,CLPZ14} and framed composite invariants
\cite{Mar10,CZ20} states that, for $h=0$ and $1$,
\begin{align}
g_{\vec{\mu}}^{(h)}(\mathcal{L}_{\vec{\tau}};q,a)\in
z^{-2}\mathbb{Z}[z^2,a^{\pm 1}].
\end{align}

Combining with Theorem \ref{Theorem-reformulatedLMOV}, we obtain
\begin{conjecture}[Refined LMOV conjecture for framed links]
\label{conjecture-refinedLMOV} For  $h=0$ and $1$, the reformulated
LMOV functions can be written as
\begin{align}
g^{(h)}_{\vec{\mu}}(\mathcal{L}_{\vec{\tau}};q,a)=(a-a^{-1})\tilde{g}^{(h)}_{\vec{\mu}}(\mathcal{L}_{\vec{\tau}};q,a).
\end{align}
where
\begin{align}
\tilde{g}^{(h)}_{\vec{\mu}}(\mathcal{L}_{\vec{\tau}};q,a)\in
z^{-2}a^{\epsilon}\mathbb{Z}[z^2,a^{\pm 2}]
\end{align}
$\epsilon\in \{0,1\}$ is determined by $ |\vec{\mu}|-1\equiv
\epsilon \mod 2. $

In other words, there are integral invariants
\begin{align}
\tilde{N}_{\vec{\mu},g,Q}^{(h)}(\mathcal{L}_{\vec{\tau}})\in
\mathbb{Z},
\end{align}
 such that
\begin{align}
\tilde{g}^{(h)}_{\vec{\mu}}(\mathcal{L}_{\vec{\tau}};q,a)=\sum_{g\geq
0}\sum_{Q\in
\mathbb{Z}}\tilde{N}_{\vec{\mu},g,Q}^{(h)}(\mathcal{L}_{\vec{\tau}})z^{2g-2}a^{2Q+\epsilon}\in
z^{-2}a^{\epsilon}\mathbb{Z}[z^2,a^{\pm 2}].
\end{align}
\end{conjecture}

\begin{remark}
Similarly, we introduce the notation of special LMOV functions for
framed colored HOMFLY-PT invariants and composite invariants, there
are also integrality conjecture for them, see Conjecture
\ref{conjecture-specialLMOV} Section \ref{section-LMOV functions}.
\end{remark}
\begin{remark}
We name the Conjecture \ref{conjecture-refinedLMOV} as ``refined
LMOV conjecture". It should be noted  that the meaning of refinement
here is different from the notion ``refined" used in \cite{KN17},
where they actually proposed the LMOV conjecture for the refined
colored HOMFLY-PT invariants.
\end{remark}

As a special case of the refined LMOV Conjecture
\ref{conjecture-refinedLMOV}, we have
\begin{conjecture}
Given a framed knot $\mathcal{K}_{\tau}$, for any prime $p$, we have
\begin{align}
\tilde{g}_{p}^{(0)}(\mathcal{K}_{\tau};q,a)&=
\frac{1}{a-a^{-1}}\frac{1}{\{p\}}\left(\mathcal{H}(\mathcal{K}_{\tau}\star
P_{p};q,a)\right.\\\nonumber&\left.-(-1)^{(p-1)\tau}\mathcal{H}(\mathcal{K}_{\tau};q^p,a^p)\right)\in
z^{-2}a^{\epsilon}\mathbb{Z}[z^{2},a^{\pm 2}],
\end{align}
where $\epsilon=1$ when $p=2$ and $\epsilon=0$ when $p$ is an odd
prime.
\end{conjecture}
Furthermore, we prove in Section \ref{Section-coloredAlender} that
\begin{theorem}
If the Conjecture \ref{conjecture-Alexander0} holds, then we have
\begin{align}
\tilde{g}_p^{(0)}(\mathcal{K}_\tau;q,1)\in z^{-2}\mathbb{Z}[z^2].
\end{align}
In particular, $\tilde{g}_p^{(0)}(T_{r,s};q,1)\in
z^{-2}\mathbb{Z}[z^2]$ for the torus knot $T_{r,s}$.
\end{theorem}

\subsection{Related works and future directions}
\subsubsection{The degrees of the $a$ and $q$}
As shown in Theorem \ref{Theorem-strongintegral2}, for a link with
$L$-components, the normalized framed full colored HOMFLY-PT
invariant
$\mathcal{P}_{[\lambda^{\alpha},\mu^{\alpha}]}(\mathcal{L};q,a)$
lies in the ring $a^{\epsilon}\mathbb{Z}[q^2,a^2]$, so we can write
it in the following two forms
\begin{align}
\mathcal{P}_{[\lambda^{\alpha},\mu^{\alpha}]}(\mathcal{L};q,a)&=\sum_{k=m}^{M}c_{k}(a)q^{2k}\\
\mathcal{P}_{[\lambda^{\alpha},\mu^{\alpha}]}(\mathcal{L};q,a)&=\sum_{k=n}^{N}d_{k}(q)a^{2k+\epsilon}.
\end{align}
for some integers $m,M,n,N$ which depends on $\mathcal{L}$ and the
partitions $\lambda^{\alpha},\mu^{\alpha}$.

At this point, one natural question is how to estimate these
integers $m,M,n,N$. For the  HOMFLY-PT polynomial \cite{LiMi87},
there was a classical result \cite{Mo86} due to Morton
 which states that the above integers are bounded by algebraic
crossing numbers and the number of Seifert circles of the knot
diagram. In \cite{Ve15}, R. van der Veen generalized Morton's result
to the case of colored HOMFLY-PT polynomial colored with one-box
partition. So it is natural to consider if there are analogue
bounded formulas in our general case. Furthermore, the topological
interpretation of the degrees would be another interesting question.
We hope to address this question in future work.

\subsubsection{Categorifications}
The HOMFLY-PT polynomial can be categorified to the
Khovanov-Rozansky homology \cite{KR08}, after that, there have being
a lot works devoted to the categorified theory for the colored
HOMFLY-PT invariants, e.g. \cite{WW17,Wed19-2,GNSSS15}.
  Our formulas (\ref{formula-strongintegrality0}) and (\ref{formula-strongintegrality1})
suggests that the normalized framed full colored HOMFLY-PT
invariants may have categorifications whose homology is finite
dimensional. It is a challenging question to construct such
categorifications for them.

\subsubsection{New integral link invariants}
\label{subsubsection-new} The refined LMOV conjecture
\ref{conjecture-refinedLMOV} predicts the existence of new integral
link invariants
$\tilde{N}_{\vec{\mu},g,Q}^{(h)}(\mathcal{L}_{\vec{\tau}})$ for the
framed link $\mathcal{L}_{\vec{\tau}}$. A natural question is how to
define these new integral invariants directly. In their original
literatures \cite{OV00, LaMaVa00}, Vafa et al. proposed that these
new integral link invariants can be interpreted as the Euler
characteristic number of certain moduli space which are expected to
being exist. In \cite{LuoZhu16}, through a straightforward
computation, we found these integral invariants for framed unknot
$U_{\tau}$ were related to the Betti number of cohomological Hall
algebra of a corresponding quiver. Then the idea of knot-quiver
correspondence was further extended in
\cite{Zhu19-1,KRSS17-1,KRSS17-2}, and see
\cite{EKL20,PSS18,SW20-1,SW20-2} for more recent developments.

Obviously, these integral link invariants are fully determined by
the Chern-Simons partition function of the framed link. From the
point of view of topological string theory, Aganagic and Vafa
\cite{AV12} introduced the deformed A-polynomial for colored
HOMFLY-PT invariants which can be viewed as the mirror geometry of
the topological string theory. According to the large $N$ duality of
Chern-Simons and topological string theory
\cite{GV99,OV00,LaMaVa00}, these new integral link invariants are
fully determined by the deformed A-polynomial. Later, the existence
of the deformed A-polynomials was proved rigidly in \cite{GLL18}.

On the other hand side, there is a proposal initiated by Aganagic,
Ekholm, Ng and Vafa in \cite{AENV13} which connects the topological
strings and contact homology theory. In this framework, it is
conjectured that the deformed A-polynomial is equal to the
augmentation polynomial from knot contact homology \cite{Ng14}.
Therefore, there should be a geometric way to define these new
integral link invariants from knot contact homology, we leave the
detailed discussion about it to another paper \cite{Zhu21}.

\subsubsection{Geometry of the colored HOMFLY-PT invariants}
Recall that the large $N$ duality proposes that the colored
HOMFLY-PT invariant of a link in $S^3$ gives the count of all
holomorphic curves in the resolved conifold with boundary on the
shifted Lagrangian conormal of this link. However, a theory of open
Gromov-Witten invariants was still missing. In \cite{ES19}, Ekholm
and Skende developed a new method to define the invariant counts of
holomorphic curves with Lagrangian boundary $L$. The basic idea is
to use framed HOMFLY-PT skein module $Sk(L)$ (see Section
\ref{Section-homflyptskein}
 for the definition of skein module) to describe the
 obstruction to invariance. They introduced the notion of skein-valued open Gromov-Witten
 invariant for a link in $S^3$ which is the generating function of
 colored HOMFLY-PT invariants of this link in the skein $Sk(L)$
 \cite{ES20,ES21}.  It is expected that the new structures for colored HOMFLY-PT
 invariants obtained in this paper would provide new insights in this counting theory of holomorphic curves.

\subsubsection{Colored Kauffman invariants}
It is well-known that Kauffman  polynomials \cite{Kau90} is another
useful two-variable polynomial invariants of knots and links. Its
colored version was established in \cite{CC12} following the work of
\cite{LiZh10}.

The corresponding skein theory for the colored Kauffman invariants
was established in \cite{BB01,Ry08}. So we hope to establish the
analogue structural properties for colored Kauffman invariants via
Kauffman skein theory \cite{CZ}. From the point of view of
Chern-Simons theory, the colored Kauffman invariants correspond to
the gauge group $SO(N)/Sp(N)$. In \cite{CC12}, L. Chen and Q. Chen
proposed an analogue LMOV integrality conjecture for colored
Kauffman invariants. See also \cite{Mar10} for a mixed version of
the LMOV integrality conjecture. In \cite{CZ20}, we have shown that
the composite invariants and colored Kauffman invariants carries the
same congruence skein relations. It would be interesting to consider
the analogue refined LMOV integrality structure for the colored
Kauffman invariants.

\subsection{The structure of this paper}
The rest of this paper is organized as follow. In Section
\ref{Section-homflyptskein} we survey the main results of the
HOMFLY-PT skein theory and fix the notations. In Section
\ref{Section-Idempotent} we briefly review the construction of the
idempotents of Hecke algebras $H_{n}(q,a)$ following the work
\cite{MA98}, then we establish a key formula related to the
idempotent which will be applied in Section
\ref{Section-strongintegrality} to prove the strong integrality for
the framed full normalized colored HOMFLY-PT invariants. In Section
\ref{section-LMOV functions} we introduce the notion of LMOV
structure, then we
 refine the LMOV integrality structure for the framed
colored HOMFLY-PT invariants and composite invariants by using the
results obtained in Section \ref{Section-strongintegrality}. The
next Section \ref{section-special polynomials} and Section
\ref{Section-coloredAlender} are devoted to two specializations of
the normalized framed full colored HOMFLY-PT invariants. We prove an
identity for special polynomials in Section \ref{section-special
polynomials} and present a conjectural formula for colored Alexander
polynomial in Section \ref{Section-coloredAlender}. Finally, we
present that this conjectural formula implies a special case of LMOV
conjecture for frame knots. In the final appendix Section
\ref{section-Appendix}, we provide the basic notations related to
partitions and symmetric functions used in this article. Then we
present the equivalence of the two approaches to the colored
HOMFLY-PT invariants via quantum group theory and HOMFLY-PT skein
theory.

\section{HOMFLY-PT skein theory} \label{Section-homflyptskein}
Given a 3-manifold $M$ (possibly with boundary), roughly speaking, a
skein module is a quotient of the free module over isotopy classes
of links in $M$ by suitably chosen skein relations. Let us recall
briefly the history of development of skein modules.

 Motivated by the
pioneering work of V. Jones \cite{Jo85}, Kauffman \cite{Kau87}
introduced the Kauffman bracket which provided a simple approach to
the Jones polynomials. More precisely, the Kauffman bracket is a
framed knot invariant defined by a local relation named the Kauffman
bracket skein relation together with the value of the unknot. By
imposing this relation and the value of unknot, any link in
$\mathbb{R}^3$ can be reduced to a Laurent polynomial in $q$ times
the empty link. In order to extend the definition of Jones
polynomial of knots in $\mathbb{R}^3$ to a general oriented
$3$-manifold $M$,  Turaev \cite{Tu88-1,Tu89} and Przytrychi
\cite{Pr91} independently introduced the notion of skein module of
$M$. In particular, the Kauffman bracket skein modules of manifolds
were for the first time defined in \cite{Pr91}, which are formal
linear combinations of framed unoriented links considered up to the
Kauffman bracket skein relation together with the value of the
unknot. Furthermore, the Kauffman bracket skein module of a surface
$\Sigma$ is defined as the Kauffman bracket skein module of the
cylinder $\Sigma\times [0,1]$ over the surface $\Sigma$. In this
setup, the skein module carries a natural multiplication structure
and becomes an algebra which is referred as the Kauffman bracket
skein algebra. In the past decades, the Kauffman bracket skein
algebras of surfaces have been proved to have connections and
applications to many interesting objects such as character
varieties,  quantum invariants, quantum Teichm\"uller spaces,
cluster algebras and so on (See e.g. \cite{Bu97,BW11,Le15, Mu16,
Th14}).

On the other hand side, it is well-known that there exists the
two-variable generalization of the Jones polynomial, i.e. HOMFLY-PT
polynomial \cite{FYHLMO85} which  can be defined through two local
relations named HOMFLY-PT skein relations as shown in Figure 1,
where $z=q-q^{-1}$.

Therefore, it is natural to define the HOMFLY-PT skein modules for a
general oriented $3$-manifold $M$ by imposing the HOMFLY-PT skein
relations \cite{Mo93}. The HOMFLY-PT skein theory was studied
extensively by Morton and his collaborators
\cite{Mo93,Ai96,Ai97-1,Ai97-2,MA98,Lu01,Mo02,Lu05,HaMo06,Mo07,MoMa08,MoSa17}
during the past decades. In the following context, we attempt to
provide a comprehensive review of HOMFLY-PT skein theory.

\subsection{HOMFLY-PT skein algebras}
The account here largely follows those of \cite{MoSa17,HaMo06,Mo07}.
Throughout this article, we will work over the coefficient ring
$R=\mathbb{C}[a^{\pm 1},q^{\pm 1}, (q-q^{-1})^{-k}]$ with $k$ ranges
over $\mathbb{N}$.

 Suppose $M$ is an
oriented 3-manifold. A framed oriented link in $M$ is a smooth
embedding of $\sqcup S^1\times [0,1]$ up to isotopy in $M$. Let
$\mathcal{L}(M)$ be the free $R$-module spanned by framed oriented
links in $M$, let $\mathcal{L}'(M)\subset \mathcal{L}(M)$ be the
$R$-submodule generated by the skein relations in Figure 1.

\begin{definition}
The {\em framed HOMFLY-PT skein module} $Sk(M)$ of the manifold $M$
is defined as the quotient
\begin{align}
Sk(M):=\mathcal{L}(M)/\mathcal{L}'(M).
\end{align}
\end{definition}

By its definition, the HOMFLY-PT skein $Sk(M)$ has several basic
properties, such as

(i) The $Sk(M)$ is graded by the first homology group $H_1(M)$,
since each skein relation involves only links in the same homology
class.

(ii) An oriented embedding $f: M\rightarrow N$ induced an $R$-linear
map $f_*: Sk(M)\rightarrow Sk(N)$. When $f$ is a homeomorphism the
map $f_*$ is an isomorphism.

(iii) In particular, when $F$ is an orientable surface and
$M=F\times [0,1]$, we use the notation $Sk(F)$ in place of
$Sk(F\times [0,1])$, and refer to $Sk(F)$ as the HOMFLY-PT skein of
the surface $F$. In this setup, $Sk(F)$ becomes an algebra over the
coefficient ring $R$. The product is given by stacking links. The
grading is additive under the product. Usually, we call $Sk(F)$ the
{\em HOMFLY-PT skein algebra} of the surface $F$.

\begin{remark}
One can fix $n$-points set $S\subset F$ and include $n$ arcs from
$S\times \{0\}$ to $S\times \{1\}$ which are called the $n$-tangles.
Similarly, all the $R$-linear combinations of $n$-tangles in $(F,S)$
up to isotopy module the skein relations gives an algebra denoted by
$Sk_n(F)$. When $n=0$, $Sk_0(F)$ is just the skein algebra $Sk(F)$.
\end{remark}

\begin{remark}
Framed links in $F\times [0,1]$ can be represented by diagrams in
$F$, with the blackboard framing from $F$. Therefore, one can also
regard elements of $Sk(F)$ as diagrams in $F$ modulo Reidemeister
moves $RII$ and $RIII$ and skein relations as shown in Figure 1. It
is easy to follow that the removal of an unknot is equivalent to
time a scalar $s=\frac{a-a^{-1}}{q-q^{-1}}$, i.e. we have the
relation showed in Figure 2.
\begin{figure}[!htb]
\begin{center}
\includegraphics[width=150 pt]{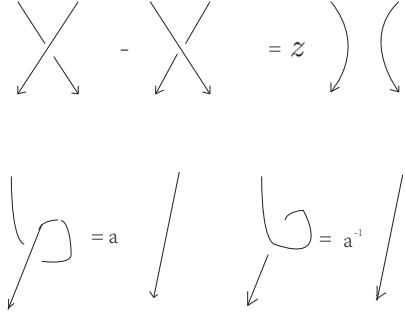}
\end{center}
\caption{Local relations}
\end{figure}

\begin{figure}[!htb]
\begin{center}
\includegraphics[width=80 pt]{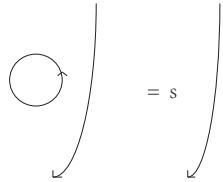}
\end{center}
\caption{Removal of an unknot}
\end{figure}
\end{remark}

\begin{remark}
For every surface $F$, we define the mirror map $- :
Sk(F)\rightarrow Sk(F)$ as follows. For every tangle $T$ in $F$, we
define $\overline{T}$ to be the $T$ with all its crossings switched.
As to the coefficient ring $R$, we define $\overline{q}=q^{-1}$ and
$\overline{a}=a^{-1}$. Then, these operations induce a linear
automorphism respects to the skein relation $- : Sk(F)\rightarrow
Sk(F)$, which was call the mirror map on $Sk(F)$.
\end{remark}

In the following, we present several basic examples of the HOMFLY-PT
skein algebras.

\subsubsection{The plane $\mathbb{R}^2$}

When $F$ is the plane $\mathbb{R}^2$, it is easy to follow that
every element in $Sk(\mathbb{R}^2)$ can be represented as a scalar
in the ring $R$. For a link $\mathcal{L}$
with a diagram $D_{\mathcal{L}}$, the resulting scalar $\langle D_{\mathcal{L%
}} \rangle \in R$ is the framed  HOMFLY-PT polynomial $%
\mathcal{H}(\mathcal{L};q,a)$ of the link $\mathcal{L}$, i.e. $\mathcal{H}(%
\mathcal{L};q,a)=\langle D_\mathcal{L}\rangle$. We use the convention $%
\langle \ \rangle=1$ for the empty diagram, so $\mathcal{H}(U;q,a)=\frac{%
a-a^{-1}}{q-q^{-1}}$. The two relations shown in Figure 1 lead to
\begin{align}
&\mathcal{H}(\mathcal{L}_+;q,a)-\mathcal{H}(\mathcal{L}_-;q,a)=z\mathcal{H}(%
\mathcal{L}_0;q,a), \label{classicalskein1} \\
&\mathcal{H}(\mathcal{L}^{+1};q,a)=a\mathcal{H}(\mathcal{L};q,a) \
\text{and} \
\mathcal{H}(\mathcal{L}^{-1};q,a)=a^{-1}\mathcal{H}(\mathcal{L};q,a).
\label{classicalskein2}
\end{align}
Under our notation, the classical HOMFLY-PT polynomial of a link
$\mathcal{L}$ is given by
\begin{align}
P(\mathcal{L};q,a)=\frac{a^{-w(\mathcal{L})}\mathcal{H}(\mathcal{L};q,a)}{%
\mathcal{H}(U;q,a)},
\end{align}
where $w(\mathcal{L})$ denotes the writhe number of the link
$\mathcal{L}$.

\subsubsection{The rectangle}
We write $H_{n,m}(q,a)$  for the skein algebra $Sk_{n,m}(F)$ of
$(n,m)$-tangles, where $F$ is the rectangle with $n$ inputs and $m$
outputs at the top and matching inputs and outputs at the bottom.
There is a natural algebra structure on $H_{n,m}(q,a)$ by placing
tangles one above the another. When $m=0$, we write
$H_n(q,a)=H_{n,0}(q,a)$ for brevity.
\begin{remark}
It is well-known that the Hecke algebra $H_n$ of type $A_{n-1}$  is
the algebra generated by $\sigma_i: i=1,...,n-1$, subjects to the
relations:
\begin{align*}
&\sigma_i\sigma_j=\sigma_j\sigma_i: |i-j|>1 \\\nonumber
&\sigma_i\sigma_{i+1}\sigma_i=\sigma_{i+1}\sigma_i\sigma_{i+1}:
1\leq i\leq n-1; \\\nonumber
 &\sigma_i-\sigma_i^{-1}=z.
\end{align*}
It is straight to see that the Hecke algebra $H_n$, with
$z=q-q^{-1}$ and coefficient ring extended to include $a^{\pm 1}$
and $q^{\pm 1}$, is isomorphic to the skein algebra $H_{n}(q,a)$
described above. In this algebra the extra variable $a$ in the
coefficient ring allows us to reduce general tangles to linear
combinations of braids, by means of the skein relations.
\end{remark}

The skein theory construction of the idempotent elements
$y_{\lambda}$ of $H_{n}(q,a)$ was presented in \cite{MA98}. This
construction and the main results are briefly reviewed in next
section.

\begin{figure}[!htb] \label{Tn}
\begin{center}
\includegraphics[width=150 pt]{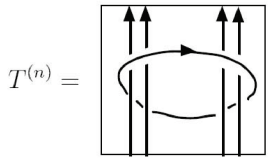}
\end{center}
\caption{The element $T^{(n)}$}
\end{figure}

Set $T^{(n)}$ to be the element  in $H_n(q,a)$ as illustrated in
Figure 3, and let
\begin{align}
t_{\lambda}=za\sum_{x\in \lambda}q^{2cn(x)}+\frac{a-a^{-1}}{z}.
\end{align}
where $cn(x)=j-i$ is the content of the cell in row $i$ and column
$j$ of the Young diagram for the partition $\lambda$. Then we have
\begin{align} \label{formula-TElamda}
T^{(n)}y_{\lambda}=t_{\lambda}y_{\lambda}.
\end{align}

Since a partition is uniquely determined by its contents, it is
readily verified that $t_{\lambda}=t_{\mu}$ if and only if
$\lambda=\mu$.

Similarly, we set $T^{(n,m)}$ to be the analogue elements in
$H_{n,m}(q,a)$, and let
\begin{align} \label{formula-tlambdamu}
t_{\lambda,\mu}=z\left(a\sum_{x\in
\lambda}q^{2cn(x)}-a^{-1}\sum_{x\in
\mu}q^{-2cn(x)}\right)+\frac{a-a^{-1}}{z}.
\end{align}
It is also clear that $t_{\lambda,\mu}=t_{\lambda',\mu'}$ if and
only if $\lambda=\lambda'$ and $\mu=\mu'$.

\subsubsection{The annulus}
When $F$ is the annulus $S^1\times [0,1]$, usually, we use
$\mathcal{C}$ to denote the skein algebra $Sk(S^1\times [0,1])$.
$\mathcal{C}$ is a commutative algebra with the product induced by
placing one annulus outside another.

The closure map $\hat{} : H_{n,m}(q,a)\rightarrow \mathcal{C}$ is an
$R$-linear map induced by taking an $(n,m)$-tangle $T$ to its
closure $\hat{T}$ in the annulus, whose image we call
$\mathcal{C}^{(n,m)}$. It is obvious that every diagram in the
annulus represents an element in some $\mathcal{C}^{(n,m)}$.

Turaev \cite{Tu88-1} showed that $\mathcal{C}$ is freely generated
as an algebra by the set $\{A_m: m\in \mathbb{Z}\}$ where $A_m,
m\neq 0$, is represented by the closure of the braid
$\sigma_{|m|-1}\cdots \sigma_2\sigma_1$. The orientation of the
curve around the annulus is counter-clockwise for positive $m$ and
clockwise for negative $m$. The element $A_0$ is the identity
element and is represented by the empty diagram. The algebra
$\mathcal{C}$ is the product of two subalgebras $\mathcal{C}^+$ and
$\mathcal{C}^-$ generated by $\{A_m: m\in \mathbb{Z}, m\geq 0\}$ and
$\{A_m: m\in \mathbb{Z}, m\leq 0\}$ respectively.

Take a diagram $X$ in the annulus and link it once with a simple
meridian loop, oriented in either direction, to give diagrams
$\varphi(X)$ and $\bar{\varphi}(X)$ in the annulus. This induces
linear endomorphisms $\varphi, \bar{\varphi}$ of the skein
$\mathcal{C}$, called the {\em meridian maps}. Each space
$\mathcal{C}^{(n,m)}$ is invariant under $\varphi$ and
$\bar{\varphi}$. We set $\varphi^{(n)}=\varphi|_{\mathcal{C}^{(n)}}$
and $\varphi^{(n,m)}=\varphi|_{\mathcal{C}^{(n,m)}}$.

\begin{remark}
Rotating diagrams in the annulus $S^1\times I$ about the horizontal
axis through the distinguished boundary points induces a linear
automorphism of the skein $\mathcal{C}$, which we denote by $* :
\mathcal{C}\rightarrow \mathcal{C}$. It is clear to see that
$A_m^{*}=A_{-m}$ and then $(\mathcal{C}^{+})^*=\mathcal{C}^{-}$ and
$(\mathcal{C}^{n,m})^*=\mathcal{C}^{m,n}$.
\end{remark}

\subsubsection*{The subspace $\mathcal{C}^{(n)}\subset \mathcal{C}^+$}

The subspace $\mathcal{C}^{(n)}$ is spanned by monomials in
$\{A_m\}$, with $m\in \mathbb{Z}^+$, of total weight $n$, where the
weight of $A_m$ is $m$. It is clear that this spanning set consists
of $\pi(n)$ elements, the number of partitions of $n$.
$\mathcal{C}^+$ is then graded as an algebra
\begin{align}
\mathcal{C}^+=\bigoplus_{n=0}^{\infty}\mathcal{C}^{(n)}.
\end{align}

 This subspace $\mathcal{C}^{+}$ can also viewed as the
image of $H_n(q,a)$ under the closure map $\hat{}$.
 Given the element $T^{(n)}\in H_n(q,a)$ as illustrated in Figure 3, take an element
$S\in H_n(q,n)$ with $\hat{S}\in \mathcal{C}^{(n)}$ and compose it
by $T^{(n)}$. Then
\begin{align}
\widehat{T^{(n)}S}=\varphi^{(n)}(\hat{S}).
\end{align}
Similarly,
\begin{align}
\widehat{\bar{T}^{(n)}S}=\bar{\varphi}^{(n)}(\hat{S}).
\end{align}

Then it is easy to show  that the eigenvalues of $\varphi^{(n)}$ are
all distinct. Indeed, set $Q_{\lambda}=\widehat{y_{\lambda}}\in
\mathcal{C}^{(n)}$. Taking closures on both sides of the formula
(\ref{formula-TElamda}), we immediately obtain
\begin{align}
\varphi^{(n)}(Q_{\lambda})=t_{\lambda}Q_{\lambda}.
\end{align}
The element $Q_\lambda$ is then an eigenvector of $\varphi^{(n)}$
with eigenvalue $t_{\lambda}$. There are $\pi(n)$ of these
eigenvectors, and the eigenvalues are all distinct. Hence all these
$Q_{\lambda}$ are linear independent. Since $\mathcal{C}^{(n)}$ is
spanned by $\pi(n)$ elements we can deduce that all the elements
$Q_{\lambda}$ with $|\lambda|=n$ form a basis for
$\mathcal{C}^{(n)}$ and that the eigenspaces are all
$1$-dimensional.

It was shown in \cite{Lu01,Lu05} that  $Q_\lambda$ can be expressed
as an explicit integral polynomial in $\{h_m\}_{m\geq 0}$ and
$\mathcal{C}^+$ can be regarded as the ring of symmetric functions
in variables $x_1,..,x_N,..$ with the coefficient ring $R$. In this
situation, $\mathcal{C}^{(n)}$ consists of the homogeneous functions
of degree $n$. The power sum $P_n=\sum x_i^m$ is a symmetric
function which can be presented in terms of the complete symmetric
functions, hence it represents a skein element which is also denoted
by $P_n\in \mathcal{C}^{(m)}$. Moreover, we have the identity
\begin{align}
[m]P_m=X_m=\sum_{j=0}^{m-1}A_{m-1-j,j}
\end{align}
where $[m]=\frac{q^m-q^{-m}}{q-q^{-1}}$ and $A_{i,j}$ is the closure
of the braid $$\sigma_{i+j}\sigma_{i+j-1}\cdots
\sigma_{j+1}\sigma_{j}^{-1}\cdots \sigma_1^{-1}.$$ Given a partition
$\mu$, we define
\begin{align}
P_{\mu}=\prod_{i=1}^{l(\mu)}P_{\mu_i}.
\end{align}
According to the Frobenius formula for Schur function, we have
\begin{align} \label{formula-frobeniusQ}
Q_{\lambda}=\sum_{\mu}\frac{\chi_{\lambda}(\mu)}{\frak{z}_{\mu}}P_\mu.
\end{align}

\subsubsection*{The subspace $\mathcal{C}^{(n,m)}$} The image of
$H_{n,m}(q,a)$ under the closure map is denoted by
$\mathcal{C}^{(n,m)}\subset \mathcal{C}$. Unlike the case for
$\mathcal{C}^{(n)}$ where $\mathcal{C}^{(n)}\cap
\mathcal{C}^{(n-1)}=\emptyset$, we have that:

\makeatletter
\let\@@@alph\@alph
\def\@alph#1{\ifcase#1\or \or $'$\or $''$\fi}\makeatother
\begin{subnumcases}
{\mathcal{C}^{(n,m)}\supset\mathcal{C}^{(n-1,m-1)}\supset
\cdots\supset}
 \mathcal{C}^{(n-m,0)}, &$\min(n,m)=m$, \\\nonumber
 \mathcal{C}^{(0,m-n)}, &$\min(n,m)=n$.
\end{subnumcases}
\makeatletter\let\@alph\@@@alph\makeatother and
\begin{align}
\mathcal{C}^{(n,0)}=\mathcal{C}^{(n)}_{(+)}, \
\mathcal{C}^{(0,m)}=\mathcal{C}^{(m)}_{(-)}.
\end{align}
where the subscripts indicate the direction of the strings around
the center of the annulus. We have that $\mathcal{C}^{(n_1,p_1)}\cap
\mathcal{C}^{(n_2,p_2)}=\emptyset$ if $n_1-m_1\neq n_2-m_2$. Indeed,
$\mathcal{C}^{(n,m)}$ is spanned by suitably weighted monomials in
\begin{align}
\{A_{n},...,A_{1},A_0,A_{-1},...,A_{-m}\}.
\end{align}
We can then see that
\begin{align}
\mathcal{C}^{(n,m)}=\left(\mathcal{C}_{+}^{(n)}\times
\mathcal{C}_{-}^{(m)}+\mathcal{C}^{(n-1,m-1)}\right).
\end{align}
The spanning set of $\mathcal{C}^{(n,m)}$ consists of $\pi(n,m)$
elements, where
\begin{align}
\pi(n,m):=\sum_{j=0}^{k}\pi(n-j)\pi(m-j).
\end{align}
with $k=\min(n,m)$.

Similar to grading of $\mathcal{C}^+$ with the $\mathcal{C}^{(n)}$,
 the full skein $\mathcal{C}$ can be written in terms of
 $\mathcal{C}^{(n,m)}$ as follow:
\begin{align}
\mathcal{C}=\bigoplus_{k=-\infty}^{\infty}\left(\bigcup_{n,m\geq
0}\{\mathcal{C}^{(n,m)}: n-m=k\}\right).
\end{align}

Now, we consider the meridian map:
\begin{align}
\varphi^{(n,m)}: \mathcal{C}^{(n,m)}\rightarrow \mathcal{C}^{(n,m)}
\end{align}
It is proved in \cite{MoHa02} that
\begin{proposition} [See Theorem 6 in \cite{MoHa02}]
$t_{\lambda,\mu}$ given by formula (\ref{formula-tlambdamu}) with
$|\lambda|\leq n, |\mu|\leq m$ and $|\lambda|-|\mu|=m$ are
eigenvalues of the meridian map $\varphi^{(n,m)}$, moreover, they
occur with multiplicity $1$.
\end{proposition}
A straightforward consequence is
\begin{corollary}
There is a basis of $\mathcal{C}^{(n,m)}$ given by
\begin{align}
\{Q_{\lambda,\mu}: |\lambda|\leq n, |\mu|\leq m,
|\lambda|-|\mu|=n-m\}
\end{align}
such that $\varphi(Q_{\lambda,\mu})=t_{\lambda,\mu}Q_{\lambda,\mu}$.
\end{corollary}
Then in \cite{HaMo06}, R. Hadji and H. Morton extended the method of
\cite{Lu01,Lu05} and obtained the explicit expression for these
basis elements $\{Q_{\lambda,\mu}\}$.

\subsubsection*{Explicit expression for the basis elements $Q_{\lambda,\mu}$}
Let element $h_m\in \mathcal{C}^{(m)}$ be the closure of the
elements $\frac{1}{\alpha_m}a_m\in H_m(q,a)$, i.e.
$h_m=\frac{1}{\alpha_m}\hat{a}_m$, where $a_m$ is the
quasi-idempotent whose definition is given by formula
(\ref{formula-anbn}), and $\alpha_m$ is the scalar given by
$\alpha_m=q^{m(m-1)/2}\prod_{i=1}^m\frac{q^i-q^{-i}}{q-q^{-1}}\in
R$.

There is a readily defined involution on the skein $\mathcal{C}$
denoted by $*$. It is the rotation of diagrams in the annulus
$S^1\times [0,1]$ by $\pi$ about the horizontal axis through the
distinguished boundary points. We can see that $(A_m)^{*}=A_{-m}$,
so that $(\mathcal{C}^+)^*=\mathcal{C}^{-}$. Furthermore,
$(\mathcal{C}^{(n,m)})^*=\mathcal{C}^{(m,n)}$.

The skein $\mathcal{C}^+$ (resp. $\mathcal{C}^-$) is spanned by the
monomials in $\{h_m\}_{m\geq 0}$ (resp. $\{h_k^*\}_{k\geq 0}$),
indeed, the explicit relations between $\{h_m\}$ and the Turaev's
basis $\{A_m\}$ are presented in \cite{Mo02} (cf.  Theorem 3.6 in
\cite{Mo02}) and \cite{MoMa08} (cf. Theorem 8 in \cite{MoMa08} ).
Then the whole skein $\mathcal{C}$ is spanned by the monomials in
$\{h_m, h_k^*\}_{m,k\geq 0}$.

 Given
two partitions $\lambda, \mu$ with $l$ and $r$ parts. We first
construct a $(l+r)\times (l+r)$-matrix $M_{\lambda,\mu}$ with
entries in $\{h_m, h_k^*\}_{m,k\in \mathbb{Z}}$ as follows, where we
have let $h_{m}=0$ if $m<0$ and $h_{k}^*=0$ if $k <0$.

\begin{align}
M_{\lambda,\mu}=
\begin{pmatrix}
h_{\mu_{r}}^* & h_{\mu_{r}-1}^* & \cdots & h_{\mu_{r}-r-l+1}^* \\
h_{\mu_{r-1}+1}^* & h_{\mu_{r-1}}^* & \cdots & h_{\mu_{r-1}-r-l}^*\\
\cdot & \cdot & \cdots & \cdot \\
h_{\mu_1+(r-1)}^* & h_{\mu_1+(r-2)}^* & \cdots & h_{\mu_1-l}^*\\
h_{\lambda_1-r} & h_{\lambda_1-(r-1)} & \cdots & h_{\lambda_1+l-1}\\
\cdot & \cdot & \cdots & \cdot \\
h_{\lambda_l-l-r+1} & h_{\lambda_l-s-r+2} & \cdots & h_{\lambda_l}
\end{pmatrix}
\end{align}
It is easy to see that the subscripts of the diagonal entries in the
$h$-rows are the parts $\lambda_1,\lambda_2,...,\lambda_l$ of
$\lambda$ in order, while the subscripts of the diagonal entries in
the $h^*$-rows are the parts $\mu_1,\mu_2,..,\mu_r$ of $\mu$ in
reverse order.

Then, $Q_{\lambda,\mu}$ is defined as the determinant of the matrix
$M_{\lambda,\mu}$, i.e.
\begin{align} \label{formula-Qdet}
Q_{\lambda,\mu}=\det M_{\lambda,\mu}.
\end{align}

\begin{example}
For two partitions $\lambda=(4,2,2)$ and $\mu=(3,2)$, we have
\begin{align}
Q_{\lambda,\mu}=\det
\begin{pmatrix}
h_{2}^* & h_{1}^* & 1 & 0 & 0 \\
h_{4}^* & h_{3}^* & h_2^* & h_1^* & 1 \\
h_{2} & h_{3} & h_4 & h_5 & h_6 \\
0 & 1 & h_1 & h_2 & h_3 \\
0 & 0 & 1 & h_1 & h_2
\end{pmatrix}.
\end{align}
\end{example}

Furthermore, we have
\begin{align} \label{formula-Qlambdamu}
Q_{\lambda,\mu}=\sum_{\sigma,\rho,\nu}(-1)^{|\sigma|}c_{\sigma,\rho}^{\lambda}c_{\sigma^t,\nu}^{\mu}Q_{\rho,\emptyset}
Q_{\emptyset,\nu},
\end{align}
where $c_{\sigma,\rho}^{\lambda}$ is the Littlewood-Richardson
coefficient which is determined by the product formula for Schur
functions:
\begin{align}
s_{\sigma}(\mathbf{x})s_{\rho}(\mathbf{x})=\sum_{\lambda}c_{\sigma,\rho}^{\lambda}s_{\lambda}(\mathbf{x}).
\end{align}
Moreover, the basis elements $Q_{\lambda,\mu}$ of $\mathcal{C}$ have
the property that the product of any two is a non-negative integral
linear combination of basis elements.

\subsection{Colored HOMFLY-PT invariants}
Let $\mathcal{L}$ be a framed link with $L$ components with a fixed
numbering. For diagrams $Q_1,..,Q_L$ in the skein model of annulus
with the positive oriented core $\mathcal{C}^+$, we define the
decoration of $\mathcal{L}$ with $Q_1,..,Q_L$ as the link
\begin{align}
\mathcal{L}\star \otimes_{i=1}^{L} Q_i
\end{align}
which derived from $\mathcal{L}$ by replacing every annulus
$\mathcal{L}$ by the annulus with the diagram $Q_i$ such that the
orientations of the cores match. Each $Q_i$ has a small backboard
neighborhood in the annulus which makes the decorated link
$\mathcal{L}\otimes_{i=1}^{L}Q_i$ into a framed link. (see Figure 4
for a framed trefoil $\mathcal{K}$ decorated by skein element
$\mathcal{Q}$).

\begin{figure}[!htb] \label{figure-trefoil}
\begin{align*}
\mathcal{K} \qquad\qquad\qquad\quad \mathcal{Q}
\qquad\qquad\qquad\quad
\mathcal{K}\star \mathcal{Q} \\
\includegraphics[width=50 pt]{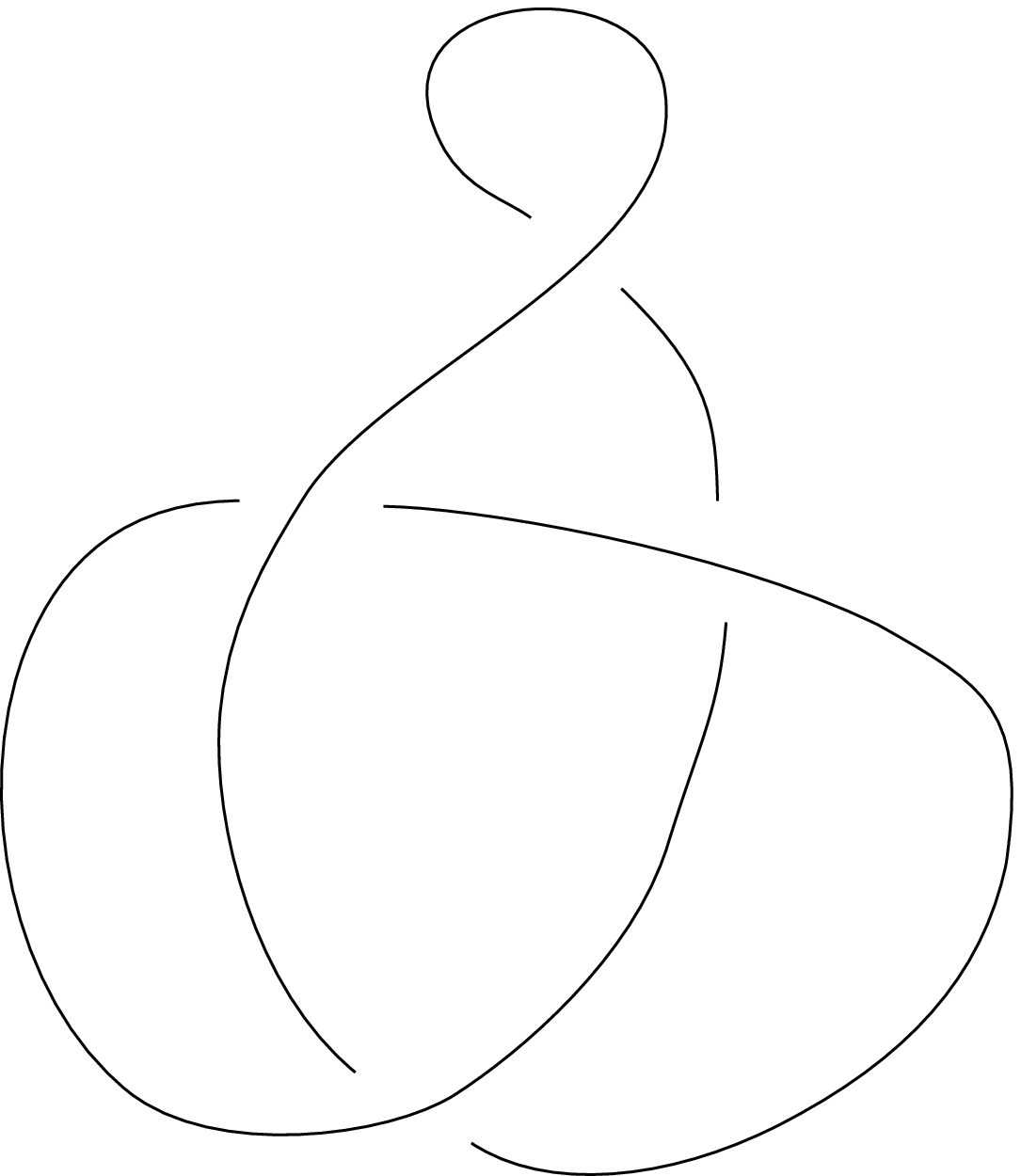}\qquad\qquad
\includegraphics[width=50
pt]{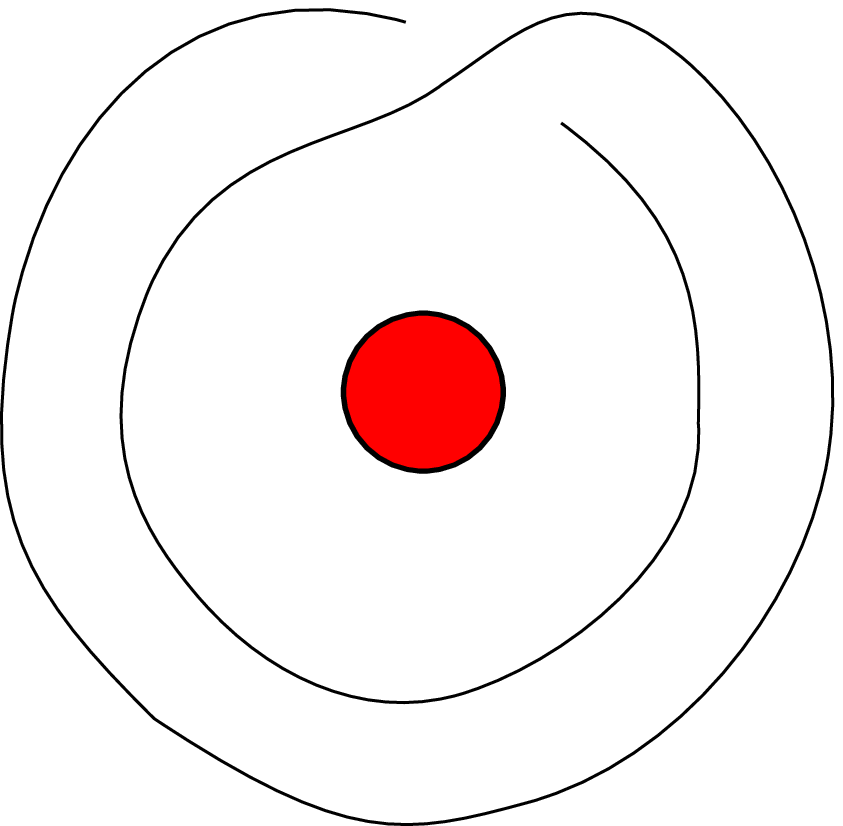} \qquad\quad
\includegraphics[width=50
pt]{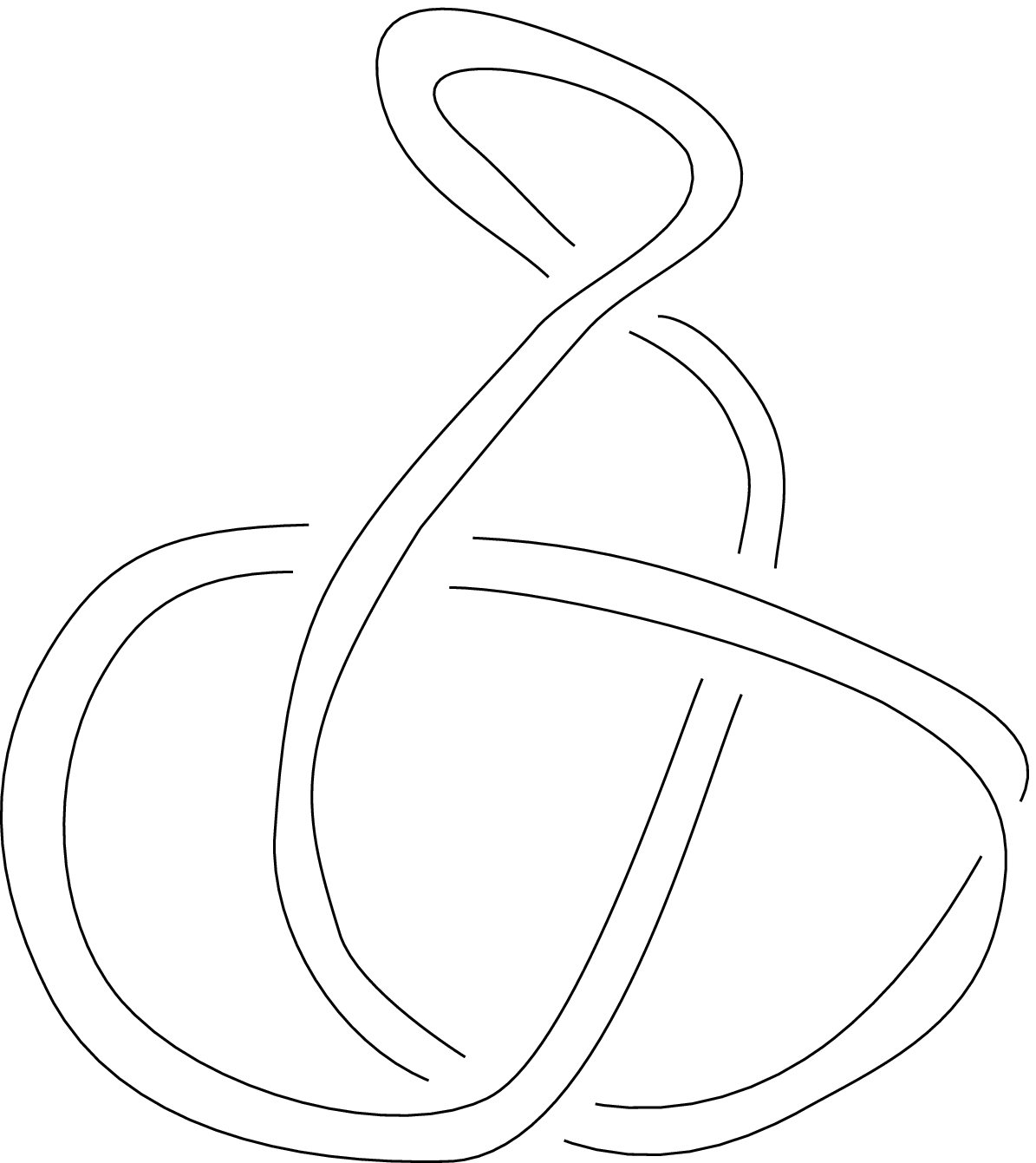}
\end{align*}%
\caption{$\mathcal{K}$ decorated by $\mathcal{Q}$}
\end{figure}

The framed colored HOMFLY-PT polynomial of $\mathcal{L}$ is defined
to be the framed HOMFLY-PT polynomial of the decorated link
$\mathcal{L}\star\otimes_{i=1}^{L}Q_i$ which is given by
\begin{align}
\mathcal{H}(\mathcal{L}\star\otimes_{i=1}^{L}Q_i;q,a)=\langle
\mathcal{L}\star\otimes_{i=1}^{L}Q_i \rangle.
\end{align}

In particular, if we choose $L$ basis elements
 $Q_{\lambda^{\alpha},\mu^{\alpha}}$ with $\alpha=1,...,L$ in the full skein $\mathcal{C}$,

\begin{definition}
The {\em framed full colored HOMFLY-PT invariant}
 of the framed link
$\mathcal{L}$ denoted by
$\mathcal{H}_{[\vec{\lambda},\vec{\mu}]}(\mathcal{L};q,a)$, is
defined as the framed HOMFLY-PT polynomial of the decorated link
$\mathcal{L}\star
\otimes_{\alpha=1}^LQ_{\lambda^{\alpha},\mu^{\alpha}}$, i.e.
\begin{align}
\mathcal{H}_{[\vec{\lambda},\vec{\mu}]}(\mathcal{L};q,a):=\mathcal{H}(\mathcal{L}\star
\otimes_{\alpha=1}^L Q_{\lambda^{\alpha},\mu^{\alpha}};q,a).
\end{align}
\end{definition}

In particular, when all $\mu^{\alpha}=\emptyset$, i.e.
$Q_{\lambda^{\alpha},\emptyset}=Q_{\lambda^{\alpha}}$. In this case,
the framing factor for $Q_{\lambda}$ is given by
$q^{-\kappa_\lambda}a^{-|\lambda|}$. So we can add a framing factor
to eliminate the framing dependency. It makes the framed colored
HOMFLY-PT invariant $ \mathcal{H}_{\vec{\lambda}}(\mathcal{L};q,a) $
into a framing independent invariant.
\begin{definition}
The (framing independent) colored HOMFLY-PT invariant for the link
$\mathcal{L}$ is given by
\begin{align} \label{formula-coloredhomfly-skein}
W_{\vec{\lambda}}(\mathcal{L};q,a)=q^{-\sum_{\alpha=1}^{L}\kappa_{\lambda^\alpha}|w(\mathcal{K_\alpha})|}
a^{-\sum_{\alpha=1}^{L}|\lambda^\alpha|w(\mathcal{K}_\alpha)}\mathcal{H}_{\vec{\lambda}}(
\mathcal{L};q,a)
\end{align}
where $\vec{\lambda}=(\lambda^1,..,\lambda^L)\in \mathcal{P}^L$.
\end{definition}

\begin{example}
For the unknot $U$,
\begin{align} \label{formula-unknot1}
W_{\mu}(U;q,a)=\mathcal{H}_{\mu}(U;q,a)=\prod_{x\in
\mu}\frac{aq^{cn(x)}-a^{-1}q^{-cn(x)}}{q^{h(x)}-q^{-h(x)}}.
\end{align}
By the formula (\ref{formula-Qlambdamu}), we have
\begin{align} \label{formula-unknot2}
\mathcal{H}_{[\lambda,\mu]}(U;q,a)=\langle Q_{\lambda,\mu}
\rangle=\sum_{\sigma,\rho,\nu}(-1)^{|\sigma|}c_{\sigma,\rho}^{\lambda}c_{\sigma^t,\nu}^{\mu}W_{\rho}(U;q,a)W_{\nu}(U;q,a).
\end{align}
\end{example}

\begin{definition}
Given a knot framed $\mathcal{K}$, the {\em normalized framed full
HOMFLY-PT invariants} of $\mathcal{K}$ is defined as follow
\begin{align}
\mathcal{P}_{[\lambda,\mu]}(\mathcal{K};q,a)=\frac{\mathcal{H}_{[\lambda,\mu]}(\mathcal{K};q,a)}{\mathcal{H}_{[\lambda,\mu]}(U;q,a)}.
\end{align}
\end{definition}
In other words,
\begin{align} \label{formula-Hlambdamu}
\mathcal{H}_{[\lambda,\mu]}(\mathcal{K};q,a)=
\mathcal{P}_{[\lambda,\mu]}(\mathcal{K};q,a)\mathcal{H}_{[\lambda,\mu]}(U;q,a).
\end{align}

If we draw the framed knot $\mathcal{K}$ in the annulus as the
closure of a $1$-tangle, then decorating it by $Q\in \mathcal{C}$
gives a diagram of $\mathcal{K}\star Q$ in the annulus, see Figure 5
for the case when $\mathcal{K}$ is the trefoil knot.

\begin{figure}[!htb]
\includegraphics[width=100 pt]{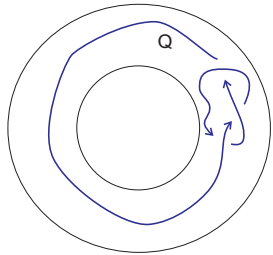}
\caption{The element $\mathcal{K}\star Q$ in the annulus}
\end{figure}

This construction induces a linear map $T_{\mathcal{K}}:
\mathcal{C}\rightarrow \mathcal{C}$. Suppose $Q$ is the eigenvector
of $T_{\mathcal{K}}$ with eigenvalue $t_{\mathcal{K}}(q,a)$, then
\begin{align} \label{formula-tkqa}
\mathcal{K}\star
Q=T_{\mathcal{K}}(Q)=t({\mathcal{K}};q,a)Q=t({\mathcal{K}};q,a)U\star
Q.
\end{align}

In particular, when $Q=Q_{\lambda,\mu}$, taking the HOMFLY-PT
polynomial, we obtain
\begin{align}
t(\mathcal{K};q,a)=\frac{\mathcal{H}_{[\lambda,\mu]}(\mathcal{K};q,a)}{\mathcal{H}_{[\lambda,\mu]}(U;q,a)}
\end{align}

Therefore,  $\mathcal{P}_{[\lambda,\mu]}(\mathcal{K};q,a)$ is also
referred as the framed HOMFLY-PT 1-tangle invariants in \cite{Mo07}.

For a link $\mathcal{L}$ with $L$ components, and let
$\vec{\nu}=(\nu^1,...,\nu^L), \
\vec{\lambda}=(\lambda^1,...,\lambda^L),
\vec{\mu}=(\mu^1,...,\mu^L)\in \mathcal{P}^L$, set
$c_{\vec{\lambda},\vec{\mu}}^{\vec{\nu}}=\prod_{\alpha=1}^Lc_{\lambda^\alpha,\mu^\alpha}^{\nu^\alpha}$,
where $c_{\lambda^\alpha,\mu^\alpha}^{\nu^\alpha}$ is the
Littlewood-Richardson coefficient.

\begin{definition} We define the framed composite invariants for $\mathcal{L}$ as follow
\begin{align}
\mathcal{C}_{\vec{\nu}}(\mathcal{L};q,a)=\sum_{\vec{\lambda},\vec{\mu}}c_{\vec{\lambda},\vec{\mu}}^{\vec{\nu}}
\mathcal{H}_{[\vec{\lambda},\vec{\mu}]}(\mathcal{L};q,a).
\end{align}
\end{definition}

\section{Idempotents in Hecke algebras}  \label{Section-Idempotent}

In this section, we  first briefly review the skein theory
construction of the idempotent of the Hecke algebra $H_n(q,a)$
follow \cite{MA98}. Then we establish an important formula which
will be used in the next section.
\subsection{Young diagrams}
Let $\lambda=(\lambda_1,...,\lambda_l)$ be a partition of $n$, then
$\lambda$ can be represented by a Young diagram which is a
collection of $n$ cells arranged in rows, with $\lambda_1$ cells in
the first row, $\lambda_2$ cells in the second row up to
$\lambda_{l(\lambda)}$. Usually, we denote both the partition and
its Young diagram by $\lambda$. The conjugate of $\lambda^{\vee}$ of
$\lambda$ is the Young diagram whose rows form the columns of
$\lambda$. For the cell, in the $i$-th row and $j$-th column of
$\lambda$ we write $(i,j)\in \lambda$, and refer to $(i,j)$ as the
coordinates of the cell.

A standard tableau $T'(\lambda)$ is an assignment of the numbers $1$
to $n$ to the cells of $\lambda$ such that the numbers increase from
left to right along the rows and top to bottom down the columns. In
particular, $T(\lambda)$ will denote the tableau where the cells of
the Young diagram are numbered from $1$ to $n$ along rows. Note that
the transposition of rows and columns doesn't take $T(\lambda)$ to
$T(\lambda^{\vee})$. We define the permutation $\pi_{\lambda}$ by
$\pi_{\lambda}(i)=j$ where the transposition of $\lambda$ carries
the cell $i$ in $T(\lambda)$ to the cell $j$ in $T(\lambda^{\vee})$.

Let $\lambda$ and $\mu$ be Young diagrams with $|\lambda|=|\mu|=n$.
We say that a permutation $\pi\in  S_n$ separates $\lambda$ from
$\mu$ if no pair of numbers in the same row of $T(\lambda)$ are
mapped by $\pi$ to the same row of $T(\mu)$. For example, the
permutation $\pi_\lambda$ separates $\lambda$ from its conjugate
$\lambda^{\vee}$.

Denoted by $R(\lambda)\subset S_n$ the subgroup of the permutations
which preserve the rows of $T(\lambda)$. It is easy to see that if
$\pi$ separates $\lambda$ from $\mu$ then so does $\rho\pi \sigma$
for any $\rho\in R(\lambda)$ and $\sigma\in R(\mu)$, and conversely,
if $\pi$ separates $\lambda$ from $\lambda^{\vee}$ then
$\pi=\rho\pi_{\lambda}\sigma$ with $\rho\in R(\lambda)$ and
$\sigma\in R(\lambda^{\vee})$. We say that $\lambda$ and $\mu$ are
inseparable if no permutation $\pi\in S_{n}$ separates $\lambda$
from $\mu$.

For every permutation $\pi\in S_n$, there exists a unique braid
$\omega_{\pi}$ ( called a positive permutation braid ) which is
uniquely determined by the following properties:

(i) all strings are oriented from top to bottom;

(ii) for $i=1,...,n$ the $i$-th string joins the point numbered $i$
at the top of the braid to the point numbered $\pi(i)$ at the bottom
of the braid,

(iii) all the crossings occur with positive sign and each pair of
strings cross at most once.

\begin{lemma} [cf. \cite{MA98}]
Let $\pi\in S_n$ be a permutation which separates $\lambda$ from
$\lambda^{\vee}$. Then there are $\rho\in R(\lambda)$ and $\sigma\in
R(\lambda^{\vee})$, such that
\begin{align}
\omega_{\pi}=\omega_{\rho}\omega_{\pi_{\lambda}}\omega_{\sigma}.
\end{align}
\end{lemma}
\begin{proof}
Since $\pi$ separates $\lambda$ from $\lambda^{\vee}$, there are
$\rho\in R(\lambda)$ and $\sigma\in R(\lambda^{\vee})$, such that
$\pi=\rho\pi_{\lambda}\sigma$. By the definition of
$\omega_{\pi_{\lambda}}$, the pairs of strings which start in the
same row of $T(\lambda)$ or end in the same row of
$T(\lambda^{\vee})$ would never cross. The only pairs which cross in
$\omega_{\rho}$ or $\omega_{\sigma}$ are in the same row of
$T(\lambda)$ or $T(\lambda^{\vee})$ respectively. It follows that
$\omega_{\rho}\omega_{\pi_{\lambda}}\omega_{\sigma}$ is a positive
permutation braid. Hence
$\omega_{\pi}=\omega_{\rho}\omega_{\pi_{\lambda}}\omega_{\sigma}.$
\end{proof}

\subsection{Constructions and properties}
Let $F=R_{n}^{n}$ be a rectangle with $n$ inputs at the top and $n$
outputs at the bottom. Let $H_n(q,a)$ be the skein $Sk_n(R_n^{n})$
of $n$-tangles. See Figure 6 for an element in $H_{n}(q,a)$.
\begin{figure}[!htb]
\begin{center}
\includegraphics[width=80 pt]{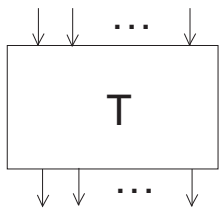}
\end{center}
\caption{An element in $H_n(q,a)$}
\end{figure}

Composing $n$-tangles by placing one above another induces a product
which makes $H_n(q,a)$ into the Hecke algebra of type $A_{n-1}$ with
the coefficients ring $R$. $H_n(q,a)$ has a presentation generated
by the elementary brads $\sigma_1,...,\sigma_{n-1}$ subjects to
relations
\begin{align}
&\sigma_i\sigma_j=\sigma_j\sigma_i, \ \  \ \ |i-j|\geq 2,\\
&\sigma_i\sigma_{i+1}\sigma_i=\sigma_{i+1}\sigma_i\sigma_{i+1}  \ \  \ |i-j|=1  \notag \\
&(\sigma_i-q)(\sigma_i+q^{-1})=0, \ \  i=1,2,...,n-1. \notag
\end{align}

The Hecke algebra $H_n$ can be viewed as the quantum deformation of
the group algebra $\mathbb{C}[S_n]$ of the symmetric group $S_n$,
whose idempotents are described by the classical Young symmetrisers.
For a Young diagram $\lambda$ its Young symmetriser is the product
of the sum of permutation which preserve the rows of $T(\lambda)$
and the alternating sum of permutations which preserve columns. In
order to construct the idempotents in Hecke algebra, the idea is to
make suitably quantum deformation of the classical Young
symmetrisers. The two simplest idempotents in $H_n$, corresponding
to the single row and column Young diagrams, are given algebraically
by Jones \cite{Jo87}. And the skein version of them are described by
H. Morton \cite{Mo93} in terms of positive permutation braids
$\omega_{\pi}, \pi\in S_n$. Then Gyoja \cite{Gy86} constructed the
idempotents for general $\lambda$. Finally, Morton and Aiston
\cite{MA98} provided skein theory construction of the idempotents in
the HOMFLY-PT skein $H_{n}(q,a)$.

We first define Jones's row and column elements $a_j$ and $b_j$,
following the account in Morton \cite{Mo93}. Write
$E_n(\sigma_1,\sigma_2,...,\sigma_{n-1})=\sum_{\pi\in
S_n}\omega_{\pi}$ for the sum of positive permutation braids. We
define
\begin{align} \label{formula-anbn}
a_n&=E_{n}(q\sigma_1,q\sigma_2,...,q\sigma_{n})=\sum_{\pi\in
S_n}q^{l(\pi)}\omega_{\pi}, \\\nonumber
b_n&=E_{n}(-q^{-1}\sigma_1,-q^{-1}\sigma_2,...,-q^{-1}\sigma_{n})=\sum_{\pi\in
S_n}(-q)^{-l(\pi)}\omega_{\pi}.
\end{align}

\begin{lemma} \label{lemma-anT}
The element $a_n$ (resp. $b_n$) can be factorised in $H_n(q,a)$ with
$\sigma_i+q^{-1}$ (resp. $\sigma_i-q$) as a left or a right factor
for any $i\in \{1,2,..,n-1\}$. If we define two linear homomorphisms
$\varphi, \psi$ from the Hecke algebra $H_n(q,a)$ to the ring of
scalars $R$ defined by $\varphi(\sigma_i)=-q^{-1}$ and
$\psi(\sigma_i)=q$ for $i=1,...,n-1$. Then for all $T\in
H_{n}(q,a)$, we have
\begin{align} \label{formula-anT}
a_n T=Ta_n=\psi(T)a_n, \ b_nT=Tb_n=\varphi(T)b_n.
\end{align}
\end{lemma}
\begin{proof}
This lemma was first proved in \cite{Mo93}, for the sake of
completeness, we provide the the proof here. Given $i\in
\{1,2,,,.,n-1\}$, we can pair the permutations as follows. For each
permutation $\pi$, consider its composite $\pi'=\pi\circ(i,i+1)$
with transposition $(i,i+1)$, then $\pi'(i)=\pi(i+1)$ and
$\pi'(i+1)=\pi(i)$. Hence exactly one of them preserves the order of
$i$ and $i+1$, and suppose it is $\pi$ such that $\pi(i)<\pi(i+1)$.
Consider the positive braids $\omega_{\pi}$, and
$\sigma_i=\omega_{(i,i+1)}$, it is clear that $\omega_{\pi}\sigma_i$
is also a positive braid, then $\omega_{\pi'}=\omega_{\pi}\sigma_i$.

Therefore,
\begin{align} \label{formula-anfactor1}
a_n&=\sum_{\pi(i)<\pi(i+1)}q^{l(\pi)}\omega_{\pi}+\sum_{\pi'(i)>\pi'(i+1)}q^{l(\pi')}\omega_{\pi'}\\\nonumber
&=\sum_{\pi(i)<\pi(i+1)}q^{l(\pi)}\omega_{\pi}+\sum_{\pi(i)<\pi(i+1)}q^{l(\pi)+1}\omega_{\pi}\sigma_i\\\nonumber
&=a_n^{(i)}(1+q\sigma_i),
\end{align}
where  $a_n^{(i)}=\sum_{\pi(i)<\pi(i+1)}q^{l(\pi)}\omega_{\pi}$.

By using the quadratic condition of Hecck algebra:
$(\sigma_{i}+q^{-1})(\sigma_{i}-q)=0$ for $1\leq i\leq n-1$, we have
\begin{align}
a_n\sigma_i=a_n^{(i)}(1+q\sigma_i)\sigma_i=q
a_{n}^{i}(1+q\sigma_i)=qa_n.
\end{align}

On the other hand side, for given $i\in \{1,...,n-1\}$, one can also
consider the composition of permutations $\pi'=(i,i+1)\pi$. Then
\begin{align} \label{formula-anfactor2}
a_{n}&=\sum_{\pi^{-1}(i)<\pi^{-1}(i+1)}q^{l(\pi)}\omega_{\pi}+\sum_{\pi'^{-1}(i)>\pi'^{-1}(i+1)}q^{l(\pi')}\omega_{\pi'}\\\nonumber
&=\sum_{\pi^{-1}(i)<\pi^{-1}(i+1)}q^{l(\pi)}\omega_{\pi}+\sum_{\pi^{-1}(i)<\pi^{-1}(i+1)}q^{l(\pi)+1}\sigma_i\omega_{\pi}\\\nonumber
&=(1+q\sigma_i)\tilde{a}^{(i)}_n,
\end{align}
where
$\tilde{a}^{(i)}_n=\sum_{\pi^{-1}(i)<\pi^{-1}(i+1)}q^{l(\pi)}\omega_{\pi}$.
By this formula  we also obtain $\sigma_ia_n=qa_n$. Hence we prove
the formula (\ref{formula-anT}), the proof of the second part of
formula (\ref{formula-anT}) is similar.
\end{proof}

Given a partition
 $\lambda=(\lambda_1,...,\lambda_{l})$ of $n$, we define
\begin{align}
E_{\lambda}&=a_{\lambda_1}\otimes a_{\lambda_2}\otimes \cdots
\otimes a_{\lambda_{l}} \\
F_{\lambda}&=b_{\lambda_1}\otimes b_{\lambda_2}\otimes \cdots
\otimes b_{\lambda_{l}}.
\end{align}

Let $H_\lambda$ to the subalgebra of $H_n(q,a)$ generated by
$\{\omega_\rho, \rho\in R(\lambda)\}$, then $E_{\lambda},
F_{\lambda}\in H_\lambda$ and
\begin{align}
E_{\lambda}T&=TE_{\lambda}=\psi(T)E_{\lambda}, \\\nonumber
F_{\lambda}T&=TF_{\lambda}=\varphi(T)F_{\lambda},
\end{align}
 for any $T\in
H_{\lambda}$ by Lemma  \ref{lemma-anT}.

\begin{lemma} \label{lemma-elambda}
Let $\lambda$ and $\mu$ be two Young diagrams of $n$ cells, and let
$\pi\in S_n$ be a permutation which does not separate $\lambda$ from
$\mu$. Then
\begin{align} \label{formula-elambda}
E_{\lambda}\omega_{\pi}F_{\mu}=0=F_{\lambda}\omega_{\pi}E_{\mu}.
\end{align}
\end{lemma}
\begin{proof}
This proof is essentially follows the Lemma 4.4 in \cite{MA98}.
Since $\pi$ does not separate $\lambda$ from $\mu$, there are two
cells in the same row of $\lambda$ which are sent to two cells in
the same row of $\mu$ by $\pi$. Then one can find $\rho\in
R(\lambda)$ and $\sigma\in R(\mu)$, such that $\pi'=\rho\pi\sigma$
send the two adjacent cells $i$ and $i+1$ in the same row of
$\lambda$ to the two adjacent cells $j$ and $j+1$ in the same row of
$\mu$. Then we have $\omega_{\pi}=T\omega_{\pi'}T'$ for some $T\in
H_\lambda$ and $T'\in H_\mu$. Therefore,
\begin{align}
E_\lambda\omega_\pi F_{\mu}=\psi(T)\varphi(T')E_\lambda\omega_{\pi'}
F_{\mu}.
\end{align}

So we only need to show $E_\lambda\omega_{\pi'}F_\mu=0$.  As to the
permutation $\pi'$, we have $\pi'(i+1)=\pi(i)+1=j+1$, where $(i
i+1)\in R(\lambda)$ and $(j j+1)\in R(\mu)$. By the adjacency, we
obtain $ \sigma_i\omega_{\pi'}=\omega_{\pi'}\sigma_j.$ Hence,
\begin{align}
(q^{-1}+\sigma_i)\omega_{\pi'}=\omega_{\pi'}(q^{-1}+\sigma_j).
\end{align}
By formulas (\ref{formula-anfactor1}) and (\ref{formula-anfactor2}),
we have $ E_\lambda=T (q^{-1}+\sigma_i) $ and
$F_{\mu}=(q-\sigma_j)T'$ for some $T\in H_\lambda$ and $T'\in
H_{\mu}$, then
\begin{align}
E_\lambda\omega_{\pi'}F_{\mu}=T(q^{-1}+\sigma_i)\omega_{\pi'}(q-\sigma_j)T'=T\omega_{\pi'}(q^{-1}+\sigma_j)(q-\sigma_j)T'=0.
\end{align}
The proof of the second identity in (\ref{formula-elambda}) is
similar.
\end{proof}

\begin{corollary} \label{corollary-elambda}
Given a Young diagram $\lambda$, \makeatletter
\let\@@@alph\@alph
\def\@alph#1{\ifcase#1\or \or $'$\or $''$\fi}\makeatother
\begin{subnumcases}
{E_{\lambda}\omega_{\pi}F_{\lambda^{\vee}}=}
(-1)^{l(\sigma)}q^{l(\rho)-l(\sigma)}E_{\lambda}\omega_{\pi_{\lambda}}F_{\lambda^{\vee}}, & if $\omega_{\pi}=\omega_{\rho}\omega_{\pi_\lambda}\omega_{\sigma}$  \nonumber\\
0, & otherwise \nonumber.
\end{subnumcases}
\makeatletter\let\@alph\@@@alph\makeatother and \makeatletter
\let\@@@alph\@alph
\def\@alph#1{\ifcase#1\or \or $'$\or $''$\fi}\makeatother
\begin{subnumcases}
{F_{\lambda^{\vee}}\omega_{\pi}E_{\lambda}=}
(-1)^{l(\sigma)}q^{l(\rho)-l(\sigma)}F_{\lambda^{\vee}}\omega_{\pi_{\lambda}^{-1}}E_{\lambda}, & if $\omega_{\pi}=\omega_{\sigma}\omega_{\pi_\lambda^{-1}}\omega_{\rho}$  \nonumber\\
0, & otherwise \nonumber.
\end{subnumcases}
\makeatletter\let\@alph\@@@alph\makeatother for some $\rho\in
R(\lambda), \sigma\in R(\lambda^{\vee})$

\end{corollary}
\begin{proof}
By Lemma \ref{lemma-elambda}, if $\pi$ doesn't separate $\lambda$
from $\lambda^{\vee}$, then
$E_{\lambda}\omega_{\pi}F_{\lambda^{\vee}}=0$. Otherwise, there are
some $\rho\in R(\lambda), \sigma\in R(\lambda^{\vee})$, such that $
\omega_{\pi}=\omega_{\rho}\omega_{\pi_\lambda}\omega_{\sigma}.$ Then
\begin{align}
E_{\lambda}\omega_{\pi}F_{\lambda^{\vee}}=E_{\lambda}\omega_{\rho}\omega_{\pi_\lambda}\omega_{\sigma}F_{\lambda^{\vee}}
=\psi(\omega_{\rho})\varphi(\omega_{\sigma})E_{\lambda}\omega_{\pi_{\lambda}}F_{\lambda^{\vee}}=(-1)^{l(\sigma)}q^{l(\rho)-l(\sigma)}E_{\lambda}\omega_{\pi_{\lambda}}F_{\lambda^{\vee}}.
\end{align}

Similarly, if $\pi$ doesn't separate $\lambda^{\vee}$ from
$\lambda$, by Lemma \ref{lemma-elambda},
$F_{\lambda^{\vee}}\omega_{\pi}E_{\lambda}=0$. Otherwise, there are
some $\rho\in R(\lambda), \sigma\in R(\lambda^{\vee})$ such that
$\omega_{\pi}=\omega_{\sigma}\omega_{\pi_\lambda^{-1}}\omega_{\rho}$,
since $\omega_{\pi_{\lambda^{\vee}}}=\omega_{\pi_\lambda^{-1}}$
separate $\lambda^{\vee}$ from $\lambda.$ Then we obtain the second
formula.
\end{proof}

\begin{lemma} \label{Lemma-Flambda}
Given a Young diagram $\lambda$,
\begin{align} \label{formula-Fomega}
F_{\lambda^{\vee}}\omega_{\pi_{\lambda}^{-1}}E_{\lambda}=F_{\lambda^{\vee}}\overline{\omega}_{\pi_{\lambda}^{-1}}E_{\lambda}=F_{\lambda^{\vee}}\omega_{\pi_{\lambda}}^{-1}E_{\lambda}.
\end{align}
\end{lemma}
\begin{proof}
Taking any one of crossings of $\omega_{\pi_{\lambda}^{-1}}$, the
skein relation gives
\begin{align}
\omega_{\pi_{\lambda}^{-1}}=\omega'_{\pi_{\lambda}^{-1}}-(q-q^{-1})L_0,
\end{align}
where $L_0$ arises by smoothing this crossing. Clearly, the
corresponding permutation determined by $L_0$ does not separate
$\lambda^{\vee}$ from $\lambda$. With a slight modification of the
proof of Lemma \ref{lemma-elambda}, we have
$F_{\lambda^{\vee}}L_0E_{\lambda}=0$. Hence
\begin{align}
F_{\lambda^{\vee}}\omega_{\pi_{\lambda}^{-1}}E_{\lambda}=F_{\lambda^{\vee}}\omega'_{\pi_{\lambda}^{-1}}E_{\lambda}.
\end{align}

Then switch the other crossings in $\omega'_{\pi_{\lambda}^{-1}}$,
using the above process inductively, we will obtain the formula
(\ref{formula-Fomega}).

\end{proof}

We now define the idempotent elements $y_{\lambda}\in H_n(q,a)$ for
each partition $\lambda$ as follow. Let $
e_{\lambda}=E_{\lambda}\omega_{\pi_{\lambda}}F_{\lambda^{\vee}}\omega_{\pi_{\lambda}}^{-1},
$ then by using Lemma \ref{lemma-elambda} and Corollary
\ref{corollary-elambda}, it is straightforward to compute that
$e_{\lambda}e_{\mu}=0$ for $\lambda\neq \mu$, and
\begin{align}
e_{\lambda}^2=\alpha_{\lambda}e_{\lambda}
\end{align}
 for some scalar $\alpha_{\lambda}\in R$.  With the help of a formula established in
\cite{Yo97} via $SU(N)$ skein theory,  Aiston \cite{Ai96} computes
that
\begin{align}
\alpha_{\lambda}&=\prod_{(i,j)\in
\lambda}q^{j-i}[\lambda_i+\lambda_j^{\vee}-i-j+1].
\end{align}
Then we define the idempotent $y_{\lambda}$ as follow
\begin{align}
y_{\lambda}=\frac{1}{\alpha_{\lambda}}e_{\lambda}.
\end{align}

Finally, we prove the following proposition which will be used in
the next section.
\begin{proposition} \label{Proposition-ylambda}
For any $\pi\in S_{n}$, we have \makeatletter
\let\@@@alph\@alph
\def\@alph#1{\ifcase#1\or \or $'$\or $''$\fi}\makeatother
\begin{subnumcases}
{y_{\lambda}\omega_{\pi}y_{\lambda}=}
(-1)^{l(\sigma)}q^{l(\pi)-2l(\sigma)}y_{\lambda}, & if $\omega_{\pi}=\omega_{\pi_\lambda}\omega_{\sigma}\omega_{\pi_\lambda^{-1}}\omega_{\rho}$ for some $\rho\in R(\lambda), \sigma\in R(\lambda^{\vee})$ \nonumber\\
0, & otherwise \nonumber.
\end{subnumcases}
\makeatletter\let\@alph\@@@alph\makeatother
\end{proposition}
\begin{proof}
We only need to compute  $e_{\lambda}\omega_{\pi}e_{\lambda}$. By
the second formula in Corollary \ref{corollary-elambda} and Lemma
\ref{Lemma-Flambda}, when
$\omega_{\pi}=\omega_{\pi_\lambda}\omega_{\sigma}\omega_{\pi_\lambda^{-1}}\omega_{\rho}$,
we obtain
\begin{align}
e_{\lambda}\omega_{\pi}e_{\lambda}&=E_{\lambda}\omega_{\pi_{\lambda}}F_{\lambda^{\vee}}\omega_{\pi_{\lambda}}^{-1}\omega_{\pi}
E_{\lambda}\omega_{\pi_{\lambda}}F_{\lambda^{\vee}}\omega_{\pi_{\lambda}}^{-1}\\\nonumber
&=E_{\lambda}\omega_{\pi_{\lambda}}F_{\lambda^{\vee}}\omega_{\sigma}\omega_{\pi_\lambda^{-1}}\omega_{\rho}E_{\lambda}\omega_{\pi_{\lambda}}F_{\lambda^{\vee}}\omega_{\pi_{\lambda}}^{-1}\\\nonumber
&=(-1)^{l(\sigma)}q^{l(\rho)-l(\sigma)}E_{\lambda}\omega_{\pi_{\lambda}}F_{\lambda^{\vee}}\omega_{\pi_\lambda^{-1}}E_{\lambda}\omega_{\pi_{\lambda}}F_{\lambda^{\vee}}\omega_{\pi_{\lambda}}^{-1}\\\nonumber
&=(-1)^{l(\sigma)}q^{l(\rho)-l(\sigma)}E_{\lambda}\omega_{\pi_{\lambda}}F_{\lambda^{\vee}}\omega_{\pi_\lambda}E_{\lambda}\omega_{\pi_{\lambda}}F_{\lambda^{\vee}}\omega_{\pi_{\lambda}}^{-1}\\\nonumber
&=(-1)^{l(\sigma)}q^{l(\rho)-l(\sigma)}e_{\lambda}^2\\\nonumber
&=(-1)^{l(\sigma)}q^{l(\rho)-l(\sigma)}\alpha_{\lambda}e_{\lambda}\\\nonumber
&=(-1)^{l(\sigma)}q^{l(\pi)-2l(\sigma)}\alpha_{\lambda}e_{\lambda}.
\end{align}
\end{proof}

\section{Strong integrality theorem}
\label{Section-strongintegrality}

\subsection{Refined coefficient formula}
In this subsection, we study the resolutions in the skein
$H_{n,p}(q,a)$ carefully, and we obtain a refined formula for the
coefficients appearing in the resolutions. We follow the notations
in \cite{Mo07}. The diagram in a rectangle with $n$ outputs and $m$
inputs at the top, matched at the bottom are called $(n,m)$-tangles.
The resulting skein of the rectangle, denoted by $H_{n,m}(q,a)$, is
an algebra of dimension $(n+m)!$ over the ring $R$, where the
product is induced by placing one tangle above another.

A framed knot $\mathcal{K}$ can be represented as a 1-tangle
$T(\mathcal{K})$ by a single knotted arc as in Figure 7 for
$\mathcal{K}$ is the trefoil knot.

\begin{figure}[!htb]
\begin{center}
\includegraphics[width=100 pt]{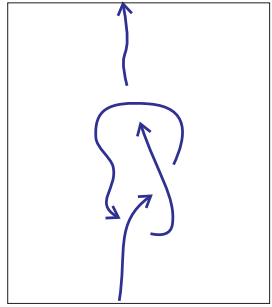}
\end{center}
\caption{The 1-tangle }
\end{figure}

The $(n,m)$-parallel of $T(\mathcal{K})$, denoted by
$T_{n,m}(\mathcal{K})$ in the skein $H_{n,m}(q,a)$ is constructed by
drawing $n+m$ parallel oriented strands to the arc $T(\mathcal{K})$,
with $n$ oriented in one sense and $m$ in the other. Clearly,
$T_{n,m}(\mathcal{K})\in H_{n,m}(q,a)$.

Applying the standard resolution produces to the tangle
$T_{n,m}(\mathcal{K})\in H_{n,m}(q,a)$, we can write
$T_{n,m}(\mathcal{K})$ as a $R$-linear combination of $(n+m)!$
totally descending tangles without closed components. We denote the
set of all such tangles by $S=\{T_{(i)}| i=1,...,(n+m)!\}$. In
particular, for case $m=0$, such tangles are the positive
permutation braids, denoted by $\{\omega_{\pi}: \pi\in S_n\}$, with
strings oriented from top to bottom, while for $n=0$, they are the
positive permutation braids, denoted by $\{\omega^*_{\rho}: \rho\in
S_m\}$, with string orientation from bottom to top.

For each tangle, we can count the number $k$ of its arcs which
connects input and output points at the top. Then $0\leq k\leq
\min(n,m)$. We can divides the set $S$ into two parts
\begin{align}
S=S^{(0)}\cup S^{(1)},
\end{align}
where $S^{(0)}$ consists tangles with $k=0$ in $S$, while  $S^{(1)}$
consists tangles with $k\geq 1$ in $S$. In particular, the tangles
with $k=0$ in $S^{(0)}$ have the form $\omega_{\pi}\otimes
\omega_{\rho}^*$ for some $\pi\in S_n$ and $\rho\in S_m$, where
$\otimes$ denotes juxtaposition of tangles side by side.

We can write $T_{n,m}(\mathcal{K})$ in the skein $H_{n,m}(q,a)$ as
\begin{align} \label{formula-T0T1}
T_{n,m}(\mathcal{K})=T_{n,m}^{(0)}(\mathcal{K})+T_{n,m}^{(1)}(\mathcal{K})
\end{align}
where $T_{n,m}^{(0)}(\mathcal{K})$ is a combination of tangles  in
$S^{(0)}$ and $T_{n,m}^{(1)}(\mathcal{K})$ is a combination of
tangles in $S^{(1)}$.  Then
\begin{align} \label{formula-T0}
T_{n,m}^{(0)}(\mathcal{K})=\sum_{\pi\in S_n,\rho\in
S_m}c_{\pi,\rho}(\mathcal{K})(\omega_\pi\otimes \omega_\rho^*),
\end{align}
and we have
\begin{lemma} \label{Lemma-z}
$c_{\pi,\rho}(\mathcal{K})\in \mathbb{Z}[a^{\pm},z]$ and the powers
of $z$ appearing in the coefficients $c_{\pi,\rho}(\mathcal{K})$ are
simultaneously even (resp. odd) if $l(\pi)+l(\rho)$ is even (resp.
odd).
\end{lemma}

\begin{proof}
Given a $(n,m)$-tangle $T$, let $L(T)$ denote the number of the
closed components of $T$. Denote by $T'$ the tangle after removing
all the closed components of $T$, let $l(T)$ be number of the
components after closing the tangle $T'$. It is clear that
$l(\omega_\pi)=l(c(\pi))$, the length of the partition $c(\pi)$
which denotes the cycle type of the permutation $\pi$.

If we switch or smooth a crossing of sign $\pm 1$ in a tangle $T\in
H_{n,m}(q,a)$, we obtain tangles $T_{\mp }$ and $T_0$ which satisfy
the skein relation
\begin{align}
T=T_{\mp} \pm z T_0.
\end{align}
It is obvious that we have $l(T_{\mp})+L(T_{\mp})=l(T)+L(T)$ and
$l(T_{0})+L(T_{0})=l(T)+L(T)\pm 1$. Then
\begin{align}
z^{l(T)+L(T)} T=z^{l(T_{\mp})+L(T_{\mp})} T_{\mp}\pm z^{\epsilon}
z^{l(T_{0})+L(T_{0})}T_{0},
\end{align}
where $\epsilon$ takes the value $0$ or $2$.

By induction and using Reidemeister moves,  we obtain
\begin{align}
z^{l(T)+L(T)} T=\sum_{\tilde{T}_{(i)}\in
S}\tilde{c}_{(i)}z^{l(\tilde{T}_{(i)})+L(\tilde{T}_{(i)})}\tilde{T}_{(i)}
\end{align}
where $\tilde{c}_{(i)}\in \mathbb{Z}[a^{\pm 1},z^2]$,
$\tilde{T}_{(i)}$ is the tangle $T_{(i)}$ plus $L(\tilde{T}_{(i)})$
null-homotopic closed curves without crossings. Then $
z^{L(\tilde{T}_{(i)})}\tilde{T}_{_{(i)}}=(a-a^{-1})^{L(\tilde{T}_{(i)})}T_{(i)}$
by removing these null-homotopic closed curves.

In conclusion, we have the following expansion for any $T\in
H_{n,m}(q,a)$,
\begin{align}
z^{l(T)+L(T)} T=\sum_{T_{(i)}\in S}c_{(i)}z^{l(T_{(i)})}T_{(i)},
\end{align}
with $c_{(i)}\in \mathbb{Z}[a^{\pm },z^2]$.

In particular, for $T_{n,m}(\mathcal{K})\in H_{n,m}(q,a)$, it easy
to see that $L(T_{n,m}(\mathcal{K}))=0$ and
$l(T_{n,m}(\mathcal{K}))=n+m$,
 we have
\begin{align}
T_{n,m}(\mathcal{K})&=\sum_{T_{(i)}\in
S}c_{(i)}z^{l(T_{(i)})-(n+m)}T_{(i)}.
\end{align}

Since $S=S^{(0)}\cup S^{(1)}$ and $S^{(0)}=\{\omega_{\pi}\otimes
\omega^*_{\rho}: \pi\in S_n, \rho\in S_m\}$, we obtain
\begin{align}
T_{n,m}(\mathcal{K})&=\sum_{\pi\in S_n,\rho\in
S_p}c_{\pi,\rho}z^{l(\omega_{\pi}\otimes\omega_{\rho}^*)-(n+m)}\omega_{\pi}\otimes\omega_{\rho}^*\\\nonumber
&+\sum_{T_{(i)}\in S^{(1)}}c_{(i)}z^{l(T_{(i)})-(n+m)}T_{(i)},
\end{align}
with the coefficients $c_{\pi,\rho},c_{(i)}\in \mathbb{Z}[a^{\pm
1},z^2]$.

Therefore, the coefficients $c_{\pi,\rho}(\mathcal{K})$ in formula
(\ref{formula-T0}) is given by
\begin{align}
c_{\pi,\rho}(\mathcal{K})=c_{\pi,\rho}z^{l(\omega_{\pi}\otimes\omega_{\rho}^*)-(n+m)}
\end{align}

Clearly,
$l(\omega_{\pi}\otimes\omega_{\rho}^*)=l(\omega_{\pi})+l(\omega_\rho^*)=l(c(\pi))+l(c(\rho))$.
By the relation (\ref{formula-lpi}) which will be proved in the
following Lemma \ref{Lemma-symmgroup}, we obtain
\begin{align}
c_{\pi,\rho}(\mathcal{K})=c_{\pi,\rho}z^{l(c(\pi))+l(c(\rho))-(n+p)}=c_{\pi,\rho}z^{l(\pi)+l(\rho)+2k}
\end{align}
for some $k\in \mathbb{Z}$.  Combining $c_{\pi,\rho}\in
\mathbb{Z}[a^{\pm 1},z^2]$, we complete the proof.

\end{proof}

\begin{lemma} \label{Lemma-symmgroup}
Given a permutation $\pi\in S_n$, we have the following identity
\begin{align} \label{formula-lpi}
l(\pi)+l(c(\pi))=n\mod 2
\end{align}
\end{lemma}
\begin{proof}
For every permutation $\pi\in S_n$, its length $l(\pi)$ can be
obtained by calculating the minimal number of the crossings in the
positive braid $\omega_\pi$.  When $n=2$, we have
$S_2=\{\pi_1=(1)(2), \pi_2=(12)\}$. Hence $c(\pi_1)=(11)$, and
$c(\pi_2)=(2)$. It is clear that
\begin{align}
l(\pi_1)+l(c(\pi_1))=l(\pi_2)+l(c(\pi_2))=2.
\end{align}
So Lemma \ref{Lemma-symmgroup} holds when $n=2$.

Now we assume Lemma \ref{Lemma-symmgroup} holds for $n\leq d-1$. As
to $n=d$, we consider a permutation $\pi\in S_{d}$. We first study
the special case, when $\pi$ has the cycle form: $\pi=\pi'(d)$,
where $\pi'$ is a permutation in $S_{d-1}$. It is easy to see that
$l(\pi)=l(\pi')$ and $l(c(\pi))=l(c(\pi'))+1$. By the induction
hypothesis, we have
\begin{align}
l(\pi')+l(c(\pi'))=d-1\mod 2
\end{align}
Thus we get
\begin{align}
l(\pi)+l(c(\pi))=d \mod 2.
\end{align}
Thus Lemma \ref{Lemma-symmgroup} holds for $\pi\in S_d$ with the
cycle form $\pi'(d)$, $\pi'\in S_{d-1}$.

For the general case, we can assume $\pi$ has the cycle form
$\pi=\sigma \tau$, where $\tau$ is the cycle containing the element
$d$ as the form $(i_1\cdots i_j d)$ for $\{i_1,..i_j\}\subset
\{1,..,d-1\}$, $1\leq j\leq d-1$, and $\sigma$ is a cycle in
$S_{d-j-1}$. Hence
\begin{align} \label{formula-lpi1}
l(c(\pi))=l(c(\sigma))+l(c(\tau)).
\end{align}
By the property of permutation, the number of the crossings between
$\omega_\sigma$ and $\omega_\tau$ must be an even number, thus
\begin{align} \label{formula-lpi2}
l(\pi)=l(\sigma)+l(\tau) \mod 2.
\end{align}
Combining (\ref{formula-lpi1}), (\ref{formula-lpi2}) and the
induction hypothesis, we have
\begin{align}
l(\pi)+l(c(\pi))=d \mod 2.
\end{align}
So, we finish the proof of Lemma \ref{Lemma-symmgroup}.
\end{proof}

\begin{lemma} \label{Lemma-a}
The powers of $a$ in the coefficients $c_{\pi,\rho}(\mathcal{K})$
are simultaneously even (resp. odd) if $(n+m)w(\mathcal{K})$ is even
(resp. odd).
\end{lemma}
\begin{proof}
Applying the Reidemeister moves RII and RII moves and the HOMFLY-PT
skein relation to diagram $T_{n,m}(\mathcal{K})$ in $H_{n,m}(q,a)$,
we will eventually obtain the final diagrams in the following forms:
$T_{(i)}^\alpha\otimes U^{\otimes \beta}\otimes P^{\otimes \gamma}$
for $\alpha,\beta,\gamma\in \mathbb{N}$, where $T_{(i)}^{\alpha}$
denotes the tangle $T_{(i)}$ with $\alpha$ positive kinks, $U$
denotes the unknot (i.e. trivial knot), $P$ denotes the unknot with
one positive kink.

We let $S(\mathcal{K})$ be the set of all such final diagrams. Then,
we have

\begin{align}
T_{n,m}(\mathcal{K})=\sum_{T_{(i)}^\alpha\otimes U^{\otimes
\beta}\otimes P^{\otimes \gamma}\in
S(\mathcal{K})}c_{(i),\alpha,\beta,\gamma}(\mathcal{K})T_{(i)}^\alpha\otimes
U^{\otimes \beta}\otimes P^{\otimes \gamma},
\end{align}
and the coefficients $c_{(i),\alpha,\beta,\gamma}(\mathcal{K})\in
\mathbb{Z}[z]$.

For any tangle $T\in H_{n,m}(q,a)$, if we switch or smooth a
crossing of sign $\pm 1$ in $T$, we obtain tangles $T_{\pm 1}$ and
$T_0$. The key observation is that,  we have
\begin{align}
w(T_+)+L(T_+)+l(T_+)&=w(T_-)+L(T_-)+l(T_-) \mod 2\\\nonumber
&=w(T_0)+L(T_0)+l(T_0) \mod 2.
\end{align}

Therefore,
\begin{align} \label{formula-w1}
&w(T_{(i)}^\alpha\otimes U^{\otimes \beta}\otimes P^{\otimes
\gamma})+L(T_{(i)}^\alpha\otimes U^{\otimes \beta}\otimes P^{\otimes
\gamma})+l(T_{(i)}^\alpha\otimes U^{\otimes \beta}\otimes P^{\otimes
\gamma})\\\nonumber
&=w(T_{n,m}(\mathcal{K}))+L(T_{n,m}(\mathcal{K}))+l(T_{n,m}(\mathcal{K}))
\mod 2 \\\nonumber &=(n+m)^2w(\mathcal{K})+0+(n+m)  \mod 2.
\end{align}

In particular, when $T_{(i)}\in S^{(0)}$, we know that
$T_{(i)}=\omega_{\pi}\otimes \omega^*_{\rho}$ for some $\pi\in S_n$
and $\rho\in S_m$. It is easy to get that

\begin{align} \label{formula-w2}
w((\omega_{\pi}\otimes \omega_{\rho}^*)^{\alpha}\otimes U^{\otimes
\beta}\otimes P^{\otimes \gamma})&=l(\pi)+l(\rho)+\alpha+\gamma,
\\\nonumber L((\omega_{\pi}\otimes \omega_{\rho}^*)^{\alpha}\otimes U^{\otimes
\beta}\otimes P^{\otimes \gamma})&=\beta+\gamma,
\\\nonumber
 l((\omega_{\pi}\otimes \omega_{\rho}^*)^{\alpha}\otimes U^{\otimes
\beta}\otimes P^{\otimes \gamma})&=l(c(\pi))+l(c(\rho)).
\end{align}

According to Lemma \ref{Lemma-symmgroup}, we have
\begin{align}
l(\pi)+l(c(\pi))&=n \mod 2 \\\nonumber
 l(\rho)+l(c(\rho))&=p \mod 2
\end{align}

Combining the formulas (\ref{formula-w1}) and (\ref{formula-w2}), we
obtain
\begin{align}
\alpha+\beta\equiv (n+m)^2w(\mathcal{K}) \mod 2.
\end{align}

Finally, applying the Reidemeister move RI to the tangle
$(\omega_{\pi}\otimes \omega_{\rho}^*)^{\alpha}\otimes U^{\otimes
\beta}\otimes P^{\otimes \gamma}$, we obtain
\begin{align}
&(\omega_{\pi}\otimes \omega_{\rho}^*)^{\alpha}\otimes U^{\otimes
\beta}\otimes P^{\otimes \gamma}\\\nonumber
&=a^{\alpha+\gamma}(\omega_{\pi}\otimes \omega_{\rho}^*)\otimes
U^{\otimes \beta}\otimes U^{\otimes \gamma}\\\nonumber
&=\frac{a^{\alpha+\gamma}(a-a^{-1})^{\beta+\gamma}}{z^{\beta+\gamma}}\omega_\pi\otimes
\omega_{\rho}^{*}.
\end{align}

It is obvious that the degrees of $a$ in the term
$a^{\alpha+\gamma}(a-a^{-1})^{\beta+\gamma}$ have the same parity
with $\alpha+\beta$ which is equal to $(n+m)^2w(\mathcal{K})$ mod 2.
Note that the coefficient $c_{\pi,\rho}(\mathcal{K})$ is given by

\begin{align}
c_{\pi,\rho}(\mathcal{K})=c_{(i),\alpha,\beta,\gamma}(\mathcal{K})\frac{a^{\alpha+\gamma}(a-a^{-1})^{\beta+\gamma}}{z^{\beta+\gamma}}
\end{align}

Therefore, the degrees of $a$ appearing in the coefficient
$c_{\pi,\rho}(\mathcal{K})$ is equal to $(n+m)^2w(\mathcal{K})$ mod
2.

\end{proof}

In conclusion, combining Lemma \ref{Lemma-z} and Lemma \ref{Lemma-a}
together, we obtain the following refined coefficient Theorem.
\begin{theorem} \label{Theorem-cpirho}
The coefficient $c_{\pi,\rho}(\mathcal{K})$ given by formula
(\ref{formula-T0}) satisfies the following strong integrality
\begin{align}
c_{\pi,\rho}(\mathcal{K})\in
q^{\epsilon_1}a^{\epsilon_2}\mathbb{Z}[q^{\pm 2},a^{\pm 2}],
\end{align}
where $\epsilon_1,\epsilon_2\in \{0,1\}$ are determined by
\begin{align}
l(\pi)+l(\rho)&\equiv \epsilon_1 \mod 2\\\nonumber
(n+m)w(\mathcal{K})&\equiv \epsilon_2 \mod 2.
\end{align}
\end{theorem}

\subsection{Proof of the strong integrality theorem }
In this subsection, we apply the technique used in \cite{Mo07}
together with Proposition \ref{Proposition-ylambda} and the refined
coefficient Theorem \ref{Theorem-cpirho} to prove the following
strong integrality theorem.
\begin{theorem} \label{Theorem-strongintegrality}
For any framed knot $\mathcal{K}$, the normalized framed full
colored HOMFLY-PT invariant
\begin{align}
\mathcal{P}_{[\lambda,\mu]}(\mathcal{K};q,a)\in
a^{\epsilon}\mathbb{Z}[q^{\pm 2},a^{\pm 2}].
\end{align}
where $\epsilon\in \{0,1\}$  is determined by
\begin{align}
(|\lambda|+|\mu|)w(\mathcal{K})&\equiv \epsilon \mod 2.
\end{align}
\end{theorem}
\begin{proof}
First, it is easy to see that the meridian map $\varphi$ commutes
with the map $T_{\mathcal{K}}$. The eigenvalues of $\varphi$ are
distinct and nonzero, then any eigenvector $Q_{\lambda,\mu}$ of
$\varphi$ is also an eigenvector of $T_{\mathcal{K}}$. Suppose
$t(\lambda,\mu)$ is the eigenvalue of $T_{\mathcal{K}}$ for its
eigenvector $Q_{\lambda,\mu}$,
  By formula (\ref{formula-tkqa}), we only need to prove the strong integrality for
  $t(\lambda,\mu)$.

For $Q_{\lambda}\in \mathcal{C}_{|\lambda|,0}$ and $Q_{\mu}^*\in
\mathcal{C}_{0,|\mu|}$, their product $Q_{\lambda}Q_{\mu}^*\in
\mathcal{C}_{|\lambda|,|\mu|}$.  It is shown in \cite{HaMo06} that $
Q_{\lambda,\mu}=Q_{\lambda}Q_{\mu}^*+W, $ where $W\in
\mathcal{C}_{|\lambda|-1,|\mu|-1}$. Since $
T_{\mathcal{K}}(Q_{\lambda,\mu})=t(\lambda,\mu)Q_{\lambda,\mu}, $
then
\begin{align} \label{formula-TKQ}
T_{\mathcal{K}}(Q_{\lambda}Q_{\mu}^*)=t(\lambda,\mu)Q_{\lambda}Q_{\mu}^*+V
\end{align}
where $V\in \mathcal{C}_{|\lambda|-1,|\mu|-1}$.

One can express $T_{\mathcal{K}}(Q_{\lambda}Q_{\mu}^*)$ as the
closure of the element $(y_{\lambda}\otimes
y_{\mu}^*)T_{n,m}(\mathcal{K})$ in $H_{n,m}(q,a)$.  By formula
(\ref{formula-T0T1}),
\begin{align}
(y_{\lambda}\otimes
y_{\mu}^*)T_{n,m}(\mathcal{K})=(y_{\lambda}\otimes
y_{\mu}^*)T_{n,m}^{(0)}(\mathcal{K})+(y_{\lambda}\otimes
y_{\mu}^*)T_{n,m}^{(1)}(\mathcal{K}).
\end{align}

Clearly, the closure of $(y_{\lambda}\otimes
y_{\mu}^*)T_{n,m}^{(1)}(\mathcal{K})$ denoted by $V'$ belongs to
$\mathcal{C}_{n-1,m-1}$, and
\begin{align}
(y_{\lambda}\otimes
y_{\mu}^*)T_{n,m}^{(0)}(\mathcal{K})=\sum_{\pi\in S_n, \rho\in
S_m}c_{\pi,\rho}(\mathcal{K})(y_\lambda \omega_{\pi}\otimes
y_{\mu}^*\omega_{\rho}^*).
\end{align}

 Clearly, the closure of
$y_\lambda \omega_{\pi}$ and $y_{\mu}^*\omega_{\rho}^*$ are equal to
$c(\pi,\lambda)Q_{\lambda}$ and $c(\rho,\mu)Q_{\mu}^*$.  By
Proposition \ref{Proposition-ylambda}, we have
\begin{align}
c(\pi,\lambda)&\in q^{l(\pi)}\mathbb{Z}[q^2], \label{formula-cpilambda} \\
c(\rho,\mu)&\in q^{l(\rho)}\mathbb{Z}[q^2].  \label{formula-crhomu}
\end{align}

Therefore, the closure of $(y_{\lambda}\otimes
y_{\mu}^*)T_{n,m}^{(0)}(\mathcal{K})$ is equal to
$C(\lambda,\mu)Q_{\lambda}Q_{\mu}^*$, with
\begin{align}
C(\lambda,\mu)=\sum_{\pi \in S_n,\rho\in
S_m}c_{\pi,\rho}(\mathcal{K})c(\pi,\lambda)c(\rho,\mu).
\end{align}
Hence we have
\begin{align}
T_{\mathcal{K}}(Q_{\lambda}Q_{\mu}^*)=C(\lambda,\mu)Q_{\lambda}Q_{\mu}^*+V',
\end{align}
where $V'\in \mathcal{C}_{n-1,m-1}$.

Therefore, comparing with formula (\ref{formula-TKQ}),
$C(\lambda,\mu)=t(\lambda,\mu)$ is the normalized framed full
colored HOMFLY-PT invariant. Applying Theorem \ref{Theorem-cpirho}
for the coefficients $c_{\pi,\rho}(\mathcal{K})$ and formulas
(\ref{formula-cpilambda}) and (\ref{formula-crhomu}), we obtain
\begin{align}
C(\lambda,\mu)\in a^{\epsilon}\mathbb{Z}[q^{\pm 2},a^{\pm 2}].
\end{align}
where $\epsilon\in \{0,1\}$ is determined by
\begin{align}
(|\lambda|+|\mu|)w(\mathcal{K})&\equiv \epsilon \mod 2.
\end{align}
\end{proof}

\begin{remark}
In \cite{Le00}, T. Le proved the strong integrality for the
normalized quantum group invariant of a knot $\mathcal{K}$ colored
by an irreducible module over any simple Lie algebra $\mathfrak{g}$
via quantum group theory.  It is easy to see Theorem
\ref{Theorem-strongintegrality} implies Le's strong integrality when
$\mathfrak{g}$ is the Lie algebra $sl_{N}\mathbb{C}$.
\end{remark}

\begin{corollary} For any framed knot $\mathcal{K}$, we have
\begin{align}
\mathcal{H}_{[\lambda,\mu]}(\mathcal{K};-q,a)&=(-1)^{(|\lambda|+|\mu|)}\mathcal{H}_{[\lambda,\mu]}(\mathcal{K};q,a)\label{formula-qsym}\\
\mathcal{H}_{[\lambda,\mu]}(\mathcal{K};q,-a)&=(-1)^{(|\lambda|+|\mu|)(w(\mathcal{K})+1)}\mathcal{H}_{[\lambda,\mu]}(\mathcal{K};q,a).
\end{align}
\end{corollary}
\begin{proof}
Using the expression for the colored HOMFLY-PT invariant of unknot
(\ref{formula-unknot1}), it is straightforward to obtain
\begin{align}
W_{\lambda}(U;q,-a)&=(-1)^{|\lambda|}W_{\lambda}(U;q,a), \\
W_{\lambda}(U;-q,a)&=(-1)^{|\lambda|}W_{\lambda}(U;q,a),
\end{align}
where in the second identity, we need to use the following identity
for the partition $\lambda$:
\begin{align}
(-1)^{|\lambda|}=(-1)^{\sum_{x\in \lambda}(hl(x)+cn(x))}.
\end{align}

Then, by formula (\ref{formula-unknot2}),
\begin{align}
\mathcal{H}_{[\lambda,\mu]}(U;q,-a)&=(-1)^{|\lambda|+|\mu|}\mathcal{H}_{[\lambda,\mu]}(U;q,a), \\
\mathcal{H}_{[\lambda,\mu]}(U;-q,a)&=(-1)^{|\lambda|+|\mu|}\mathcal{H}_{[\lambda,\mu]}(U;q,a).
\end{align}

Finally, combining Theorem \ref{Theorem-strongintegrality} and
formula (\ref{formula-Hlambdamu}) together, we complete the proof.
\end{proof}

\begin{corollary}
For any framed knot $\mathcal{K}$, we have
\begin{align}
\mathcal{H}_{[\lambda,\mu]}(\mathcal{K};q^{-1},a)=(-1)^{(|\lambda|+|\mu|)}\mathcal{H}_{[\lambda^\vee,\mu^\vee]}(\mathcal{K};q,a)
\end{align}
\end{corollary}
\begin{proof}
By using the formula (\ref{formula-Qdet})  for $Q_{\lambda,\mu}$,
\begin{align}
Q_{\lambda,\mu}|_{q\rightarrow -q^{-1}}=Q_{\lambda^\vee,\mu^\vee},
\end{align}
It follows that
\begin{align}
\mathcal{H}_{[\lambda,\mu]}(\mathcal{K};-q^{-1},a)=\mathcal{H}_{[\lambda^\vee,\mu^\vee]}(\mathcal{K};q,a).
\end{align}
Then, combining with formula (\ref{formula-qsym}), we complete the
proof.
\end{proof}

All the above formulas can be generalized to the case of link.
\begin{theorem} \label{Theorem-linksym}
For any framed link $\mathcal{L}$, we have
\begin{align}
\mathcal{H}_{[\vec{\lambda},\vec{\mu}]}(\mathcal{L};-q,a)&=(-1)^{\sum_{\alpha=1}^L(|\lambda^{\alpha}|+|\mu^{\alpha}|)}
\mathcal{H}_{[\vec{\lambda},\vec{\mu}]}(\mathcal{L};q,a)\\
\mathcal{H}_{[\vec{\lambda},\vec{\mu}]}(\mathcal{L};q,-a)&=(-1)^{\sum_{\alpha=1}^L(|\lambda^{\alpha}|+|\mu^{\alpha}|)(w(\mathcal{K}_{\alpha})+1)}
\mathcal{H}_{[\vec{\lambda},\vec{\mu}]}(\mathcal{L};q,a)\\
\mathcal{H}_{[\vec{\lambda},\vec{\mu}]}(\mathcal{L};q^{-1},a)&=(-1)^{\sum_{\alpha=1}^L(|\lambda^{\alpha}|+|\mu^{\alpha}|)}
\mathcal{H}_{[\vec{\lambda}^\vee,\vec{\mu}^\vee]}(\mathcal{L};q,a).
\end{align}
\end{theorem}

Suppose $\mathcal{L}$ is a link with $L$ components
$\mathcal{K}_1,...,\mathcal{K}_L$, we denote it by
$\mathcal{L}=\mathcal{K}_1\bigvee \mathcal{K}_2\bigvee \cdots
\bigvee \mathcal{K}_L$. We cut the component $\mathcal{K}_\alpha$
open and get a 1-tangle, and we can draw $\mathcal{L}$ in the
annulus as the closure of this 1-tangle. Decorating
$\mathcal{K}_\alpha$ with a diagram $Q_{\alpha}$ gives a diagram
$\mathcal{K}_1\bigvee \cdots\bigvee \mathcal{K}_\alpha \star
Q_{\alpha}\bigvee \cdots \bigvee \mathcal{K}_L$ in $\mathcal{C}$, it
induces a linear map $T_{\mathcal{K}_\alpha}^{\mathcal{L}}:
\mathcal{C}\rightarrow \mathcal{C}$.

Similarly, the eigenvectors of
$T_{\mathcal{K}_\alpha}^{\mathcal{L}}$ are given by
$Q_{\lambda,\mu}$, and we denote the eigenvalue of
$T_{\mathcal{K}_\alpha}^{\mathcal{L}}$ corresponding to eigenvector
$Q_{\lambda,\mu}$ by
$t_{\mathcal{K}_\alpha}^{\mathcal{L}}(\lambda,\mu)$.

Then the identity
\begin{align}
\mathcal{K}_1\bigvee \cdots\bigvee \mathcal{K}_\alpha \star
Q_{\lambda,\mu}\bigvee \cdots \bigvee
\mathcal{K}_L&=T_{\mathcal{K}_\alpha}^{\mathcal{L}}(Q_{\lambda,\mu})\\\nonumber
&=t_{\mathcal{K}_\alpha}^{\mathcal{L}}(\lambda,\mu)Q_{\lambda,\mu}\\\nonumber
&=t_{\mathcal{K}_\alpha}^{\mathcal{L}} (\lambda,\mu)U\star
Q_{\lambda,\mu}
\end{align}
implies that
\begin{align}
t_{\mathcal{K}_\alpha}^{\mathcal{L}}(\lambda,\mu)=\frac{\mathcal{H}(\mathcal{K}_1\bigvee
\cdots\bigvee \mathcal{K}_\alpha \star Q_{\lambda,\mu}\bigvee \cdots
\bigvee \mathcal{K}_L;q,a)}{\mathcal{H}(U\star
Q_{\lambda,\mu};q,a)}.
\end{align}

With a slight modification of the proof of Theorem
\ref{Theorem-strongintegrality},  one can show that
\begin{proposition}
Under the above setting, we have
\begin{align} \label{formula-tLK}
t_{\mathcal{K}_\alpha}^{\mathcal{L}}(\lambda,\mu)\in
a^{\epsilon}\mathbb{Z}[q^{\pm 2},a^{\pm 2}].
\end{align}
where $\epsilon\in \{0,1\}$ is determined by
\begin{align}
(|\lambda|+|\mu|)w(\mathcal{K}_{\alpha})+\sum_{\beta\neq
\alpha}w(\mathcal{K}_\beta)+(L-1)&\equiv \epsilon \mod 2.
\end{align}
\end{proposition}

By ormulas (\ref{formula-Qlambdamu}) and (\ref{formula-frobeniusQ}),
we obtain
\begin{align}
Q_{\lambda,\mu}&=\sum_{\sigma,\rho,\nu}(-1)^{|\sigma|}c_{\sigma,\rho}^{\lambda}c_{\sigma^t,\nu}^{\mu}Q_{\rho}
Q_{\nu}^*\\\nonumber
&=\sum_{\sigma,\rho,\nu,\tau,\delta}(-1)^{|\sigma|}c_{\sigma,\rho}^{\lambda}c_{\sigma^t,\nu}^{\mu}\frac{\chi_{\rho}(\tau)}{\mathfrak{z}_\tau}
\frac{\chi_{\nu}(\delta)}{\mathfrak{z}_\delta}\frac{1}{\{\tau\}\{\delta\}}X_{\tau}X_{\delta}^*.
\end{align}
Therefore,
\begin{align}
\mathcal{\mathcal{L}}\star \otimes_{\beta=1}^{L}
Q_{\lambda^\alpha,\mu^\alpha}& =\mathcal{K}_1\star
Q_{\lambda^1,\mu^1}\bigvee \cdots\bigvee \mathcal{K}_\alpha \star
Q_{\lambda^\alpha,\mu^\alpha}\bigvee \cdots \bigvee
\mathcal{K}_L\star Q_{\lambda^L,\mu^L}\\\nonumber
&=\sum_{\tau^\beta,\delta^\beta, \beta\neq
\alpha}C_{\tau^1,\delta^1,..,\hat{\tau^\alpha},\hat{\delta^\alpha},..,\tau^L,\delta^L}
\prod_{\beta=1,\beta\neq
\alpha}\frac{1}{\{\tau^\beta\}\{\delta^\beta\}}\\\nonumber
&\cdot\mathcal{K}_1\star X_{\tau^1}X^*_{\delta^1}\bigvee
\cdots\bigvee \mathcal{K}_\alpha \star
Q_{\lambda^\alpha,\mu^\alpha}\bigvee \cdots \bigvee
\mathcal{K}_L\star X_{\tau^L}X^{*}_{\delta^L},
\end{align}
where
\begin{align}
C_{\tau^1,\delta^1,..,\hat{\tau^\alpha},\hat{\delta^\alpha},..,\tau^L,\delta^L}=
\sum_{\sigma^\beta,\rho^\beta,\nu^\beta,\tau^\beta,\delta^\beta,\beta\neq
\alpha}\prod_{\beta,\beta\neq
\alpha}(-1)^{|\sigma^\beta|}c_{\sigma^\beta,\rho^\beta}^{\lambda^\beta}c_{(\sigma^\beta)^t,\nu^\beta}^{\mu^\beta}
\frac{\chi_{\rho^\beta}(\tau^\beta)}{\mathfrak{z}_{\tau^\beta}}
\frac{\chi_{\nu^\beta}(\delta^\beta)}{\mathfrak{z}_{\delta^\beta}},
\end{align}
and $\hat{\tau^\alpha}, \hat{\delta^\alpha}$ denote the indexes
$\tau^\alpha, \delta^\alpha$ do not appear in the summation.

We denote the link
$\mathcal{L}_{\tau^1,\delta^1,..,\hat{\tau^\alpha},\hat{\delta^\alpha},..,\tau^L,\delta^L}=\mathcal{K}_1\star
X_{\tau^1}X^*_{\delta^1}\bigvee \cdots\bigvee \mathcal{K}_\alpha
\bigvee \cdots \bigvee \mathcal{K}_L\star
X_{\tau^L}X^{*}_{\delta^L}$. By formula (\ref{formula-tLK}) and a
careful computations of the writhe numbers for this link, we obtain
\begin{align}
t_{\mathcal{K}_\alpha}^{
\mathcal{L}_{\tau^1,\delta^1,..,\hat{\tau^\alpha},\hat{\delta^\alpha},..,\tau^L,\delta^L}}(\lambda^\alpha,\mu^\alpha)\in
a^{\epsilon}\mathbb{Z}[q^{\pm 2},a^{\pm 2}],
\end{align}
where $\epsilon\in \{0,1\}$ is determined by the following formula
\begin{align}
(|\lambda^{\alpha}|+|\mu^{\alpha}|)w(\mathcal{K}_\alpha)+\sum_{\beta\neq
\alpha}(|\lambda^{\beta}|+|\mu^{\beta}|)(w(\mathcal{K}_\beta)+1)\equiv
\epsilon \mod 2.
\end{align}

Therefore,
\begin{align} \label{formula-HLlamdamu}
\mathcal{H}(\mathcal{\mathcal{L}}\star \otimes_{\beta=1}^{L}
Q_{\lambda^\alpha,\mu^\alpha};q,a)&=\sum_{\tau^\beta,\delta^\beta,
\beta\neq
\alpha}C_{\tau^1,\delta^1,..,\hat{\tau^\alpha},\hat{\delta^\alpha},..,\tau^L,\delta^L}
\prod_{\beta=1,\beta\neq
\alpha}\frac{1}{\{\tau^\beta\}\{\delta^\beta\}}\\\nonumber &\cdot
t_{\mathcal{K}_\alpha}^{
\mathcal{L}_{\tau^1,\delta^1,..,\hat{\tau^\alpha},\hat{\delta^\alpha},..,\tau^L,\delta^L}}(\lambda^\alpha,\mu^\alpha)
\mathcal{H}(U\star Q_{\lambda^\alpha,\mu^\alpha};q,a).
\end{align}

Together with the formulas
\begin{align}
\mathcal{H}(U\star
Q_{\lambda^\alpha,\mu^\alpha};q,-a)&=(-1)^{|\lambda^{\alpha}|+|\mu^{\alpha}|}\mathcal{H}(U\star
Q_{\lambda^\alpha,\mu^\alpha};q,a), \\
\mathcal{H}(U\star
Q_{\lambda^\alpha,\mu^\alpha};-q,a)&=(-1)^{|\lambda^{\alpha}|+|\mu^{\alpha}|}\mathcal{H}(U\star
Q_{\lambda^\alpha,\mu^\alpha};q,a), \\
\mathcal{H}(U\star
Q_{\lambda^\alpha,\mu^\alpha};q,-a)&=(-1)^{|\lambda^{\alpha}|+|\mu^{\alpha}|}\mathcal{H}(U\star
Q_{(\lambda^\alpha)^\vee,(\mu^\alpha)^\vee};q,a).
\end{align}
we finish  the proof of  Theorem \ref{Theorem-linksym}.

\section{Refined LMOV integrality structure}  \label{section-LMOV functions}
\subsection{Symmetric functions and plethysms}
Recall that the power sum symmetric function of infinite variables
$\mathbf{x}=(x_1,..,x_N,..)$ is defined by
$p_{n}(\mathbf{x})=\sum_{i}x_i^n. $ We refer to Section 8.1 for more
detailed definitions about the symmetric functions. For a partition
$\lambda=(\lambda_1,...,\lambda_l)$, we define
$p_\lambda(\mathbf{x})=\prod_{j=1}^{l}p_{\lambda_j}(\mathbf{x}). $
The Schur function $s_{\lambda}(\mathbf{x})$ is determined by the
Frobenius formula
\begin{align}  \label{Frobeniusformula}
s_\lambda(\mathbf{x})=\sum_{\mu}\frac{\chi_{\lambda}(\mu)}{\mathfrak{z}_\mu}p_\mu(\mathbf{x}),
\end{align}

We let $\Lambda(\mathbf{x})$ be the ring of symmetric functions of
$\mathbf{x}=(x_1,x_2,...)$ over the ring $\mathbb{Q}(q,a)$, and let
$\langle \cdot,\cdot \rangle$ be the Hall pair on
$\Lambda(\mathbf{x})$ determined by
\begin{align}
\langle s_\lambda(\mathbf{x}),s_{\mu}(\mathbf{x})
\rangle=\delta_{\lambda,\mu}.
\end{align}
For $\vec{\mathbf{x}}=(\mathbf{x}^1,...,\mathbf{x}^L)$, denote by
$\Lambda(\vec{\mathbf{x}}):=\Lambda(\mathbf{x}^1)\otimes_{\mathbb{Z}}\cdots
\otimes_{\mathbb{Z}}\Lambda(\mathbf{x}^L)$ the ring of functions
separately symmetric in $\mathbf{x}^1,...,\mathbf{x}^L$, where
$\mathbf{x}^i=(x^{i}_{1},x^{i}_{2},...)$. We will study functions in
the ring $\Lambda(\vec{\mathbf{x}})$. For
$\vec{\mu}=(\mu^1,...,\mu^L)\in \mathcal{P}^{L}$, we let
$a_{\vec{\mu}}(\vec{\mathbf{x}})=a_{\mu^1}(\mathbf{x}^1)\cdots
a_{\mu^L}(\mathbf{x}^L)\in\Lambda(\vec{\mathbf{x}})$ be homogeneous
of degree $(|\mu^1|,..,|\mu^L|)$. Moreover, the Hall pair on
$\Lambda(\vec{\mathbf{x}})$ is given by
\begin{align}
\langle a_1(\mathbf{x}^1)\cdots a_L(\mathbf{x}^L),
b_1(\mathbf{x}^1)\cdots b_L(\mathbf{x}^L)\rangle=\langle
a_1(\mathbf{x}^1),b_1(\mathbf{x}^1) \rangle \cdots \langle
a_L(\mathbf{x}^n),b_L(\mathbf{x}^L) \rangle
\end{align}
 for
$a_1(\mathbf{x}^1)\cdots a_L(\mathbf{x}^L), b_1(\mathbf{x}^1)\cdots
b_L(\mathbf{x}^L)\in \Lambda(\vec{\mathbf{x}})$.
For $d\in \mathbb{Z}_+$, we define the $d$-th Adams operator
$\Psi_d$ as the $\mathbb{Q}$-algebra map on
$\Lambda(\vec{\mathbf{x}})$
\begin{align}
\Psi_{d}(f(\vec{\mathbf{x}};q,a))= f(\vec{\mathbf{x}}^d;q^d,a^d).
\end{align}
Denote by $\Lambda(\vec{\mathbf{x}})^+$ the set of  symmetric
functions with degree $\geq 1$. The plethystic exponential Exp and
logarithmic Log are inverse maps
\begin{align}
\text{Exp}: \Lambda(\vec{\mathbf{x}})^+\rightarrow
1+\Lambda(\vec{\mathbf{x}})^+, \ \text{Log}:1+
\Lambda(\vec{\mathbf{x}})^+\rightarrow \Lambda(\vec{\mathbf{x}})^+
\end{align}
respectively defined by (see \cite{HLRV11})
\begin{align}
\text{Exp}(f)=\exp\left(\sum_{d\geq 1}\frac{\Psi_d(f)}{d}\right), \
\text{Log}(f)=\sum_{d\geq 1}\frac{\mu(d)}{d}\Psi_d(\log(f)),
\end{align}
where $\mu$ is the M\"{o}bius function.  It is clear that
\begin{align}
\text{Exp}(f+g)=\text{Exp}(f)\text{Exp}(g), \
\text{Log}(fg)=\text{Log}(f)+\text{Log}(g),
\end{align}
and Exp$(x)=\frac{1}{1-x}$, if we use the expansion
$\log(1-x)=-\sum_{d\geq 1}\frac{x^d}{d}$.

\subsection{LMOV functions}
Consider a series of functions $S_{\vec{\lambda}}(q,a)\in
\mathbb{Q}(q,a)$, where $\vec{\lambda}\in \mathcal{P}^L$, we
introduce the partition function for
$\{S_{\vec{\lambda}}(q,a)|\vec{\lambda}\in \mathcal{P}^L\}$ which is
the following generating function
\begin{align}
Z(\vec{\mathbf{x}};q,a)=\sum_{\vec{\lambda}\in
\mathcal{P}^L}S_{\vec{\lambda}}(q,a)s_{\vec{\lambda}}(\vec{\mathbf{x}}).
\end{align}

\begin{definition}
The {\em LMOV function} for the series
$\{S_{\vec{\lambda}}(q,a)|\vec{\lambda}\in \mathcal{P}^L\}$ is given
by
\begin{align}
f_{\vec{\lambda}}(q,a)=\langle
\text{Log}(Z(\vec{\mathbf{x}};q,a)),s_{\vec{\lambda}}(\vec{\mathbf{x}})
\rangle
\end{align}
where $\langle \cdot, \cdot\rangle$ denotes the Hall pair in the
ring of symmetric functions $\Lambda(\vec{\mathbf{x}})$.
\end{definition}

In particular, if we take $\mathbf{x}^\alpha=(x_\alpha,0,0,...)$ for
all $\alpha=1,...,L$, i.e.
$$\vec{\mathbf{x}}=((x_{1},0,..),(x_{2},0,..),...,(x_L,0,..)).$$

In this special case, we obtain the following special partition
function for $\{S_{r_1,...,r_L}(q,a)\}$ which is given by
\begin{align}
\mathcal{Z}(x_{1},x_{2},...,x_L;q,a)=\sum_{r_1,...,r_L\geq
0}S_{r_1,...,r_L}(q,a)x_1^{r_1}\cdots x_L^{r_{L}}.
\end{align}

\begin{definition}
The {\em special LMOV function} for the series
$\{S_{r_1,...,r_L}(q,a)\}$ is
\begin{align}
\mathfrak{f}_{n_1,...,n_L}(q,a) =[x_1^{n_1}\cdots
x_L^{n_L}]\text{Log} \mathcal{Z}(x_{1},x_{2},...,x_L;q,a),
\end{align}
where the notation $[x_1^{n_1}\cdots x_L^{n_L}]f(x_1,...,x_L)$
denotes the coefficients of $x_1^{n_1}\cdots x_L^{n_L}$ in the
function $f(x_1,...,x_L)$.
\end{definition}

We introduce the notion of $F$-invariants $F_{\vec{\mu}}(q,a)$
(reps. $\mathcal{F}_{r_1,...,r_L}(q,a)$) for
$\{S_{\vec{\lambda}}(q,a)\}$ (resp. $\{S_{r_1,...,r_L}(q,a)\}$)
which is determined by the formula
\begin{align}
\log
Z(\vec{\mathbf{x}};q,a)=\sum_{\vec{\mu}}F_{\vec{\mu}}(q,a)p_{\vec{\mu}}(\vec{\bf{x}})
\end{align}
(resp. $\log\mathcal{Z}(x_{1},x_{2},...,x_L;q,
a)=\sum_{r_1,...,r_L\geq 0}\mathcal{F}_{r_1,...,r_L}(q,
a)x_1^{r_1}\cdots x_L^{r_L}$).

Then, by the formula (\ref{formula-logZ}) in the Appendix Section
\ref{section-Appendix}, we obtain
\begin{align} \label{formula-logZx}
F_{\vec{\mu}}(q,a)=\sum_{\Lambda\in \mathcal{P}(\mathcal{P}^L),
|\Lambda|=\vec{\mu}\in \mathcal{P}^L}\Theta_{\Lambda}Z_{\Lambda}.
\end{align}
and
\begin{align}
\mathcal{F}_{r_1,...,r_L}(q,a)=\sum_{|\Lambda|=(r_1,..,r_L)}\Theta_{\Lambda}S_{\Lambda}(q,a).
\end{align}
By a straightforward computation, we obtain
\begin{align} \label{formula-flambda}
f_{\vec{\lambda}}(q,a)&=\langle
Log(Z),s_{\vec{\lambda}}(\vec{\mathbf{x}}) \rangle\\\nonumber
&=\langle \sum_{d\geq 1}\frac{\mu(d)}{d}\circ
\Psi_d(\sum_{\vec{\mu}}F_{\vec{\mu}}(q,a)p_{\vec{\mu}}(\mathbf{x})),s_{\vec{\lambda}}(\vec{\mathbf{x}})
\rangle\\\nonumber &=\langle \sum_{d\geq
1}\frac{\mu(d)}{d}\sum_{\vec{\mu}}F_{\vec{\mu}}(q^d,a^d)p_{d\vec{\mu}}(\vec{\mathbf{x}}),s_{\vec{\lambda}}(\vec{\mathbf{x}})
\rangle\\\nonumber &=\langle \sum_{d\geq
1}\frac{\mu(d)}{d}\sum_{\vec{\mu}}F_{\vec{\mu}}(q^d,a^d)\sum_{\vec{\nu}}\chi_{\vec{\nu}}(d\cdot\vec{\mu})
s_{\vec{\nu}}(\vec{\mathbf{x}}),s_{\vec{\lambda}}(\vec{\mathbf{x}})\rangle\\\nonumber
&=\sum_{\vec{\mu}}\chi_{\vec{\lambda}}(\vec{\mu})\sum_{d|\vec{\mu}}\frac{\mu(d)}{d}F_{\vec{\mu}/d}(q^d,a^d),
\end{align}
and
\begin{align}
\mathfrak{f}_{n_1,...,n_L}(q,\vec{a})=\sum_{d\geq
1}\frac{\mu(d)}{d}\mathcal{F}_{\vec{n}/d}(q^d,a^d).
\end{align}

\begin{remark}
Here we  mention that when
$\vec{\lambda}=(\lambda^{1},...,\lambda^L)=((n_1),...,(n_L))$,
according to their definitions,
$
 f_{n_1,...,n_L}(q,a)\neq
 \mathfrak{f}_{n_1,...,n_L}(q,a).
$ That's why we use the different symbols there.
\end{remark}

\begin{definition}
We define the {\em reformulated LMOV function} for
$\{S_{\vec{\lambda}}(q,a)\}$ as a character transformation of the
LMOV functions,
\begin{align} \label{formula-gmu}
g_{\vec{\mu}}(q,a)=\frac{1}{\{\vec{\mu}\}}\sum_{\vec{\lambda}}f_{\vec{\lambda}}(q,a)
\chi_{\vec{\lambda}}(\vec{\mu}).
\end{align}
\end{definition}

Plugging the formula (\ref{formula-flambda}) into \ref{formula-gmu},
we obtain
\begin{align}
g_{\vec{\mu}}(q,\vec{a})=\frac{\mathfrak{z}_{\vec{\mu}}}{\{\vec{\mu}\}}\sum_{d|\vec{\mu}}\frac{\mu(d)}{d}F_{\vec{\mu}/d}(q^d,\vec{a}^d).
\end{align}

\begin{remark}
The original definition of the reformulated LMOV function
\cite{LaMaVa00,LaMa02,LP10} is slightly different from here. They
introduce the  formula $
T_{\lambda\mu}(\mathbf{x})=\sum_{\nu}\frac{\chi_{\lambda}(\nu)\chi_{\mu}(\nu)}{\mathfrak{z}_\nu}p_{\nu}(\mathbf{x}).
$ In particularly, $
T_{\lambda\mu}(q^\rho)=\sum_{\mu}\frac{\chi_{\lambda}(\nu)\chi_{\mu}(\nu)}{\mathfrak{z}_\nu}\prod_{i=1}^{l(\nu)}\frac{1}{\{\nu_i\}},
$ if one lets $q^\rho=(q^{-1},q^{-3},q^{-5},...)$. In
\cite{LaMaVa00,LaMa02,LP10}, the reformulated LMOV function for
$\{S_{\vec{\lambda}}(q,a)\}$ is defined as $
\hat{f}_{\vec{\mu}}(q,a)=\sum_{\vec{\lambda}}f_{\vec{\lambda}}(q,a)T_{\vec{\lambda}\vec{\mu}}(q^\rho).
$ Therefore, we have
\begin{align} \label{formula-hatfmu}
\hat{f}_{\vec{\lambda}}(q,a)=\chi_{\vec{\lambda}}(\vec{\mu})\sum_{\vec{\mu}}\frac{1}{\{\vec{\mu}\}}
\sum_{d|\vec{\mu}}\frac{\mu(d)}{d}F_{\vec{\mu}/d}(q^d,a^d).
\end{align}
It is obvious that two definitions by formulas (\ref{formula-gmu})
and (\ref{formula-hatfmu}) are related by a character transformation
\begin{align}
g_{\vec{\mu}}(q,a)=\sum_{\vec{\lambda}}\chi_{\vec{\lambda}}(\vec{\mu})\hat{f}_{\vec{\lambda}}(q,a).
\end{align}
\end{remark}

\subsection{Refined LMOV conjecture}
Given a framed link $\mathcal{L}$ with $L$ components
$\mathcal{K}_1,...,\mathcal{K}_L$. Denote by
$\vec{\tau}=(\tau^1,...,\tau^L)\in \mathbb{Z}^L$ the framing of
$\mathcal{L}$, i.e. $\tau^\alpha=w(\mathcal{K}_{\alpha})$ for
$\alpha=1,...,L$. In the following, we use the notation
$\mathcal{L}_{\vec{\tau}}$ to denote this framed link $\mathcal{L}$
if we want to emphasize its framing.

For $ \vec{\lambda}=(\lambda^1,...,\lambda^L),
\vec{\mu}=(\mu^1,...,\mu^L), \vec{\nu}=(\nu^1,...,\nu^L) \in
\mathcal{P}^L$, set
$c_{\vec{\lambda},\vec{\mu}}^{\vec{\nu}}=\prod_{\alpha=1}^Lc_{\lambda^\alpha,\mu^\alpha}^{\nu^\alpha}$,
where $c_{\lambda^\alpha,\mu^\alpha}^{\nu^\alpha}$ is the
Littlewood-Richardson coefficient.

\begin{definition}
The {\em $\tau$-framed full colored HOMFLY-PT invariants} for the
framed link $\mathcal{L}_{\vec{\tau}}$ is given by
\begin{align}
H_{\vec{\lambda}}(\mathcal{L}_{\vec{\tau}};q,a)=(-1)^{\sum_{\alpha=1}^L|\lambda^\alpha|\tau^\alpha}
a^{-\sum_{\alpha=1}^L|\lambda^\alpha|\tau^\alpha}\mathcal{H}_{\vec{\lambda}}(\mathcal{L}_{\vec{\tau}};q,a).
\end{align}
\end{definition}

\begin{definition}
The {\em $\tau$-framed composite invariant} for
$\mathcal{L}_{\vec{\tau}}$ is given as follow
\begin{align}
C_{\vec{\nu}}(\mathcal{L}_{\vec{\tau}};q,a)=(-1)^{\sum_{\alpha=1}^L|\lambda^\alpha|\tau^\alpha}\sum_{\vec{\lambda},\vec{\mu}}
a^{-\sum_{\alpha=1}^L|\nu^\alpha|\tau^\alpha}
c_{\vec{\lambda},\vec{\mu}}^{\vec{\nu}}
\mathcal{H}_{[\vec{\lambda},\vec{\mu}]}(\mathcal{L}_{\vec{\tau}};q,a).
\end{align}
\end{definition}
By using Theorem \ref{Theorem-linksym}, we obtain
\begin{align} \label{formula-Hlambda}
H_{\vec{\lambda}}(\mathcal{L}_{\vec{\tau}};q,-a)&=(-1)^{|\vec{\lambda}|}H_{\vec{\lambda}}(\mathcal{L}_{\vec{\tau}};q,a),
\\\nonumber
H_{\vec{\lambda}}(\mathcal{L}_{\vec{\tau}};-q,a)&=(-1)^{|\vec{\lambda}|}H_{\vec{\lambda}}(\mathcal{L}_{\vec{\tau}};q,a),
\\\nonumber
H_{\vec{\lambda}}(\mathcal{L}_{\vec{\tau}};q^{-1},a)&=(-1)^{|\vec{\lambda}|}H_{\vec{\lambda^{t}}}(\mathcal{L}_{\vec{\tau}};q,a),
\end{align}
and which hold similarly for the framed composite invariant
$C_{\vec{\nu}}(\mathcal{L}_{\vec{\tau}};q,a)$.

We denote the reformulated LMOV functions for $\tau$-framed colored
HOMFLY-PT invariants
$\{H_{\vec{\lambda}}(\mathcal{L}_{\vec{\tau}};q,a)\}$ and
$\tau$-framed composite invariants
$\{C_{\vec{\nu}}(\mathcal{L}_{\vec{\tau}};q,a)\}$ by
$g_{\vec{\mu}}^{(0)}(\mathcal{L}_{\vec{\tau}};q,a)$ and
$g_{\vec{\mu}}^{(1)}(\mathcal{L}_{\vec{\tau}};q,a)$ respectively.

Similarly, we write
$\mathfrak{f}^{(0)}_{\vec{n}}(\mathcal{L}_{\vec{\tau}};q,a)$ and
$\mathfrak{f}^{(1)}_{\vec{n}}(\mathcal{L}_{\vec{\tau}};q,a)$ for the
corresponding special LMOV functions for the $\tau$-framed colored
HOMFLY-PT invariants $\{H_{\vec{r}}(\mathcal{L}_{\vec{\tau}};q,a)\}$
and composite invariants
$\{C_{\vec{r}}(\mathcal{L}_{\vec{\tau}};q,a)\}$ respectively.

By formula (\ref{formula-Hlambda}), we obtain that the corresponding
$F$-invariants satisfies
\begin{align}
F_{\vec{\mu}}^{(h)}(\mathcal{L}_{\vec{\tau}};q,-a)&=(-1)^{|\vec{\mu}|}F_{\vec{\mu}}^{(h)}(\mathcal{L}_{\vec{\tau}};q,a),
\\\nonumber
F_{\vec{\mu}}^{(h)}(\mathcal{L}_{\vec{\tau}};-q,a)&=(-1)^{|\vec{\mu}|}F_{\vec{\mu}}^{(h)}(\mathcal{L}_{\vec{\tau}};q,a),
\\\nonumber
F_{\vec{\mu}}^{(h)}(\mathcal{L}_{\vec{\tau}};q^{-1},a)&=(-1)^{|\vec{\mu}|}F_{\vec{\mu}}^{(h)}(\mathcal{L}_{\vec{\tau}};q,a).
\end{align}
and these symmetries similarly hold for $
\mathcal{F}_{\vec{n}}^{(h)}(\mathcal{L}_{\vec{\tau}};q,-a). $

It follows that
\begin{align} \label{formula-gmusym}
g^{(h)}_{\vec{\mu}}(\mathcal{L}_{\vec{\tau}};q,-a)&=(-1)^{|\vec{\mu}|}g^{(h)}_{\vec{\mu}}(\mathcal{L}_{\vec{\tau}};q,a),\\\nonumber
g^{(h)}_{\vec{\mu}}(\mathcal{L}_{\vec{\tau}};-q,a)&=g^{(h)}_{\vec{\mu}}(\mathcal{L}_{\vec{\tau}};q,a),\\\nonumber
g^{(h)}_{\vec{\mu}}(\mathcal{L}_{\vec{\tau}};q^{-1},a)&=g^{(h)}_{\vec{\mu}}(\mathcal{L}_{\vec{\tau}};q,a).
\end{align}
and
\begin{align}
\mathfrak{f}^{(h)}_{\vec{n}}(\mathcal{L}_{\vec{\tau}};q,-a)&=(-1)^{|\vec{n}|}\mathfrak{f}^{(h)}_{\vec{n}}(\mathcal{L}_{\vec{\tau}};q,a),
\\\nonumber
\mathfrak{f}^{(h)}_{\vec{n}}(\mathcal{L}_{\vec{\tau}};-q,a)&=(-1)^{|\vec{n}|}\mathfrak{f}^{(h)}_{\vec{n}}(\mathcal{L}_{\vec{\tau}};q,a),
\\\nonumber
\mathfrak{f}^{(h)}_{\vec{n}}(\mathcal{L}_{\vec{\tau}};q^{-1},a)&=(-1)^{|\vec{n}|}\mathfrak{f}^{(h)}_{\vec{n}}(\mathcal{L}_{\vec{\tau}};q,a).
\end{align}
for both $h=0$ and $1$.
\begin{remark}
We know that $g^{(h)}_{\vec{\mu}}(\mathcal{L}_{\vec{\tau}};q,a)$
belongs to the ring $\mathbb{Q}[q^{\pm},a^{\pm}]$ with the
$(q^{r}-q^{-r})$ as its denominators, formula (\ref{formula-gmusym})
implies that there exists some $d_0\in \mathbb{N}$, such that
\begin{align}
g^{(h)}_{\vec{\mu}}(\mathcal{L}_{\vec{\tau}};q,a)\in
z^{-2d_0}a^{\epsilon}\mathbb{Q}[[z^2]][a^{\pm 2}]
\end{align}
where $\epsilon\in\{0,1\}$ which is determined by $|\vec{\mu}|\equiv
\epsilon \mod 2$.
\end{remark}

From formula (\ref{formula-HLlamdamu}), we know that for
$\vec{\lambda}\neq \vec{0}$, both
$H_{\vec{\lambda}}(\mathcal{L}_{\vec{\tau}};q,a)$  and
$C_{\vec{\lambda}}(\mathcal{L}_{\vec{\tau}};q,a)$ contain a factor
$\frac{a-a^{-1}}{q-q^{-1}}$.  We obtain
\begin{theorem} \label{Theorem-LMOVfunction}
The reformulated LMOV functions
$g^{(h)}_{\vec{\mu}}(\mathcal{L}_{\vec{\tau}};q,a)$ for both $h=0$
and $1$ can be written in the following form
\begin{align}
g^{(h)}_{\vec{\mu}}(\mathcal{L}_{\vec{\tau}};q,a)=(a-a^{-1})\tilde{g}^{(h)}_{\vec{\mu}}(\mathcal{L}_{\vec{\tau}};q,a),
\end{align}
and $\tilde{g}^{(h)}_{\vec{\mu}}(\mathcal{L}_{\vec{\tau}};q,a)$ has
the properties
\begin{align}
\tilde{g}^{(h)}_{\vec{\mu}}(\mathcal{L}_{\vec{\tau}};q,-a)&=(-1)^{|\vec{\mu}|-1}\tilde{g}^{(h)}_{\vec{\mu}}(\mathcal{L}_{\vec{\tau}};q,a),
\\
\tilde{g}^{(h)}_{\vec{\mu}}(\mathcal{L}_{\vec{\tau}};-q,a)&=\tilde{g}^{(h)}_{\vec{\mu}}(\mathcal{L}_{\vec{\tau}};q,a),
\\
\tilde{g}^{(h)}_{\vec{\mu}}(\mathcal{L}_{\vec{\tau}};q^{-1},a)&=\tilde{g}^{(h)}_{\vec{\mu}}(\mathcal{L}_{\vec{\tau}};q,a).
\end{align}

\end{theorem}

The LMOV conjecture for framed colored HOMFLY-PT invariants
\cite{MV02,CLPZ14} which is an extension form of the original LMOV
conjecture for (unframed) colored HOMFY-PT invariant
\cite{OV00,LaMaVa00,LaMa02,LP10}
 states that
\begin{align} \label{formula-LMOV1}
g_{\vec{\mu}}^{(0)}(\mathcal{L}_{\vec{\tau}};q,a)\in
z^{-2}\mathbb{Z}[z^2,a^{\pm 1}].
\end{align}
and
\begin{align} \label{formula-LMOV2}
f_{\vec{n}}^{(0)}(\mathcal{L}_{\vec{\tau}};q,a)\in
(q-q^{-1})^{-1}\mathbb{Z}[q^{\pm 1},a^{\pm 1}].
\end{align}
Later such framed version LMOV conjecture was generalized to the
case of the framed composite invariants \cite{Mar10,CZ20} which
states that $g_{\vec{\mu}}^{(1)}(\mathcal{L}_{\vec{\tau}};q,a)$ and
$f_{\vec{n}}^{(1)}(\mathcal{L}_{\vec{\tau}};q,a)$  satisfy formulas
(\ref{formula-LMOV1}) and (\ref{formula-LMOV2}) respectively.

By Theorem \ref{Theorem-LMOVfunction}, we obtain
\begin{conjecture}[Refined LMOV conjecture for framed links]
\label{conjecture-refinedLMOV1} For  $h=0$ and $1$, the reformulated
LMOV functions can be written as
\begin{align}
g^{(h)}_{\vec{\mu}}(\mathcal{L}_{\vec{\tau}};q,a)=(a-a^{-1})\tilde{g}^{(h)}_{\vec{\mu}}(\mathcal{L}_{\vec{\tau}};q,a).
\end{align}
where
\begin{align}
\tilde{g}^{(h)}_{\vec{\mu}}(\mathcal{L}_{\vec{\tau}};q,a)\in
z^{-2}a^{\epsilon}\mathbb{Z}[z^2,a^{\pm 2}]
\end{align}
$\epsilon\in \{0,1\}$ is determined by $ |\vec{\mu}|-1\equiv
\epsilon \mod 2. $

In other words, there are integral invariants
\begin{align}
\tilde{N}_{\vec{\mu},g,Q}^{(h)}(\mathcal{L}_{\vec{\tau}})\in
\mathbb{Z},
\end{align}
 such that
\begin{align}
\tilde{g}^{(h)}_{\vec{\mu}}(\mathcal{L}_{\vec{\tau}};q,a)=\sum_{g\geq
0}\sum_{Q\in
\mathbb{Z}}\tilde{N}_{\vec{\mu},g,Q}^{(h)}(\mathcal{L}_{\vec{\tau}})z^{2g-2}a^{2Q+\epsilon}\in
z^{-2}a^{\epsilon}\mathbb{Z}[z^2,a^{\pm 2}].
\end{align}
\end{conjecture}

Similarly, as to the special LMOV function, we also have the
following integrality conjecture.
\begin{conjecture} \label{conjecture-specialLMOV}
For  $h=0$ or $1$, the special LMOV functions can be written as
\begin{align}
\mathfrak{f}^{(h)}_{\vec{n}}(\mathcal{L}_{\vec{\tau}};q,a)=\frac{(a-a^{-1})}{(q-q^{-1})}
\tilde{\mathfrak{f}}^{(h)}_{\vec{n}}(\mathcal{L}_{\vec{\tau}};q,a).
\end{align}
where
\begin{align}
\tilde{\mathfrak{f}}^{(h)}_{\vec{n}}(\mathcal{L}_{\vec{\tau}};q,a)\in
a^{\epsilon} z\mathbb{Z}[q^{\pm 1},a^{\pm 2}]
\end{align}
and where $\epsilon\in \{0,1\}$ is determined by $ |\vec{n}|\equiv
\epsilon \mod 2. $ In other words, there are integral invariants
\begin{align}
\tilde{\mathcal{N}}_{\vec{\mu},i,j}^{(h)}(\mathcal{L}_{\vec{\tau}})\in
\mathbb{Z},
\end{align}
such that
\begin{align}
(q-q^{-1})\mathfrak{f}^{(h)}_{\vec{\mu}}(\mathcal{L}_{\vec{\tau}};q,a)=\sum_{i\in
\mathbb{Z}}\sum_{j\in
\mathbb{Z}}\tilde{\mathcal{N}}_{\vec{\mu},i,j}^{(h)}(\mathcal{L}_{\vec{\tau}})q^{i}a^{2j+\epsilon}\in
a^{\epsilon}\mathbb{Z}[q^{\pm 1},a^{\pm 2}].
\end{align}
\end{conjecture}

\begin{remark}
K. Liu and P. Peng \cite{LP10} first studied the mathematical
structures of LMOV conjecture for general links without framing
contribution, which is equivalent to the LMOV conjecture for colored
HOMFLY-PT invariants $W_{\vec{\lambda}}(\mathcal{L};q,a)$. They
provided a proof for this case by using cut-and-join analysis and
the cabling technique \cite{LP10}. Motivated by the work
\cite{MV02}, in \cite{CLPZ14}, the author together with Q. Chen, K.
Liu and P. Peng applied the HOMFLY-PT skein theory to study the
mathematical structures of LMOV conjecture for $\tau$-framed colored
HOMFLY-PT invariants
$H_{\vec{\lambda}}(\mathcal{L}_{\vec{\tau}};q,a)$.
\end{remark}

\subsection{A special case of the refined LMOV conjecture}
From the formula (\ref{formula-tildegmu}), we have
\begin{align}
\tilde{g}_{\vec{\mu}}^{(h)}(\mathcal{L}_{\vec{\tau}};q,a)&=\frac{1}{a-a^{-1}}g_{\vec{\mu}}^{(h)}(\mathcal{L}_{\vec{\tau}};q,a)\\\nonumber
&=\frac{1}{a-a^{-1}}\frac{\mathfrak{z}_{\vec{\mu}}}{\{\vec{\mu}\}}\sum_{d|\vec{\mu}}\frac{\mu(d)}{d}
F^{(h)}_{\frac{\vec{\mu}}{d}}(\mathcal{L}_{\vec{\tau}};q^d,a^d).
\end{align}

In particular, when $\mathcal{L}_{\vec{\tau}}$ is a framed knot
$\mathcal{K}_{\tau}$ (i.e. a link with only one component), where
$\tau=w(\mathcal{K})$ and $\mu=(p)$. By formula
(\ref{formula-logZx}), we obtain
\begin{align}
F_{(p)}^{(0)}(\mathcal{K}_{\tau};q,a)=\frac{1}{p}(-1)^{p\tau}\mathcal{H}(\mathcal{K}_\tau\star
P_{p};q,a),
\end{align}
Then the refined LMOV conjecture (\ref{conjecture-refinedLMOV1})
implies that
\begin{conjecture} \label{conjecture-Specailcase}
For any prime $p$, we have
\begin{align}
\tilde{g}_{p}^{(0)}(\mathcal{K}_{\tau};q,a)&=
\frac{1}{a-a^{-1}}\frac{1}{\{p\}}\left(\mathcal{H}(\mathcal{K}_\tau\star
P_{p};q,a)\right.\\\nonumber&\left.-(-1)^{(p-1)\tau}\Psi_p\left(\mathcal{H}(\mathcal{K}_\tau;q,a)\right)\right)\in
z^{-2}a^{\epsilon}\mathbb{Z}[z^{2},a^{\pm 2}],
\end{align}
where $\epsilon=1$ when $p=2$ and $\epsilon=0$ when $p$ is an odd
prime.
\end{conjecture}

\begin{remark}
In other words, the refined LMOV Conjecture
\ref{conjecture-refinedLMOV1} implies that, for any prime number
$p$,
\begin{align} \label{formula-hecke}
\{p\}\mathcal{H}_{p}(\mathcal{K};q,a)- (-1)^{(p-1)w(\mathcal{K})}
\Psi_p\left(\{1\}\mathcal{H}(\mathcal{K};q,a)\right)\in
[p]^2a^{\epsilon}\mathbb{Z}[z^2,a^2].
\end{align}
where $\epsilon=0$ when $p=2$ and $\epsilon$ when $p$ is an odd
prime. Formula (\ref{formula-hecke}) is referred as Hecke lifting in
\cite{CLPZ14}.  In order to prove it,  we introduced the notion of
congruence skein relations for colored HOMFLY-PT invariants
\cite{CLPZ14} which may have independent interests elsewhere.
\end{remark}

\subsection{New integral link invariants}
The refined LMOV conjecture \ref{conjecture-refinedLMOV1} predicts
the new integral link invariants
$\tilde{N}_{\vec{\mu},g,Q}^{(h)}(\mathcal{L})$ and
$\tilde{\mathcal{N}}_{\vec{n},i,j}^{(h)}(\mathcal{L})$ for the
framed link $\mathcal{L}$. A central question in this direction is
how to define these new integral invariants directly by
geometric/algeberic/combinatoric method and find the geometric
meaning for them.

As to integral invariants
$\tilde{N}_{\vec{\mu},g,Q}^{(0)}(\mathcal{L})$,  there is an
interpretation from the physics literatures \cite{OV00, LaMaVa00}.
It was conjectured in \cite{OV00} that for the link $\mathcal{L}$,
one can construct a Lagrangian submanifold
$\mathcal{C}_{\mathcal{L}}$ of the resolved conifold, with
$b_1(\mathcal{C}_{\mathcal{L}})=L$ which is the number of the
components of $\mathcal{C}_{\mathcal{L}}$. Let $\gamma_{\alpha}$,
$\alpha=1,...,L$ be the one-cycles representing a basis of
$H_1(\mathcal{C}_{\mathcal{L}};\mathbb{Z})$. Let
$\mathcal{M}_{g,h,Q}$ be the conjectural moduli space of Riemann
surfaces of genus $g$ with $h$ holes embedded into the resolved
conifold, such that there are $h_\alpha$ holes ending on the cycles
$\gamma_{\alpha}$ for $\alpha=1,...,L$. The symmetric group
$\Sigma_{h_1}\times \cdots \times \Sigma_{h_L}$ acts on the Riemann
surfaces by exchanging the $h_\alpha$ holes that end on
$\gamma_\alpha$. Then the integral invariants
$\tilde{N}_{\vec{\mu},g,Q}^{(0)}(\mathcal{L})$ can be interpreted as
the Euler number
$\chi(\mathbf{S}_{\vec{\mu}}(H^*(\mathcal{M}_{g,h,Q})))$, where
$\mathbf{S}_{\vec{\mu}}=\mathbf{S}_{\mu^1}\otimes \cdots\otimes
\mathbf{S}_{\mu^L}$, and $\mathbf{S}_{\mu^{\alpha}}$ is referred as
being the Schur functor.

For the integral invariants
$\tilde{\mathcal{N}}_{n,i,j}^{(0)}(U_{\tau})$ for framed unknot
$U_{\tau}$, , we related it to the Betti number of the cohomological
Hall algebra and quiver variety of a corresponding quiver in
\cite{LuoZhu16,Zhu19-1}.  Then the idea of the knot-quiver
correspondence was further extended in \cite{KRSS17-1,KRSS17-2}, see
\cite{EKL20,PSS18,SW20-1,SW20-2} for more recent developments.

Obviously, these integral link invariants are fully determined by
the Chern-Simons partition function of link. From the point of view
of topological string theory, Aganagic and Vafa \cite{AV12}
introduced the  $a$-deformed $A$-polynomial for colored HOMFLY-PT
invariants which is the encode of the mirror geometry of the
topological string theory. According to the large $N$ duality of
Chern-Simons and topological string theory, these new integral link
invariants are fully determined by the $a$-deformed A-polynomial.

For a series of colored HOMFLY-PT invariant
$\{\mathcal{H}_n(\mathcal{K}; q,a)\}_{n\geq 0}$ of symmetric
representation of a knot $\mathcal{K}$, we introduce two operators
$M$ and $L$ act on $\{\mathcal{H}_n(\mathcal{K};q,a)\}_{n\geq 0}$ as
follow:
\begin{align*}
M\mathcal{H}_n=q^{n}\mathcal{H}_n,\
L\mathcal{H}_{n}=\mathcal{H}_{n+1}.
\end{align*}
then $LM=qML$.
\begin{definition}
The noncommutative $a$-deformed A-polynomial for series
$\{\mathcal{H}_n(q,a)\}_{n\geq 0}$ is a polynomial
$\hat{A}(M,L;q,a)$ of operators $M, L$, such that
\begin{align*}
\hat{A}(M,L;q,a)\mathcal{H}_n(q,a)=0, \text{for } \ n\geq 0.
\end{align*}
and $A(M,L;a)=\lim_{q\rightarrow 1}\hat{A}(M,L;q,a)$ is called the
$a$-deformed A-polynomial.
\end{definition}

The existence of $a$-deformed A-polynomial was proved rigidly in
\cite{GLL18}.

\begin{example}
As to the framed unknot $U_\tau$, the noncommutative a-deformed
A-polynomial for $U_\tau$ as follow:
\begin{align*}
\hat{A}_{U_\tau}(M,L,q;a)=(-1)^\tau(qM-1)L-M^{\tau}(a^{\frac{1}{2}}q^{\frac{1}{2}}M-a^{-\frac{1}{2}}q^{\frac{1}{2}}).
\end{align*}
and the a-deformed A-polynomial is
\begin{align*}
A_{U_\tau}(M,L;a)=\lim_{q\rightarrow
1}\hat{A}(M,L,q;a)=(-1)^\tau(M-1)L-M^{\tau}(a^{\frac{1}{2}}M-a^{-\frac{1}{2}}),
\end{align*}
From this formula, one can compute the explicit formula for the
integral invariant $\tilde{N}_{m,0,Q}(U_\tau)$, see \cite{Zhu19-3}.
\end{example}

On the other hand side, there is a proposal initiated by M.
Aganagic, T. Ekholm, L. Ng and Vafa \cite{AENV13} which connects the
topological strings and contact homology theory. In this framework,
it is conjectured that the $a$-deformed A-polynomial is equal to the
augmentation polynomial of knot contact homology after variable
transformation. Therefore, there should be a geometric way to define
these new integral link invariant from knot contact homology, we
leave the detailed discussion to another paper \cite{Zhu21}.

\section{Special polynomials} \label{section-special polynomials}
From this section, we begin to study two specializations of the
normalized framed full colored HOMFLY-PT invariants
$\mathcal{P}_{[\lambda,\mu]}(\mathcal{K};q,a)$.

By Theorem \ref{Theorem-strongintegrality}, we know that when $q=1$,
$\mathcal{P}_{[\lambda,\mu]}(\mathcal{K};q=1,a)$ is  a well-defined
polynomial of $a$ lie in the ring $\mathbb{Z}[a]$.

For a knot $\mathcal{K}$ and a partition $\lambda\in \mathcal{P}$,
P. Dunin-Barkowski, A. Mironov, A. Morozov, A. Sleptsov and A.
Smirnov \cite{BMMSS11} introduced the following special polynomial
for colored  HOMFLY-PT invariant of the knot $\mathcal{K}$
\begin{align} \label{formula-specialHlambda}
H_{\lambda}^\mathcal{K}(a)=\lim_{q\rightarrow
1}\frac{W_{\lambda}(\mathcal{K};q,a)}{W_{\lambda}(U;q,a)}.
\end{align}

After testing many examples \cite{BMMSS11,IMMM12}, they proposed the
following conjectural formula:
\begin{align} \label{formula-special1}
H_{\lambda}^{\mathcal{K}}(a)=H_{(1)}^{\mathcal{K}}(a)^{|\lambda|}.
\end{align}

A rigid mathematical proof of the formula (\ref{formula-special1})
was provided in \cite{LP10} and  in \cite{Zhu13} with different
methods.

According to the formula (\ref{formula-specialHlambda}), indeed, we
have
\begin{align}
H_{\lambda}^\mathcal{K}(a)=\lim_{q\rightarrow 1}
a^{-|\lambda|w(\mathcal{K})}\mathcal{P}_{[\lambda,\emptyset]}(\mathcal{K};q,a)=a^{-|\lambda|w(\mathcal{K})}
\mathcal{P}_{[\lambda,\emptyset]}(\mathcal{K};1,a).
\end{align}

Note that when $\lambda=(1)$,
\begin{align}
\mathcal{P}_{[(1),\emptyset]}(\mathcal{K};q,a)=\frac{\mathcal{H}(\mathcal{K};q,a)}{\mathcal{H}(U;q,a)},
\end{align}
which will be denoted by $\mathcal{P}(\mathcal{K};q,a)$ for brevity
in the following.

Therefore, the formula (\ref{formula-special1}) can rewritten as
\begin{align}
\mathcal{P}_{[\lambda,\emptyset]}(\mathcal{K};1,a)=\mathcal{P}_{[(1),\emptyset]}(\mathcal{K};1,a)^{|\lambda|}=\mathcal{P}(\mathcal{K};1,a)^{|\lambda|}.
\end{align}

Motivated by the above results, it is natural to introduce
\begin{definition}
Given a knot $\mathcal{K}$ and partitions $\lambda,\mu\in
\mathcal{P}_+$, the special normalized framed full colored HOMFLY-PT
invariants of $\mathcal{K}$ is defined by
$\mathcal{P}_{[\lambda,\mu]}(\mathcal{K};1,a)$.
\end{definition}

A direct consequence of Theorem \ref{Theorem-strongintegrality}
gives
\begin{corollary}
For any knot $\mathcal{K}$, we have
\begin{align}
\mathcal{P}_{[\lambda,\mu]}(\mathcal{K};1,a)\in
a^{\epsilon}\mathbb{Z}[a^2],
\end{align}
where $\epsilon\in \{0,1\}$ is determined by
$(|\lambda|+|\mu|)w(\mathcal{K})\equiv \epsilon \mod 2$.
\end{corollary}

As a generalization of the formula (\ref{formula-special1}),  we
proved in \cite{CZ20} (cf. Theorem 5.2  in \cite{CZ20}) that
\begin{theorem} \label{Theorem-special-1-tangle}
For any knot $\mathcal{K}$ and partitions $\lambda,\mu\in
\mathcal{P}_+$,
\begin{align}
\mathcal{P}_{[\lambda,\mu]}(\mathcal{K};1,a)=\mathcal{P}(\mathcal{K};1,a)^{|\lambda|+|\mu|}.
\end{align}
\end{theorem}

Since the special polynomial plays important role in studying the
genus expansion of the quantum invariant,  in this article, we also
introduce the special polynomial for the composite invariants.
\begin{definition}
Given a knot $\mathcal{K}$, the special composite invariants is
define as the following limit
\begin{align}
\mathcal{D}_{\lambda}(\mathcal{K};a)=\lim_{q\rightarrow
1}\frac{\mathcal{C}_{\lambda}(\mathcal{K};q,a)}{\mathcal{C}_{\lambda}(U;q,a)}.
\end{align}
\end{definition}

\begin{theorem} \label{Theorem-special-composite}
For any knot $\mathcal{K}$, we have
\begin{align}
\mathcal{D}_{\lambda}(\mathcal{K};a)=\mathcal{P}(\mathcal{K};1,a)^{|\lambda|}.
\end{align}
\end{theorem}
We need the following lemma for the algebraic structure of HOMFLY-PT
polynomial.

\begin{lemma}[See for example \cite{LiMi87}, \cite{ZhouZhu19}] For a link $\mathcal{L}$ with $L$ components $\mathcal{K}_\alpha,
\alpha=1,..,L$, the HOMFLY-PT polynomial $P(\mathcal{L};q,a)$ of
$\mathcal{L}$ can be written in the following form
\begin{align}
P(\mathcal{L};q,a)=\sum_{g\geq
0}p^\mathcal{L}_{2g+1-L}(a)(q-q^{-1})^{2g+1-L}.
\end{align}
Moreover,
\begin{align}
p_{1-L}^{\mathcal{L}}(a)=a^{-2lk(\mathcal{L})}(a-a^{-1})^{L-1}\prod_{\alpha=1}^{L}p_0^{\mathcal{K}_\alpha}(a)
\end{align}
where $p_0^{\mathcal{K}_\alpha}(a)$ is the HOMFLY-PT polynomial of
the $\alpha$-th component $\mathcal{K}_\alpha$ of the link
$\mathcal{L}$ with $q=1$, i.e.
$p_0^{\mathcal{K}_\alpha}(a)=P(\mathcal{K}_\alpha;1,a)=a^{-w(\mathcal{K}_\alpha)}\mathcal{P}(\mathcal{K}_\alpha;1,a)$.
\end{lemma}
By the definition of the HOMFLY-PT polynomial, we have
\begin{align} \label{homflyexpansion}
\mathcal{H}(\mathcal{L};q,a)=\sum_{g\geq
0}\hat{p}^{\mathcal{L}}_{2g+1-L}(a)(q-q^{-1})^{2g-L}
\end{align}
where
$\hat{p}^{\mathcal{L}}_{2g+1-L}(a)=a^{w(\mathcal{L})}p^{\mathcal{L}}_{2g+1-L}(a)(a-a^{-1})$.
Hence
\begin{align} \label{homflyexpansioninitial}
\hat{p}^{\mathcal{L}}_{1-L}(a)=a^{\sum_{\alpha=1}^Lw(\mathcal{K}_\alpha)}(a-a^{-1})^L\prod_{\alpha=1}^{L}p_{0}^{\mathcal{K}_\alpha}(a)
=(a-a^{-1})^L\prod_{\alpha=1}^L\mathcal{P}(\mathcal{K}_\alpha;1,a).
\end{align}

We now prove the Theorem \ref{Theorem-special-composite}.
\begin{proof}
We consider the skein element
\begin{align}
Q^{\lambda}=\sum_{\mu,\nu}c_{\mu,\nu}^{\lambda}Q_{\mu,\nu},
\end{align}
where
\begin{align}
c_{\mu,\nu}^{\lambda}=\sum_{A,B}\frac{\chi_{\mu}(A)\chi_{\nu}(B)}{\mathfrak{z}_{A}\mathfrak{z}_B}\chi_{\lambda}(A\cup
B)
\end{align}
and
\begin{align} \label{formula-Qlambdamu2}
Q_{\mu,\nu}&=Q_{\mu}Q_{\nu}^*+\sum_{\sigma\neq
\emptyset}(-1)^{|\sigma|}c_{\sigma,\rho}^{\mu}c_{\sigma^t,\tau}^{\nu}Q_{\rho}Q_{\tau}^*.
\end{align}

By using the Frobenius formula (\ref{formula-frobeniusQ}), we obtain

\begin{align}
Q^{\lambda}&=\sum_{\mu,\nu}\sum_{A,B}\frac{\chi_{\mu}(A)\chi_{\nu}(B)}{\mathfrak{z}_{A}\mathfrak{z}_B}\chi_{\lambda}(A\cup
B)Q_{\mu}Q_{\nu}^*\\\nonumber
&+\sum_{\mu,\nu}c_{\mu,\nu}^{\lambda}\sum_{\sigma\neq
\emptyset}(-1)^{|\sigma|}c_{\sigma,\rho}^{\mu}c_{\sigma^t,\tau}^{\nu}Q_{\rho}Q_{\tau}^*\\\nonumber
&=\sum_{A,B}\frac{\chi_{\lambda}(A\cup
B)}{\mathfrak{z}_{A}\mathfrak{z}_B}P_{A}P_{B}^*+\sum_{\mu,\nu}c_{\mu,\nu}^{\lambda}\sum_{\sigma\neq
\emptyset}(-1)^{|\sigma|}c_{\sigma,\rho}^{\mu}c_{\sigma^t,\tau}^{\nu}Q_{\rho}Q_{\tau}^*\\\nonumber
&=\sum_{|A|+|B|=|\lambda|}\frac{\chi_{\lambda}(1^{|\lambda|})}{\mathfrak{z}_{(1^{|A|})}\mathfrak{z}_{(1^{|B|})}}P_{(1^{|A|})}P_{(1^{|B|})}^{*}+\sum_{s}LT_s.
\end{align}

The main observation is that every terms in the first summation
\begin{align}
\sum_{|A|+|B|=|\lambda|}\frac{\chi_{\lambda}(1^{|\lambda|})}{\mathfrak{z}_{(1^{|A|})}\mathfrak{z}_{(1^{|B|})}}P_{(1^{|A|})}P_{(1^{|B|})}^{*}
\end{align}
contains $|A|+|B|=|\lambda|$ components in the skein $\mathcal{C}$,
while the remain terms $LT_s$ in the second summation have the
components less than $|\lambda|$.

By definition, we have
\begin{align}
\mathcal{C}_{\lambda}(\mathcal{K};q,a)&=\sum_{|A|+|B|=|\lambda|}\frac{\chi_{\lambda}(1^{|\lambda|})}{\mathfrak{z}_{(1^{|A|})}\mathfrak{z}_{(1^{|B|})}}
\mathcal{H}(\mathcal{K}\star
P_{(1^{|A|})}P_{(1^{|B|})}^{*};q,a)\\\nonumber
&+\sum_s\mathcal{H}(\mathcal{K}\star LT_s;q,a)
\end{align}
and
\begin{align}
\mathcal{C}_{\lambda}(U;q,a)&=\sum_{|A|+|B|=|\lambda|}\frac{\chi_{\lambda}(1^{|\lambda|})}{\mathfrak{z}_{(1^{|A|})}\mathfrak{z}_{(1^{|B|})}}
\left(\frac{a-a^{-1}}{q-q^{-1}}\right)^{|\lambda|}\\\nonumber
&+\sum_s\mathcal{H}(U\star LT_s;q,a).
\end{align}

Since $\mathcal{K}\star P_{(1^{|A|})}P_{(1^{|B|})}^*$ is a link with
$|A|+|B|=|\lambda|$ components,  according to the expansion formula
(\ref{homflyexpansion}), we have
\begin{align}
\mathcal{H}(\mathcal{K}\star
P_{(1^{|A|})}P_{(1^{|B|})}^*;q,a)=\sum_{g\geq
0}\hat{p}_{2g+1-(|\lambda|)}^{\mathcal{K}\star
P_{(1^{|A|})}P_{(1^{|B|})}^*}(a)(q-q^{-1})^{2g-|\lambda|}.
\end{align}
For the link $\mathcal{K}\star LT_s$ with the number of components
$L(\mathcal{K}\star LT_s)\leq |\lambda|-1$, we also have
\begin{align}
\mathcal{H}(\mathcal{K}\star LT_s;q,a)=\sum_{g\geq
0}\hat{p}_{2g+1-L(\mathcal{K}\star LT_s)}^{\mathcal{K}\star
LT_s}(a)(q-q^{-1})^{2g-L(\mathcal{K}\star LT_s)}.
\end{align}
Since
$\frac{\chi_{\lambda}(1^{|\lambda|})}{\mathfrak{z}_{(1^{|A|})}\mathfrak{z}_{(1^{|B|})}}\neq
0$ when $|A|+|B|=|\lambda|$, by a direct calculation, we obtain
\begin{align}
\lim_{q\rightarrow
1}\frac{\mathcal{C}_{\lambda}(\mathcal{K};q,a)}{\mathcal{C}_{\lambda}(U;q,a)}=\frac{\sum_{|A|+|B|=|\lambda|}
\frac{1}{\mathfrak{z}_{(1^{|A|})}\mathfrak{z}_{(1^{|B|})}}\hat{p}_{1-|\lambda|}^{\mathcal{K}\star
P_{(1^{|A|})}P_{(1^{|B|})}^*}(a)}{(a-a^{-1})^{|\lambda|}\sum_{|A|+|B|=|\lambda|}\frac{1}{\mathfrak{z}_{(1^{|A|})}\mathfrak{z}_{(1^{|B|})}}}
\end{align}
Moreover, the formula (\ref{homflyexpansioninitial}) implies
\begin{align}
\hat{p}_{1-|\lambda|}^{\mathcal{K}\star
P_{(1^{|A|})}P_{(1^{|B|})}^*}(a)&=(a-a^{-1})^{|\lambda|}\mathcal{P}(\mathcal{K};1,a)^{|\lambda|},
\end{align}
since the number of components $L(\mathcal{K}\star
P_{(1^{|A|})}P_{(1^{|B|})}^*)=|A|+|B|=|\lambda|$.  Therefore, we
have
\begin{align}
\lim_{q\rightarrow
1}\frac{\mathcal{C}_{\lambda}(\mathcal{K};q,a)}{\mathcal{C}_{\lambda}(U;q,a)}=\mathcal{P}_\mathcal{K}(1,a)^{|\lambda|}.
\end{align}
\end{proof}

\section{Colored Alexander polynomials} \label{Section-coloredAlender}
In \cite{IMMM12}, the authors also considered another limit for
colored HOMFLY-PT invariants
\begin{align}
A_{\lambda}(\mathcal{K};q)=\lim_{a\rightarrow
1}\frac{W_{\lambda}(\mathcal{K};q,a)}{W_{\lambda}(\mathcal{K};q,a)}
\end{align}

When $\lambda=(1)$, by its definition, we obtain
$A_{(1)}(\mathcal{K};q):=A(\mathcal{K};q)$ which is the Alexander
polynomial for knot $\mathcal{K}$, so we call
$A_{\lambda}(\mathcal{K};q)$ the colored Alexander polynomial of
$\mathcal{K}$. It motivates us to consider the limit of the full
framed colored HOMFLY-PT invariants.
\begin{definition}
The full (framed) colored Alexander polynomial of $\mathcal{K}$ is
defined by
\begin{align}
\mathcal{A}_{[\lambda,\mu]}(\mathcal{K};q):=\lim_{a\rightarrow
1}\frac{\mathcal{H}_{[\lambda,\mu]}(\mathcal{K};q,a)}{\mathcal{H}_{[\lambda,\mu]}(U;q,a)}.
\end{align}
In particular, when $\mu=\emptyset$, we have the relationship
\begin{align}
\mathcal{A}_{[\lambda,\emptyset]}(\mathcal{K};q)=q^{\kappa_{\lambda}w(\mathcal{K})}A_{\lambda}(\mathcal{K};q),
\end{align}
according to their definitions.
\end{definition}

Recall the definition of the normalized framed full colored
HOMFLY-PT invariant, actually,
\begin{align} \label{formula-Alamdbamu}
\mathcal{A}_{[\lambda,\mu]}(\mathcal{K};q)=\mathcal{P}_{[\lambda,\mu]}(\mathcal{K};q,a=1).
\end{align}
By Theorem \ref{Theorem-strongintegrality}, we obtain that
$\mathcal{P}_{[\lambda,\mu]}(\mathcal{K};q,1)$ is a well-defined
polynomial of $q$ lies in the ring $\mathbb{Z}[q^2]$.
\begin{corollary}
The full colored Alexander polynomial
$\mathcal{A}_{[\lambda,\mu]}(\mathcal{K};q)$ is well-defined and
\begin{align}
\mathcal{A}_{[\lambda,\mu]}(\mathcal{K};q)\in \mathbb{Z}[q^2].
\end{align}
\end{corollary}

In \cite{BMMSS11}, the authors studied the properties for
$A_{\lambda}(\mathcal{K};q)$, they proposed the conjectural formula
\begin{align} \label{formula-Alambda}
A_{\lambda}(\mathcal{K};q)=A(\mathcal{K};q^{|\lambda|}).
\end{align}
However, in \cite{Zhu13}, we found that the formula
(\ref{formula-Alambda}) does not hold for the non-hook partition
such as $\lambda=(2,2)$, and we modified the above formula in the
following form
\begin{conjecture} \label{conjecture-Alexander}
When $\lambda$ is a hook partition,
\begin{align}
A_{\lambda}(\mathcal{K};q)=A(\mathcal{K};q^{|\lambda|}).
\end{align}
\end{conjecture}
We proved in \cite{Zhu13} that
\begin{theorem} \label{Theorem-torusknot}
The Conjecture \ref{conjecture-Alexander} holds for torus knot
$T_{m,r}$.
\end{theorem}

\begin{remark}
Based on some concrete computations, we found that the analogue
conjecture doesn't hold for the full colored Alexander polynomial
(\ref{formula-Alamdbamu}).
\end{remark}

In the following, we will show that the Conjecture
\ref{conjecture-Alexander} can be reduced to an identity for the
characters of Hecke algebra.

First, we recall that every hook partition of weight $d$ can be
presented as the form $(m+1,1,...,1)$ with $n+1$ length for some $m,
n\in \mathbb{Z}_{\geq 0}$, denoted by $(m|n)$, with $m+n+1=d$. It is
clear that $\kappa_{(m|n)}=(m-n)d$.

By the property of character theory of symmetric group, one has
\makeatletter    \label{formula-chid}
\let\@@@alph\@alph
\def\@alph#1{\ifcase#1\or \or $'$\or $''$\fi}\makeatother
\begin{subnumcases}
{\chi_\lambda((d))=} (-1)^n, &\text{if $\lambda$ is a hook partition
$(m|n)$}\\\nonumber 0, &\text{otherwise}
\end{subnumcases}
\makeatletter\let\@alph\@@@alph\makeatother

Given a knot $\mathcal{K}$, let $b$ be a braid presentation of
$\mathcal{K}$. Suppose $b$ has $r$ strands,  given a partition
$\lambda$ of weight $|\lambda|=d$, we construct a $d$-cabling braid
$b^{(d,..,d)}$, that means $b^{(d,..,d)}$ is the new braid
constructed by replacing every strand in $b$ with $d$ parallel
strands. Recall the idempotent $y_{\lambda}$ of Hecke algebra
introduced Section \ref{Section-Idempotent},  let
\begin{align}
X_\lambda(\mathcal{K})=b^{(d,...,d)}\cdot (y_{\lambda}\otimes
\cdots\otimes y_{\lambda})
\end{align}
be the element in Hecke algebra $H_{dr}(q,a)$. Given a partition
$\Lambda$ of weight $|\Lambda|=dr$, denote by $\zeta^{\Lambda}$ a
character of the Hecke algebra $H_{dr}(q,a)$.

Then, we have the following formula due to X. Lin and H. Zheng
\cite{LiZh10}
\begin{align} \label{formula-Wlambda}
W_{\lambda}(\mathcal{K};q,a)=q^{-\kappa_{\lambda}w(\mathcal{K})}a^{-|\lambda|w(\mathcal{K})}\sum_{|\Lambda|=dr}\zeta^{\Lambda}(X_{\lambda}(\mathcal{K}))
s_{\lambda}^{*}(q,a),
\end{align}
where $s_{\lambda}^{*}(q,a)=\prod_{x\in
\lambda}\frac{aq^{cn(x)}-a^{-1}q^{-cn(x)}}{q^{h(x)}-q^{-h(x)}}$ is
the colored HOMFLY-PT invariant of the unknot.

Hence
\begin{align}
&A_{(m|n)}(\mathcal{K};q)=\lim_{a\rightarrow
1}\frac{W_{(m|n)}(\mathcal{K};q,a)}{s_{(m|n)}^{*}(q,a)}\\\nonumber
&=\lim_{a\rightarrow
1}\frac{q^{-d(m-n)w(\mathcal{K})}a^{-dw(\mathcal{K})}\sum_{|\Lambda|=rd}\zeta^{\Lambda}(X_{(m|n)}(\mathcal{K}))\sum_{|\Phi|=rd}
\frac{\chi_{\Lambda}(\Phi)}{\mathfrak{z}_{\Phi}}\frac{\prod_{j=1}^{l(\Phi)}\{\Phi_j\}_a}{\{\Phi_j\}}}{\sum_{\mu}\frac{\chi_{(m|n)}(\mu)}{\mathfrak{z}_{\mu}}
\prod_{j=1}^{l(\mu)}\frac{\{\mu_j\}_a}{\{\mu_j\}}}
\\\nonumber
&=\lim_{a\rightarrow
1}\frac{q^{-d(m-n)w(\mathcal{K})}\sum_{|\Lambda|=rd}\zeta^{\Lambda}(X_{(m|n)}(\mathcal{K}))\left(\frac{\chi_{\Lambda}(rd)}{rd}\frac{\{rd\}_a}{\{rd\}}
+\sum_{l(\Phi)\geq
2}\frac{\chi_{\Lambda}(\Phi)}{\mathfrak{z}_{\Phi}}\frac{\prod_{j=1}^{l(\Phi)}\{\Phi_j\}_a}{\{\Phi_j\}}\right)}{\frac{(-1)^n}{d}\frac{\{d\}_a}{\{d\}}+\sum_{l(\mu)\geq
2}\sum_{\mu}\frac{\chi_{(m|n)}(\mu)}{\mathfrak{z}_{\mu}}
\prod_{j=1}^{l(\mu)}\frac{\{\mu_j\}_a}{\{\mu_j\}}}\\\nonumber
&=\frac{q^{-d(m-n)w(\mathcal{K})}(-1)^n\{d\}}{\{rd\}}\sum_{k+l+1=rd}\zeta^{(k|l)}(X_{(m|n)}(\mathcal{K}))(-1)^l.
\end{align}

Therefore, Conjecture \ref{conjecture-Alexander} is reduced to the
following identity for the characters of the Hecke algebra.
\begin{align} \label{formula-characteridentity}
&q^{-d(m-n)w(\mathcal{K})}(-1)^n\sum_{k+l+1=dr}(-1)^l\zeta^{(k|l)}(X_{(m|n)}(\mathcal{K}))\\\nonumber
&=\sum_{k'+l'+1=r}(-1)^{l'}\zeta^{(k'|l')}(X_1(\mathcal{K}))|_{q\rightarrow
q^d}.
\end{align}

For torus knot $T_{r,s}$ which is defined to be the closure of
$(\sigma_1\cdots\sigma_{r-1})^{s}$ with $r,s$ relatively prime.  Lin
and Zheng \cite{LiZh10} computed that
\begin{align}
\zeta^{\Lambda}(X_\lambda(T_{r,s}))=c^{\Lambda}_{\lambda;r}q^{-s\kappa_{\lambda}+s\kappa_{\Lambda}/r}.
\end{align}
where
\begin{align}
c^{\Lambda}_{\lambda;r}=\sum_{\mu}\frac{\chi_{\lambda}(\mu)}{\mathfrak{z}_{\mu}}\chi_{\Lambda}(r\mu).
\end{align}

In particular, when $\lambda=(m|n)$ with $m+n+1=d$ and
$\Lambda=(k|l)$ with $k+l+1=rd$, then
$q^{\kappa_{(m|n)}}=q^{(m-n)d}$ and
$q^{\kappa_{(k|l)}}=q^{(k-l)rd}$. Hence
\begin{align}
\zeta^{(k|l)}(X_{(m|n)}(T_{r,s}))=\sum_{|\mu|=d}\frac{\chi_{(m|n)}(\mu)}{\mathfrak{z}_{\mu}}\chi_{(k|l)}(r\mu)q^{-ds(m-n)}q^{ds(k-l)}.
\end{align}

Therefore,
\begin{align}
&\sum_{k+l+1=rd}\zeta^{(k|l)}(X_{(m|n)}(T_{r,s}))(-1)^l\\\nonumber&=q^{-ds(m-n)}\sum_{|\mu|=d}\frac{\chi_{(m|n)}(\mu)}{\mathfrak{z}_{\mu}}
\sum_{k+l+1=rd}\chi_{(k|l)}(r\mu)(-1)^lq^{ds(k-l)}\\\nonumber
&=q^{-ds(m-n)}\sum_{|\mu|=d}\frac{\chi_{(m|n)}(\mu)}{\mathfrak{z}_{\mu}}\frac{\prod_{j=1}^{l(\mu)}\{drs\mu_j\}}{\{ds\}}\\\nonumber
&=q^{-ds(m-n)}\frac{\{drs\}}{\{ds\}}\sum_{|\mu|=d}\frac{\chi_{(m|n)}(\mu)}{\mathfrak{z}_{\mu}}\frac{\prod_{j=1}^{l(\mu)}\{drs\mu_j\}}{\{drs\}}\\\nonumber
&=q^{-ds(m-n)}\frac{\{drs\}}{\{ds\}}\sum_{|\mu|=d}\frac{\chi_{(m|n)}(\mu)}{\mathfrak{z}_{\mu}}\sum_{m'+n'+1=d}\chi_{(m'|n')}(\mu)(-1)^{n'}q^{rsd(m'-n')}\\\nonumber
&=q^{-ds(m-n)}\frac{\{drs\}}{\{ds\}}(-1)^nq^{drs(m-n)}.
\end{align}
where we have used the formula  (\ref{formula-sumchi}) in Lemma
\ref{lemma-sumchi} twice.

Finally, we obtain
\begin{align}
A_{(m|n)}(T_{r,s};q)=\frac{\{d\}\{drs\}}{\{dr\}\{ds\}}.
\end{align}
Comparing to the Alexander polynomial for torus knot,  Theorem
\ref{Theorem-torusknot} is proved.

\begin{lemma} (cf.  Lemma 6.5 in \cite{Zhu13}) \label{lemma-sumchi}
Given a partition $B$, we have the following identity,
\begin{align} \label{formula-sumchi}
\sum_{a+b+1=|B|}\chi_{(a|b)}(C_B)(-1)^bu^{a-b}=\frac{\prod_{j=1}^{l(B)}(u^{B_j}-u^{-B_j})}{u-u^{-1}}.
\end{align}
\end{lemma}
\begin{proof}
According to problem 14 at page 49 of \cite{Mac95}, taking $t=u$, we
have
\begin{align}
\prod_{i}\frac{1-u^{-1}x_i}{1-ux_i}=E(-u^{-1})H(u)=1+(u-u^{-1})s_{(a|b)}(x)(-1)^bu^{a-b}
\end{align}
since
$s_{(a|b)}(x)=\sum_{\lambda}\frac{\chi_{(a|b)}(C_\lambda)}{z_\lambda}p_\lambda(x)$,
and
\begin{align} \label{formula-EH}
E(-u^{-1})H(u)&=\frac{H(u)}{H(u^{-1})}\\\nonumber
&=\exp\left(\sum_{r\geq
1}\frac{p_r(x)}{r}(u^r-u^{-r})\right)\\\nonumber &=\prod_{r\geq
1}\exp\left(\frac{p_r(x)}{r}(u^r-u^{-r})\right)\\\nonumber
&=\prod_{r\geq 1}\sum_{m_r\geq
0}\frac{p_r(x)^{m_r}(u^r-u^{-r})^{m_r}}{r^{m_r}m_r!}\\\nonumber
&=\sum_{\lambda}\frac{p_{\lambda}(x)}{\mathfrak{z}_\lambda}\prod_{j=1}^{l(\lambda)}(u^{\lambda_j}-u^{-\lambda_j})
\end{align}
Comparing the coefficients of $p_B(x)$ in (\ref{formula-EH}), the
formula (\ref{formula-sumchi}) is obtained.
\end{proof}

In rest of this section, we will show that the Conjecture
\ref{conjecture-Alexander} is closely related to the Conjecture
\ref{conjecture-Specailcase} which is implied by the refined LMOV
integrality conjecture for framed links.

Recall the definition of the colored Alexander polynomial
\begin{align}
A_{(m|n)}(\mathcal{K};q)&=\lim_{a\rightarrow
1}\frac{q^{-(m-n)dw(\mathcal{K})}a^{-dw(\mathcal{K})}\mathcal{H}(\mathcal{K}\star
Q_{(m|n)};q,a)}{s_{(m|n)}(\mathbf{x})}\\\nonumber
&=q^{-(m-n)dw(\mathcal{K})}(-1)^nd(q^{d}-q^{-d})\lim_{a\rightarrow
1}\frac{\mathcal{H}(\mathcal{K}\star Q_{(m|n)};q,a)}{(a^d-a^{-d})}.
\end{align}

By the conjectural formula (\ref{conjecture-Alexander}), we obtain
\begin{align} \label{formula-lima}
\lim_{a\rightarrow 1}\frac{\mathcal{H}(\mathcal{K}\star
Q_{(m|n)};q,a)}{(a^d-a^{-d})}&=(-1)^n\frac{q^{(m-n)dw(\mathcal{K})}}{d(q^d-q^{-d})}A_{(m|n)}(\mathcal{K};q)\\\nonumber
&=(-1)^n\frac{q^{(m-n)dw(\mathcal{K})}}{d(q^d-q^{-d})}A(\mathcal{K};q^d).
\end{align}

Then, by using the Frobenius formula (\ref{formula-frobeniusQ}) and
the formula (\ref{formula-chid}), we obtain
\begin{align}
\lim_{a\rightarrow 1}\frac{\mathcal{H}(\mathcal{K}\star
P_d;q,a)}{a^{d}-a^{-d}}&=\lim_{a\rightarrow
1}\sum_{m+n+1=d}\frac{(-1)^n\mathcal{H}(\mathcal{K}\star
Q_{(m|n)};q,a)}{a^{d}-a^{-d}}\\\nonumber
&=\frac{A(\mathcal{K};q^d)}{d(q^d-q^{-d})}\sum_{m+n+1=d}q^{(m-n)dw(\mathcal{K})}.
\end{align}

On the other hand side, for the definition of  Alexander polynomial,
we have
\begin{align}
A(\mathcal{K};q)=(q-q^{-1})\lim_{a\rightarrow
1}\frac{\mathcal{H}(\mathcal{K};q,a)}{a-a^{-1}}.
\end{align}
Then
\begin{align}
\lim_{a\rightarrow
1}\Psi_{d}\left(\frac{\mathcal{H}(\mathcal{K};q,a)}{a-a^{-1}}\right)=\frac{A(\mathcal{K};q^d)}{(q^d-q^{-d})}.
\end{align}

Therefore, we obtain
\begin{align} \label{formula-limH}
&\lim_{a\rightarrow 1}\left(\frac{\mathcal{H}(\mathcal{K}\star
P_d;q,a)}{(a^{d}-a^{-d})}-(-1)^{(d-1)w(\mathcal{K})}\Psi_{d}\left(\frac{\mathcal{H}(\mathcal{K};q,a)}{a-a^{-1}}\right)\right)\\\nonumber
&=\frac{A(\mathcal{K};q^d)}{d(q^d-q^{-d})}\left(\sum_{m+n+1=d}q^{(m-n)dw(\mathcal{K})}-(-1)^{(d-1)w(\mathcal{K})}\right).
\end{align}

For any prime $p$, let
\begin{align}
G_p(\mathcal{K};q,a)=\frac{\mathcal{H}(\mathcal{K}\star
P_p;q,a)}{a-a^{-1}}-(-1)^{(p-1)w(\mathcal{K})}\frac{\Psi_{d}\left(\mathcal{H}(\mathcal{K};q,a)\right)}{a-a^{-1}}.
\end{align}
The Conjecture \ref{conjecture-Specailcase} states that,
\begin{align} \label{formula-tildeg0}
\tilde{g}_p^{(0)}(\mathcal{K};q,a)=\frac{G_p(\mathcal{K};q,a)}{\{p\}}\in
z^{-2}\mathbb{Z}[z^2,a].
\end{align}

We have the following theorem which supposes the conjectural formula
(\ref{formula-tildeg0}).
\begin{theorem}
If the Conjecture \ref{conjecture-Alexander} holds, then we have
\begin{align} \label{formula-tildegp}
\tilde{g}_p^{(0)}(\mathcal{K};q,1)\in z^{-2}\mathbb{Z}[z^2].
\end{align}
\end{theorem}
\begin{proof}
By formula  (\ref{formula-limH}), we obtain
\begin{align}
G_p(\mathcal{K};q,1)&=\lim_{a\rightarrow
1}\left(\frac{\mathcal{H}(\mathcal{K}\star
P_p;q,a)}{a-a^{-1}}-(-1)^{(p-1)w(\mathcal{K})}\frac{\Psi_{p}\left(\mathcal{H}(\mathcal{K};q,a)\right)}{a-a^{-1}}\right)\\\nonumber
&=p\lim_{a\rightarrow 1}\left(\frac{\mathcal{H}(\mathcal{K}\star
P_p;q,a)}{(a^p-a^{-p})}-(-1)^{(p-1)w(\mathcal{K})}\frac{\Psi_{p}\left(\mathcal{H}(\mathcal{K};q,a)\right)}{(a^p-a^{-p})}\right)\\\nonumber
&=\frac{A(\mathcal{K};q^p)}{(q^p-q^{-p})}\left(\sum_{m+n+1=p}q^{(m-n)dw(\mathcal{K})}-(-1)^{(p-1)w(\mathcal{K})}\right).
\end{align}

By Lemma \ref{Lemma-alpha}, there is a function
\begin{align}
\alpha_{p}^{w(\mathcal{K})}(q)\in \mathbb{Z}[z^2],
\end{align}
such that
\begin{align}
\sum_{m+n+1=p}q^{(m-n)dw(\mathcal{K})}-(-1)^{(p-1)w(\mathcal{K})}=\frac{\{p\}^2}{z^2}\alpha_{p}^{w(\mathcal{K})}(q).
\end{align}

Therefore,
\begin{align}
\tilde{g}_{p}^{(0)}(\mathcal{K};q,1)=\frac{G_p(\mathcal{K};q,a)}{\{p\}}=z^{-2}A(\mathcal{K};q^p)\alpha_{p}^{w(\mathcal{K})}(q).
\end{align}

Moreover, by using the properties for Alexander polynomial in our
notation, we have
\begin{align}
A(\mathcal{K};-q^p)=A(\mathcal{K};q^p), \
A(\mathcal{K};q^{-p})=A(\mathcal{K};q^p),
\end{align}
which implies that
\begin{align}
A(\mathcal{K};q^p)\in \mathbb{Z}[z^2].
\end{align}
Then we finish the proof of the formula (\ref{formula-tildegp}).
\end{proof}

\begin{lemma} \label{Lemma-alpha}
For any prime number $p$ and integer $\tau\in \mathbb{Z}$, there is
a function $\alpha_{p}^{\tau}(z)\in \mathbb{Z}[z^2]$, such that
\begin{align}
(q^{p\tau})^{p-1}+(q^{p\tau})^{p-3}+\cdots+(q^{p\tau})^{-(p-1)}-p(-1)^{(p-1)\tau}=[p]^2
\alpha_{p}^{\tau}(z).
\end{align}
\end{lemma}
\begin{proof}
First, we prove the case when $p=2$. In this case, we only need to
show there is a function $\alpha_{2}^{\tau}(z)$ in the ring
$\mathbb{Z}[z^2]$ such that
$q^{2\tau}+q^{-2\tau}-2(-1)^\tau=(q^2+q^{-2}+2)\alpha_{2}^{\tau}(z)$
for any $\tau\in \mathbb{N}$. It clearly holds when $\tau=0$ and
$\tau=1$. For $\tau\geq 2$, we prove it by induction. Suppose it
holds when $s\leq \tau$, now for $s=\tau+1$, indeed, we have the
identity
\begin{align}
&q^{2(\tau+1)}+q^{-2(\tau+1)}-2(-1)^{\tau+1}\\\nonumber
&=\left(q^{2\tau}+q^{-2\tau}-2(-1)^{\tau}\right)\left(q^2+q^{-2}\right)\\\nonumber
&-\left(q^{2(\tau-1)}+q^{-2(\tau-1)}-2(-1)^{\tau-1}\right)
+2(-1)^{\tau}\left(q^{2}+q^{-2}+2\right).
\end{align}
Therefore, by induction, we have
\begin{align}
&q^{2(\tau+1)}+q^{-2(\tau+1)}-2(-1)^{\tau+1}\\\nonumber
&=(q^2+q^{-2}+2)\left(\alpha_2^{\tau}(z)(z^2+2)-\alpha_2^{\tau-1}(z)+2(-1)^{\tau}\right),
\end{align}
i.e. we obtain
\begin{align}
\alpha_2^{\tau+1}(z)=\alpha_2^{\tau}(z)(z^2+2)-\alpha_2^{\tau-1}(z)+2(-1)^{\tau}.
\end{align}

When $p$ is an odd prime, given $n\in \mathbb{N}$, we introduce the
function
\begin{align} \label{Qn}
Q_n(x)=\frac{x^{n}-x^{-n}}{x-x^{-1}}=x^{n-1}+x^{n-3}+\cdots+x^{-(n-3)}+x^{-(n-1)}.
\end{align}
Then we consider the following function
\begin{align}
f_p^{\tau}(x)=\frac{(x^{\tau p})^{p-1}+(x^{\tau
p})^{p-3}\cdots+(x^{\tau p})^{-(p-3)}+(x^{\tau
p})^{-(p-1)}-p}{Q_p(x)^2}.
\end{align}

Let
\begin{align}
g_p^{\tau}(x)=(x^{\tau p})^{p-1}+(x^{\tau p})^{p-3}\cdots+(x^{\tau
p})^{-(p-3)}+(x^{\tau p})^{-(p-1)}-p.
\end{align}
Since all the roots of $Q_p(x)$ are given by
$\exp\left(\frac{\pi\sqrt{-1}j}{p}\right)$ for
$j=1,2,..,p-1,p+1,..,2p-1$. For any  root $x_0$ of $Q_p(x)$, by
direct computations, it is easy to obtain $g_p^{\tau}(x_0)=0$ and
$(g_{p}^{\tau})'(x_0)=0$. Therefore, $x_0$ is a double root of the
polynomial $g_p(x)$, then it is easy to see that $f_p^{\tau}(x)$ is
a Laurents polynomial with integral coefficients which can be
written as
\begin{align}
f_p^{\tau}(x)=c_{-i}x^{-i}+\cdots +c_0+\cdots c_{j}x^j.
\end{align}
Moreover, by the definition of $f_p^{\tau}(x)$, it is obvious that
\begin{align}
f_p^{\tau}(x)&=f_p^{\tau}(x^{-1}) \\
f_p^{\tau}(-x)&=f_p^{\tau}(x),
\end{align}
 So we must have $i=j$, $c_k=c_{-k}$ and $c_k=0$ for $k$
odd. Finally, it is easy to obtain that $f_p^{\tau}(x)\in
\mathbb{Z}[(x-x^{-1})^2]$.
\end{proof}

\section{Appendices}  \label{section-Appendix}
\subsection{Partitions and symmetric functions}
In this subsection, we fix the notations related to partitions and
symmetric functions used in this paper.

A partition $\lambda$ is a finite sequence of positive integers $%
(\lambda_1,\lambda_2,..)$ such that $\lambda_1\geq
\lambda_2\geq\cdots$. The length of $\lambda$ is the total number of
parts in $\lambda$ and denoted by
$l(\lambda)$. The weight of $\lambda$ is defined by $|\lambda|=%
\sum_{i=1}^{l(\lambda)}\lambda_i$. If $|\lambda|=d$, we say
$\lambda$ is a partition of $d$ and denoted as $\lambda\vdash d$.
The automorphism group of $\lambda$, denoted by Aut($\lambda$),
contains all the permutations that
permute parts of $\lambda$ by keeping it as a partition. Obviously, Aut($%
\lambda$) has the order $|\text{Aut}(\lambda)|=\prod_{i=1}^{l(\lambda)}m_i(%
\lambda)! $ where $m_i(\lambda)$ denotes the number of times that
$i$ occurs in $\lambda$.

Every partition is identified with a Young diagram. The Young diagram of $%
\lambda$ is a graph with $\lambda_i$ boxes on the $i$-th row for $%
j=1,2,..,l(\lambda)$, where we have enumerated the rows from top to
bottom and the columns from left to right. Given a partition
$\lambda$, we define the conjugate partition $\lambda^\vee$ whose
Young diagram is the transposed
Young diagram of $\lambda$: the number of boxes on $j$-th column of $%
\lambda^t$ equals to the number of boxes on $j$-th row of $\lambda$, for $%
1\leq j\leq l(\lambda)$. For the box in the $i$-th row and $j$-th
column of $\lambda$, we write $(i,j)\in \lambda$, and refer to
$(i,j)$ as the coordinates of the box. For $x=(i,j)\in \lambda$ we
define $cn(x)=j-i$ and $hl(x)=\lambda_i+\lambda_j^{\vee}-i-j+1$.

The following numbers associated with a given partition $\lambda$
are used frequently in this article:
\begin{align*}
\mathfrak{z}_\lambda=\prod_{j=1}^{l(\lambda)}j^{m_{j}(\lambda)}m_j(\lambda)!
\quad \text{and} \quad
k_{\lambda}=\sum_{j=1}^{l(\lambda)}\lambda_j(\lambda_j-2j+1).
\end{align*}
Obviously, $k_\lambda$ is an even number and
$k_\lambda=-k_{\lambda^\vee}$.

In the following, we will use the notation $\mathcal{P}_+$ to denote
the set of all the partitions of positive integers. Let $0$ be the
partition of $0$, i.e. the empty partition. Define
$\mathcal{P}=\mathcal{P}_+\cup \{0\}$, and $\mathcal{P}^L$ the $L$
tuple of $\mathcal{P}$.

The power sum symmetric function of infinite variables
$x=(x_1,..,x_N,..)$ is defined by $p_{n}(\mathbf{x})=\sum_{i}x_i^n.
$ Given a partition $\lambda$, define
$p_\lambda(\mathbf{x})=\prod_{j=1}^{l(\lambda)}p_{\lambda_j}(\mathbf{x}).
$ The Schur function $s_{\lambda}(\mathbf{x})$ is determined by the
Frobenius formula
\begin{align}  \label{Frobeniusformula}
s_\lambda(\mathbf{x})=\sum_{\mu}\frac{\chi_{\lambda}(\mu)}{\mathfrak{z}_\mu}p_\mu(\mathbf{x}).
\end{align}
where $\chi_\lambda$ is the character of the irreducible
representation of
the symmetric group $S_{|\lambda|}$ corresponding to $\lambda$, we have $%
\chi_{\lambda}(\mu)=0$ if $|\mu|\neq |\lambda|$. The orthogonality
of character formula gives
\begin{align}  \label{orthog}
\sum_\lambda\frac{\chi_\lambda(\mu)
\chi_\lambda(\nu)}{\mathfrak{z}_\mu}=\delta_{\mu \nu}.
\end{align}

\begin{align}
\sum_{\lambda}\chi_{\lambda}(\mu)\chi_{\lambda}(\nu)=\mathfrak{z}_{\mu}\delta_{\mu,\nu}
\end{align}

We introduce the multiple-index variable
$\vec{\mathbf{x}}=(\mathbf{x}^1,...,\mathbf{x}^L)$ where
$\mathbf{x}^{\alpha}=(\mathbf{x}^\alpha_1,\mathbf{x}^{\alpha}_2,...)$
for $\alpha=1,...,L$.

Given  $\vec{\lambda}=(\lambda^1,...,\lambda^L),
\vec{\mu}=(\mu^1,...,\mu^L) \in \mathcal{P}^L$, we introduce the
following notations
\begin{align}
[\vec{\lambda},\vec{\mu}]&=([\lambda^{1},\mu^{1}],...,[\lambda^L,\mu^L]),
\ \
\mathfrak{z}_{\vec{\mu}}=\prod_{\alpha=1}^L\mathfrak{z}_{\mu^\alpha},
\ \ \{\vec{\mu}\}=\prod_{\alpha=1}^L\{\mu^{\alpha}\} \\
\vec{\mu}^{t}&=((\mu^1)^t,...,(\mu^L)^t), \ \
\chi_{\vec{\lambda}}(\vec{\mu})=\prod_{\alpha=1}^L\chi_{\lambda^\alpha}(\mu^{\alpha})\\
 s_{\vec{\lambda}}(\vec{\mathbf{x}})&=\prod_{\alpha=1}^L
s_{\lambda^\alpha}(\mathbf{x}^\alpha), \ \
p_{\vec{\mu}}(\vec{\mathbf{x}})=\prod_{\alpha=1}^L
p_{\mu^\alpha}(\mathbf{x}^\alpha).
\end{align}

Now, we consider the set $\mathcal{P}^L$, one can define the order
of $\mathcal{P}^L$ as follow, for any $\vec{\lambda}, \vec{\mu}\in
\mathcal{P}^{L}$, $\vec{\lambda}\geq \vec{\mu}$ if and only if
$\sum_{\alpha=1}^{L}|\lambda^{\alpha}|>\sum_{\alpha=1}^L|\mu^{\alpha}|$,
or
$\sum_{\alpha=1}^{L}|\lambda^{\alpha}|=\sum_{\alpha=1}^L|\mu^{\alpha}|$
and there is a $\beta$ such that $\lambda^{\alpha}=\mu^{\alpha}$ for
$\alpha<\beta$ and $\lambda^{\beta}>\mu^{\beta}$. Define
\begin{align}
\vec{\lambda}\cup \vec{\mu}=(\lambda^1\cup \mu^1,...,\lambda^L\cup
\mu^L).
\end{align}
and $(\emptyset,...,\emptyset)$ is the empty element. Then
$\mathcal{P}^L$ is a partitionable set (cf. \cite{LP10}). For a
partitionable set $\mathcal{S}$, one can define partition with
respect to $\mathcal{S}$, a finite sequence of nonincreasing
non-minimum elements in $\mathcal{S}$. We will call it an
$\mathcal{S}$-partitions.

Let $\mathcal{P}(\mathcal{P}^L)$ be the set of all
$\mathcal{P}^L$-partitions. For a $\mathcal{P}^L$-partition
$\Lambda$, denoted by $l(\Lambda)$ the length of $\Lambda$.  One can
also define the automorphism group of $\Lambda$ as follow. Given
$\Omega\in \Lambda$, denote by $m_{\Omega}(\Lambda)$ the number of
times that $\Omega$ occurs in the parts of $\Lambda$, then we have
\begin{align}
|Aut(\Lambda)|=\prod_{\Omega\in\Lambda}m_{\Omega}(\Lambda)!.
\end{align}
We define the following quantity associated with $\Lambda$,
\begin{align}
\Theta_{\Lambda}=(-1)^{l(\Lambda)-1}\frac{(l(\Lambda)-1)!}{|Aut(\Lambda)|}.
\end{align}
We have
\begin{align} \label{formula-logZ}
\log\left(\sum_{\vec{\lambda}\in
\mathcal{P}^L}Z_{\vec{\lambda}}p_{\vec{\lambda}}(\vec{\mathbf{x}})\right)=\sum_{\Lambda\in
\mathcal{P}(\mathcal{P}^L)}Z_{\Lambda}\Theta_{\Lambda}p_{\Lambda}(\vec{\mathbf{x}})
\end{align}

\subsection{Quantum group invariants and colored HOMFLY-PT invariants}
In this section, we briefly review the definition of quantum group
invariant of link following \cite{LaMa02} and \cite{LiZh10}.
\subsubsection{Quantum group invariants}
Let $\mathfrak{g}$ be a complex simple Lie algebra and let $q$ be a
nonzero complex numbers which is not a root of unity. Let $U_{
q}(\mathfrak{g})$ be the quantum enveloping algebra of
$\mathfrak{g}$. The ribbon category structure of the set of finite
dimensional $U_{q}(\mathfrak{g})$-modules provides the following
objects.

1. For each pair of $U_{q}(\mathfrak{g})$-modules $V, W$, there is a
natural isomorphism $\check{R}_{V,W}: V\otimes W\rightarrow W\otimes
V$ such that
\begin{align} \label{formula-RUV}
\check{R}_{U\otimes V,W}&=(\check{R}_{U,W}\otimes
id_{V})(id_{U}\otimes \check{R}_{V,W}), \\\nonumber
\check{R}_{U,V\otimes W}&=(id_{V}\otimes
\check{R}_{U,W})(\check{R}_{U,V}\otimes id_{W})
\end{align}
hold for all $U_{q}(\mathfrak{g})$-modules $U,V,W$. The naturality
means
\begin{align}
(y\otimes x)\check{R}_{V,W}=\check{R}_{V',W'}(x\otimes y)
\end{align}
for $x\in Hom_{U_q(\mathfrak{g})}(V,V'), y\in
Hom_{U_{q}(\mathfrak{g})}(W,W')$. These equalities imply the
braiding relation
\begin{align}
(\check{R}_{V,W}\otimes id_{U})(id_{V}\otimes
\check{R}_{U,W})(\check{R}_{U,V}\otimes id_{W})=(id_W\otimes
\check{R}_{U,V})(\check{R}_{U,W}\otimes id_{V})(id_{U}\otimes
\check{R}_{V,W}).
\end{align}

2. There exists an element $K_{2\rho}\in U_{q}(\mathfrak{g})$, where
$\rho$ denotes the half-sum of all positive roots of $\mathfrak{g}$
such that
\begin{align}
K_{2\rho}(v\otimes w)=K_{2\rho}(v)\otimes K_{2\rho}(w)
\end{align}
for $v\in W, w\in W$. Furthermore, for any $z\in
End_{U_{q}(\mathfrak{g})}(V\otimes W)$ with $z\in \sum_{i}x_i\otimes
y_i$, $x\in End(V)$, $y_i\in End(W)$ one has the partial quantum
trace
\begin{align}
tr_W(z)=\sum_{i}tr(y_i K_{2\rho})\cdot x_i\in
End_{U_{q}(\mathfrak{g})}(V).
\end{align}

3. For every $U_q(\mathfrak{g})$-module $V$ there is a natural
isomorphism $\theta_V: V\rightarrow V$ satisfying
\begin{align}
\theta_{V}^{\pm 1}=tr_{V}\check{R}_{V,V}^{\pm 1}.
\end{align}
The naturality means $x\cdot \theta_{V}=\theta_{V'}\cdot x$ for
$x\in Hom_{U_{q}(\mathfrak{g})}(V,V')$.

In the following, we only consider the finite-dimensional
irreducible representation of $U_{q}(\mathfrak{g})$. These
representations are labeled by highest weights $\Lambda$, the
corresponding module will be denoted by $V_\Lambda$. In this case
\begin{align}
\theta_{V_{\Lambda}}=q^{(\Lambda,\Lambda+2\rho)}id_{V_{\Lambda}}
\end{align}

Let $B_n$ be the braid group of $n$ strands that is generated by
$\sigma_1,...,\sigma_{n-1}$ with two defining relations: (i)
$\sigma_i\sigma_j=\sigma_j\sigma_i$ if $|i-j|\geq 2$;  (ii)
$\sigma_i\sigma_j\sigma_i=\sigma_j\sigma_i\sigma_j$ if $|i-j|=1$.
Every link can be presented by the closure of some element, i.e.
braid in $B_n$ for some $n$. This kind of braid representation of a
link is not unique. In the following,  we first fix a braid
representation of a link, and then we construct the quantum group
invariant of the link by this braid, and finally it turn out that
such construction is independent of the choice of the braid
representation.

Let $\mathcal{L}$ be an oriented link with $L$-components
$\mathcal{K}_1,...,\mathcal{K}_L$ label by  $L$ irreducible
$U_{q}(\mathfrak{g})$-modules $V_{\Lambda^1}$,...,$V_{\Lambda^L}$
respectively. Choose a closed braid representation $\hat{b}$ of
$\mathcal{L}$ with $\beta\in B_n$ for some $n$. Then the $j$-th
strand in the braid $b$ we will associate an irreducible module
$V_{\Lambda^{\alpha_j}}$ if the $j$-th stand belongs to the
component $\mathcal{K}_{\alpha_j}$ with $\alpha_j\in \{1,2,...,L\}$.
In this way, the braid $b$ will be associated with a total module
$V_{\Lambda^{\alpha_1}}\otimes V_{\Lambda^{\alpha_2}}\otimes \cdots
\otimes V_{\Lambda^{\alpha_n}}$. The representation  of $b$ on
$V_{\Lambda^{\alpha_1}}\otimes V_{\Lambda^{\alpha_2}}\otimes \cdots
\otimes V_{\Lambda^{\alpha_n}}$ is defined as follows. If
$\sigma_i^{\pm 1}$ is an elementary braid, then
\begin{align}
\Phi(\sigma_i^{\pm 1})=id_{V_{\Lambda^{\alpha_1}}}\otimes \cdots
\otimes \check{R}_{V_{\Lambda^{\alpha_i}}\otimes
V_{\Lambda^{\alpha_{i+1}}}}^{\pm 1}\otimes \cdots \otimes
id_{V_{\Lambda^{\alpha_n}}}.
\end{align}
In this way, the braid $b$ gives an operator
\begin{align}
\Phi(b)\in End_{U_q(\mathfrak{g})}(V_{\Lambda^{\alpha_1}}\otimes
\cdots \otimes V_{\Lambda^{\alpha_n}}).
\end{align}
Due to the fundamental works \cite{Tu88-1,RT90} of Reshetikhin and
Turaev, the quantum trace
\begin{align}
tr_{V_{\Lambda^{\alpha_1}}\otimes \cdots \otimes
V_{\Lambda^{\alpha_n}}}(\Phi_{V_{\Lambda^{\alpha_1}},...,V_{\Lambda^{\alpha_n}}}(b))
\end{align}
gives a framing dependent link invariant for $\mathcal{L}$.

In order to eliminate the framing dependency, let
$w(\mathcal{K}_\alpha)$ be the {\em writhe} of $\mathcal{K}_\alpha$
in $\beta$, which is equal to the number of positive crossings minus
the number of negative crossings. Then the quantity
\begin{align} \label{formula-Ilambda}
I_{\Lambda^1,...,\Lambda^{L}}^{\mathfrak{g}}(\mathcal{L})=q^{-\sum_{\alpha=1}^{L}w(\mathcal{K}_\alpha)
(\lambda^\alpha,\lambda^\alpha+2\rho)}tr_{V_{\Lambda^{\alpha_1}}\otimes
\cdots \otimes V_{\Lambda^{\alpha_n}}}(\Phi(b))
\end{align}
defines a framing independent link invariant.

\subsubsection{$U_{q}(sl_{N}\mathbb{C})$ quantum group invariants}
Now, we consider the semi-simple Lie algebra
$\mathfrak{g}=sl_N\mathbb{C}$, it is well-known that every finite
dimensional irreducible module $V_{\lambda}$ of
$U_{\hbar}(sl_N\mathbb{C})$ is labeled by a partition $\lambda$ with
length $l(\lambda)\leq N-1$. In this case, it is straightforward to
compute that
\begin{align}
(\lambda,\lambda+2\rho)=k_\lambda+|\lambda|N-\frac{|\lambda|^2}{N},
\end{align}
where $k_\lambda=\sum_{i=1}^{l(\lambda)}\lambda_i(\lambda_i-2i+1)$.

The $U_q(sl_N\mathbb{C})$ quantum group invariants which we will
study is defined as follow (cf. \cite{LaMa02}):
\begin{align}
&W_{\lambda^1,..,\lambda^{L}}(\mathcal{L})\\\nonumber&=q^{\frac{2}{N}\sum_{\alpha<\beta}
lk(\mathcal{K}_\alpha,\mathcal{K}_{\beta})|\lambda^{\alpha}|\cdot|\lambda^{\beta}|}
I_{\lambda^1,...,\lambda^L}^{sl_{N}\mathbb{C}}(\mathcal{L})\\\nonumber
&=q^{-\sum_{\alpha=1}^{L}w(\mathcal{K}_\alpha)
(\kappa_{\lambda^{\alpha}}+|\lambda^{\alpha}|N-\frac{|\lambda^{\alpha}|^2}{N})+\frac{2}{N}\sum_{\alpha<\beta}
lk(\mathcal{K}_\alpha,\mathcal{K}_{\beta})|\lambda^{\alpha}|\cdot|\lambda^{\beta}|}tr_{V_{\lambda^{\alpha_1}}\otimes
\cdots \otimes V_{\lambda^{\alpha_n}}}(\Phi(b)).
\end{align}
\begin{remark}
The original invariant
$I_{\lambda^1,...,\lambda^{L}}^{sl_N(\mathbb{C})}(\mathcal{L})$
defined by formula (\ref{formula-Ilambda}) is a rational function in
$q^{\frac{1}{N}}$, $q$ and $q^N$. In order
 to cancel the overall powers of $q^{\frac{1}{N}}$ that
appear in
$I_{\lambda^1,...,\lambda^{L}}^{sl_N(\mathbb{C})}(\mathcal{L})$ and
simplify the notation, the term $q^{\frac{2}{N}\sum_{\alpha<\beta}
lk(\mathcal{K}_\alpha,\mathcal{K}_{\beta})|\lambda^{\alpha}|\cdot|\lambda^{\beta}|}$
was introduced in the above definition, see \cite{LaMa02}. In this
way, the $U_q(sl_N\mathbb{C})$ quantum group invariant
$W_{\lambda^1,..,\lambda^{L}}(\mathcal{L})$ becomes a rational
function of $q$ and $q^N$.
\end{remark}
\begin{definition}[cf. \cite{LiZh10}]
The {\em colored HOMFLY-PT invariants}
$W_{\lambda^1,...,\lambda^{L}}(\mathcal{L};q,a)$ of the link
$\mathcal{L}$ is a rational function of $q$ and $a$ defined by
\begin{align} \label{formula-coloredhomfly}
W_{\lambda^1,...,\lambda^{L}}(\mathcal{L};q,a)=q^{2\sum_{\alpha<
\beta}|\lambda^{\alpha}||\lambda^{\beta}|lk(\mathcal{K}_{\alpha},\mathcal{K}_{\beta})/N}
I_{\lambda^1,...,\lambda^{L}}^{sl_N(\mathbb{C})}(\mathcal{L})|_{q^{N}=a}.
\end{align}
\end{definition}
\begin{remark}
The above definition of colored HOMFLY-PT invariant does not depend
on the choice of the sufficiently large integer $N$. In other words,
given a link $\mathcal{L}$, there is an integer $N_{0}$ (which
depends on $\mathcal{L}$), such that for any $N\geq N_0$,
substituting $q^N$ with $a$ in formula (\ref{formula-coloredhomfly})
will yield the same result. We refer to \cite{Tu88-2} (cf. Lemma
4.2.4 in \cite{Tu88-2}) for the statement about how to define the
HOMFLY-PT polynomial via quantum group $U_q(sl_N\mathbb{C})$.
\end{remark}

\subsubsection{HOMFLY-PT polynomials}
In the following, we first introduce the {\em variant HOMFLY-PT
skein relations}
\begin{align} \label{formula-vskein1}
x^{-1}\mathcal{L}_{+}-x^{-1}\mathcal{L}_{-}=(q-q^{-1})\mathcal{L}_0,
\end{align}
\begin{align} \label{formula-vskein2}
\mathcal{L}(+1)=(xa)\mathcal{L},
\end{align}
to define the variant skein algebra $Sk^{(v)}(F)$ of a surface $F$,
where $\mathcal{L}(+)$ denote the link which is constructed by
adding a positive kink to $\mathcal{L}$.

 Given a framed link $\mathcal{L}$, the variant skein algebra
$Sk^{(v)}(\mathbb{R}^2)$ of the plane $\mathbb{R}^2$ will lead to
the {\em variant framed HOMPLY-PT polynomial}
$\mathcal{H}^{(v)}(\mathcal{L})$ of $\mathcal{L}$ which is
polynomial of $x,q,a$. Let $w(\mathcal{L})$ be the writhe number of
$\mathcal{L}$, a key observation is that the framing independent
HOMPLY-PT polynomial
$(xa)^{-w(\mathcal{L})}\mathcal{H}^{(v)}(\mathcal{L})$ does not
involve the variable $x$, i.e.
\begin{align} \label{formula-homfly-variant}
(xa)^{-w(\mathcal{L})}\mathcal{H}^{(v)}(\mathcal{L})=a^{-w(\mathcal{L})}\mathcal{H}(\mathcal{L}).
\end{align}

We also have the variant Hecke algebra $H_{n}^{(v)}(x,q,a)$ which is
the variant skein algebra $Sk^{(v)}(R_{n}^{n})$ of the rectangle
$R_{n}^{n}$. Correspondingly,  the idempotent $y_{\lambda}^{(v)}$ in
$H_{n}^{(v)}(x,q,a)$ are given by
\begin{align} \label{formula-idempotent-variant}
y_{\lambda}^{(v)}=\frac{1}{\alpha_\lambda}E_{\lambda}^{(v)}(a)\omega_{\pi_{\lambda}}E_{\lambda^{\vee}}^{(v)}(b)\omega_{\pi_{\lambda}}^{-1},
\end{align}
where $E_{\lambda}^{(v)}(a)=a^{(v)}_{\lambda_1}\otimes
a^{(v)}_{\lambda_2}\otimes \cdots\otimes
a^{(v)}_{\lambda_{l(\lambda)}}$ and
$E_{\mu}^{(v)}(b)=b^{(v)}_{\mu_1}\otimes b^{(v)}_{\mu_2}\otimes
\cdots\otimes b^{(v)}_{\mu_{l(\mu)}}$ for any partitions $\lambda,
\mu$, with
\begin{align}
a^{(v)}_{n}=\sum_{\pi\in S_n}(x^{-1}q)^{l(\pi)}\omega_{\pi} \ \
\text{and} \ \  b^{(v)}_{n}=\sum_{\pi\in
S_n}(-x^{-1}q^{-1})^{l(\pi)}\omega_{\pi} \in H^{(v)}_n(x,q,a).
\end{align}

Let $V$ be the module of the fundamental representation of
$U_q(sl_N\mathbb{C})$. Let $\{K_{i}^{\pm 1}, E_i, F_i|1\leq i\leq
N-1\}$ be the set of the generators of $U_{q}(sl_N\mathbb{C})$, with
a suitable basis $\{v_1,...,v_N\}$ of $V$, the fundamental
representation is given by the matrices
\begin{align}
K_i&=qE_{i,i}+q^{-1}E_{i+1,i+1}+\sum_{j\neq i}E_{j,j}\\\nonumber
E_i&=E_{i,i+1}\\\nonumber F_i&=E_{i+1,i}
\end{align}
where $E_{i,j}$ is the $N$ by $N$ matrix with $1$ in the
$(i,j)$-position and $0$ elsewhere. The action of the universal
$R$-matrix on $V\otimes V$ is given by
\begin{align}
\check{R}_{V,V}=q^{-\frac{1}{N}}\left(q\sum_{1\leq i\leq
N}E_{i,i}\otimes E_{i,i}+\sum_{1\leq i\neq j\leq N}E_{j,i}\otimes
E_{i,j}+(q-q^{-1})\sum_{1\leq i<j\leq N} E_{j,j}\otimes
E_{i,i}\right).
\end{align}
It gives
\makeatletter
\let\@@@alph\@alph
\def\@alph#1{\ifcase#1\or \or $'$\or $''$\fi}\makeatother
\begin{subnumcases}
{q^{\frac{1}{N}}\check{R}_{V,V}(v_i\otimes v_j)=}
q(v_i\otimes v_j), &$i=j$,\nonumber \\
v_j\otimes v_i, &$i<j$, \nonumber \\
v_j\otimes v_i+(q-q^{-1})v_i\otimes v_j & $i>j$.
\end{subnumcases}
\makeatletter\let\@alph\@@@alph\makeatother

Then, it is straightforward to verify the following identity
\begin{align} \label{formula-Rrelation1}
q^{\frac{1}{N}}\check{R}_{V,V}-q^{-\frac{1}{N}}\check{R}^{-1}_{V,V}=(q-q^{-1})id_{V\otimes
V}.
\end{align}

Moreover, we have
\begin{align} \label{formula-Rrelation2}
\theta_{V}^{\pm 1}=q^{\pm (N-\frac{1}{N})}id_{V}.
\end{align}

Given a framed link $\mathcal{L}$, we color all of its components by
the fundamental representation $V$. We choose a closed braid
representation $\hat{b}$ of $\mathcal{L}$ with $b\in B_{n}$ for some
$n$. Then every strand of $b$ was labeled by the fundamental
representation $V$. Since
\begin{align}
tr_{V^{\otimes n}}\Phi(b)
\end{align}
is a framed link invariant of $\mathcal{L}$, and it satisfies the
relations (\ref{formula-Rrelation1}) and (\ref{formula-Rrelation2}),
comparing with the variant skein relations (\ref{formula-vskein1})
and (\ref{formula-vskein2}), we obtain
\begin{proposition} \label{proposition-tracehomfly}
The framed link invariant $tr_{V^{\otimes n}}\Phi(b)$ is equal to
the variant HOMFLY-PT polynomial $\mathcal{H}^{(v)}(\mathcal{L})$ of
$\mathcal{L}$ after the substitutions of $x$ by $q^{-\frac{1}{N}}$
and $a$ by $q^{N}$, i.e.
\begin{align}
tr_{V^{\otimes
n}}\Phi(b)=\mathcal{H}^{(v)}(\mathcal{L})|_{x=q^{-\frac{1}{N}},a=q^{N}}.
\end{align}
\end{proposition}

Let $\tilde{H}^{N}_{n}(q)$ be the specialization of $H^{(v)}(x,q,a)$
by substituting with $x=q^{-\frac{1}{N}}$, $a=q^{N}$. Then by
relation (\ref{formula-Rrelation1}), the homomorphism $\Phi:
\mathbb{C}B_n\rightarrow End_{U_{q}(sl_N)}(V^{\otimes n})$ actually
gives a homomorphism $\varphi$ from $\tilde{H}^{N}_{n}(q)$ to
$End_{U_{q}(sl_N)}(V^{\otimes n})$ via
$q^{\frac{1}{N}}\sigma_i\rightarrow q^{\frac{1}{N}}\Phi(\sigma_i)$.
Let $\tilde{y}_{\lambda}^{(v)}\in \tilde{H}^{N}_{n}(q)$ be the
specialization of the idempotent $y_{\lambda}^{(v)}$ in
$H_{n}^{(v)}(x,q,a)$.

Let  $\mathcal{L}$ be a framed link with $L$ components
$\mathcal{K}_1,...,\mathcal{K}_L$, and given $L$ partitions
$\lambda^1,...,\lambda^L$. Suppose $b\in B_{m}$ is a braid
representation of $\mathcal{L}$, and the $j$-th strand of $b$
belongs to the component $\mathcal{K}_{\alpha_j}$ for $\alpha_j\in
\{1,2,...,L\}$. Let
$b^{(|\lambda^{\alpha_1}|,...,|\lambda^{\alpha_m}|)}\in B_n$ be the
braid obtained by cabling the $j$-th strand of $b$ to
$|\lambda^{\alpha_j}|$ parallel ones, where
$n=|\lambda^{\alpha_1}|+\cdots+|\lambda^{\alpha_m}|$.  By this
construction, the closure of the braid
$b^{(|\lambda^{\alpha_1}|,...,|\lambda^{\alpha_m}|)}\cdot(\tilde{y}_{\lambda^{\alpha_1}}^{(v)}\otimes\cdots
\otimes\tilde{y}_{\lambda^{\alpha_m}}^{(v)})$ is just the decorated
link $\mathcal{L}\star \otimes_{\alpha=1}^L
Q_{\lambda^{\alpha}}^{(v)}$ since
$(\tilde{y}_{\lambda^{\alpha_j}}^{(v)})^2=\tilde{y}_{\lambda^{\alpha_j}}^{(v)}$.

By Proposition \ref{proposition-tracehomfly}, we have
\begin{align} \label{formula-trV}
tr_{V^{\otimes
n}}\varphi(b^{(|\lambda^{\alpha_1}|,...,|\lambda^{\alpha_m}|)}\cdot(\tilde{y}_{\lambda^{\alpha_1}}^{(v)}\otimes\cdots
\otimes\tilde{y}_{\lambda^{\alpha_m}}^{(v)}))=\mathcal{H}^{(v)}(\mathcal{L}\star
\otimes_{\alpha=1}^L
Q_{\lambda^{\alpha}}^{(v)})|_{x=q^{-\frac{1}{N}},a=q^{N}}.
\end{align}

By the construction of the idempotent $y_{\lambda}^{(v)}$  (see
formula (\ref{formula-idempotent-variant}) ) and the formula
(\ref{formula-homfly-variant}), we conclude that
\begin{align}
(xa)^{-w(b^{(|\lambda^{\alpha_1}|,...,|\lambda^{\alpha_m}|)})}\mathcal{H}^{(v)}(\mathcal{L}\star
\otimes_{\alpha=1}^L
Q_{\lambda^\alpha}^{(v)})=a^{-w(b^{(|\lambda^{\alpha_1}|,...,|\lambda^{\alpha_m}|)})}\mathcal{H}(\mathcal{L}\star
\otimes_{\alpha=1}^L Q_{\lambda^{\alpha}}).
\end{align}

Thus we have
\begin{align}
\mathcal{H}^{(v)}(\mathcal{L}\star \otimes_{\alpha=1}^L
Q_{\lambda}^{(v)})=x^{w(b^{(|\lambda^{\alpha_1}|,...,|\lambda^{\alpha_m}|)})}\mathcal{H}(\mathcal{L}\star
\otimes_{\alpha=1}^L Q_{\lambda^{\alpha}}).
\end{align}

Then the formula (\ref{formula-trV}) leads to
\begin{align} (\label{formula-trVn})
tr_{V^{\otimes
n}}\varphi(b^{(|\lambda^{\alpha_1}|,...,|\lambda^{\alpha_m}|)}\cdot(\tilde{y}_{\lambda^{\alpha_1}}^{(v)}\otimes\cdots
\otimes\tilde{y}_{\lambda^{\alpha_m}}^{(v)}))=q^{-\frac{w(b^{(|\lambda^{\alpha_1}|,...,|\lambda^{\alpha_m}|)})}{N}}\mathcal{H}(\mathcal{L}\star
\otimes_{\alpha=1}^L Q_{\lambda^{\alpha}})|_{a=q^{N}}.
\end{align}

On the other hand side, since $\tilde{y}_{\lambda}^{(v)}$ is an
idempotent in the algebra $\tilde{H}^{N}_{n}(q)$, it follows that
\begin{align}
\varphi(\tilde{y}_{\lambda}^{(v)})\in End_{U_q(sl_N)}(V^{\otimes
|\lambda|})
\end{align}
 is an idempotent in $End_{U_q(sl_N)}(V^{\otimes
|\lambda|})$. Moreover, we have
\begin{align}
\varphi(\tilde{y}_{\lambda}^{(v)})V^{\otimes |\lambda|}=V_{\lambda}.
\end{align}

Finally, if we color every components $\mathcal{K}_1,...,
\mathcal{K}_L$ of the framed link $\mathcal{L}$ by irreducible
$U_q(sl_{N}\mathbb{C})$-modules
$V_{\lambda^{1}},...,V_{\lambda^{L}}$ respectively, then we have the
following cabling formula
\begin{proposition}
Under the above setting,
\begin{align} \label{formula-trVvarphi}
tr_{V_{\lambda^{\alpha_1}}\otimes\cdots \otimes
V_{\lambda^{\alpha_m}}}(\varphi(b))=tr_{V^{\otimes
n}}\varphi(b^{(|\lambda^{\alpha_1}|,...,|\lambda^{\alpha_m}|)}\cdot(\tilde{y}_{\lambda^{\alpha_1}}^{(v)}\otimes\cdots
\otimes\tilde{y}_{\lambda^{\alpha_m}}^{(v)})).
\end{align}
\end{proposition}
\begin{proof}
We refer to Lemma 3.2 in \cite{LiZh10} for the basic idea of this
proof. Let $\chi_{n,n'}$ be the following $(n,n')$-crossing braid
\begin{align}
\chi_{n,n'}=\prod_{i=1}^{n'}(\sigma_{i+n-1}\sigma_{i+n-2}\cdots
\sigma_i)\in B_{n+n'}.
\end{align}
Applying the identities (\ref{formula-RUV}) inductively, we get
\begin{align}
\varphi(\chi_{n,n'})=\check{R}_{V^{\otimes n'},V^{\otimes n}}.
\end{align}

Let $\lambda$ and $\lambda'$ be two partitions with $|\lambda|=n$
and $|\lambda'|=n'$, then
$\varphi(\tilde{y}_{\lambda}^{(v)})V^{\otimes n}=V_{\lambda}$ and
$\varphi(\tilde{y}_{\lambda'}^{(v)})V^{\otimes n}=V_{\lambda'}$. We
obtain
\begin{align} \label{formula-varphi-ylambda}
\varphi(\tilde{y}_{\lambda}^{(v)}\otimes
\tilde{y}_{\lambda'}^{(v)})\cdot\varphi(\chi_{n,n'})=\varphi(\tilde{y}_{\lambda}^{(v)})\otimes\varphi(
\tilde{y}_{\lambda'}^{(v)})\cdot \check{R}_{V^{\otimes
n'},V^{\otimes n}}=\check{R}_{V_{\lambda'},V_{\lambda}}\cdot
\varphi(\tilde{y}_{\lambda'}^{(v)})\otimes\varphi(
\tilde{y}_{\lambda}^{(v)}).
\end{align}

Therefore, by applying formula (\ref{formula-varphi-ylambda})
iteratively, we have
\begin{align}
&tr_{V^{\otimes
n}}(\varphi(b^{(|\lambda^{\alpha_1}|,...,|\lambda^{\alpha_m}|)}\cdot\tilde{y}_{\lambda^{\alpha_1}}^{(v)}\otimes\cdots
\otimes\tilde{y}_{\lambda^{\alpha_m}}^{(v)}))\\\nonumber
&=tr_{V^{\otimes |\lambda^{\alpha_1}|}\otimes \cdots \otimes
V^{\otimes |\lambda^{\alpha_m}|} }(\varphi(b)\cdot
\varphi(\tilde{y}_{\lambda^{\alpha_1}}^{(v)})\otimes\cdots
\otimes\varphi(\tilde{y}_{\lambda^{\alpha_m}}^{(v)})\\\nonumber
&=tr_{V_{\lambda^{\alpha_1}}\otimes\cdots \otimes
V_{\lambda^{\alpha_m}}}(\varphi(b)).
\end{align}
\end{proof}

\begin{theorem}
The two definitions (\ref{formula-coloredhomfly}) and
(\ref{formula-coloredhomfly-skein}) for colored HOMFLY-PT invariant
$W_{\lambda^1,...,\lambda^L}(\mathcal{L};q,a)$ are equal.
\end{theorem}
\begin{proof}
According to definition (\ref{formula-coloredhomfly}),
\begin{align}
&W_{\lambda^1,...,\lambda^{L}}(\mathcal{L};q,a)\\\nonumber
&=q^{-\sum_{\alpha=1}^{L}w(\mathcal{K}_\alpha)
(\kappa_{\lambda^{\alpha}}+|\lambda^{\alpha}|N-\frac{|\lambda^{\alpha}|^2}{N})+\frac{2}{N}\sum_{\alpha<\beta}
lk(\mathcal{K}_\alpha,\mathcal{K}_{\beta})|\lambda^{\alpha}|\cdot|\lambda^{\beta}|}tr_{V_1\otimes
\cdots \otimes V_n}(\Phi(b))|_{q^N=a}.
\end{align}

By formulas (\ref{formula-trVn}) and (\ref{formula-trVvarphi}), and
the following identity
\begin{align}
\sum_{\alpha=1}^Lw(\mathcal{K}_\alpha)|\lambda^\alpha|^2+2\sum_{1\leq
\alpha<\beta\leq
L}lk(\mathcal{K}_\alpha,\mathcal{K}_\beta)|\lambda^\alpha|\cdot|\lambda^{\beta}|=w(\beta^{(n_{i_1},...,n_{i_m})}),
\end{align}
we obtain
\begin{align}
W_{\lambda^1,...,\lambda^{L}}(\mathcal{L};q,a)=q^{-\sum_{\alpha=1}^{L}w(\mathcal{K}_\alpha)
}a^{-\sum_{\alpha=1}^{L}|\lambda^{\alpha}|w(\mathcal{K}_\alpha)}\mathcal{H}(\mathcal{L}\star
\otimes_{\alpha=1}^L Q_{\lambda^{\alpha}};q,a).
\end{align}
\end{proof}

\begin{remark}
Comparing to the character formula (\ref{formula-Wlambda}) for the
colored HOMFLY-PT invariants obtained in \cite{LiZh10}, we have
\begin{align}
\mathcal{H}(\mathcal{L}\star \otimes_{\alpha=1}^L
Q_{\lambda^{\alpha}};q,a)=\sum_{|\lambda|=n}\zeta^{\lambda}(\varphi(b^{(n_{\alpha_1},...,n_{\alpha_m})}
y_{\lambda^{\alpha_1}}\otimes \cdots \otimes
y_{\lambda^{\alpha_m}}))s_{\lambda}^*(q,a).
\end{align}
This formula would be useful when we would like to apply the
character theory of Hecke algebra to study the HOMFLY-PT skein
theory.
\end{remark}

\end{document}